\documentclass[12pt]{article}
\usepackage{setspace}

\usepackage{algorithm,algpseudocode}

\makeatletter
\def\BState{\State\hskip-\ALG@thistlm}
\makeatother

\usepackage[utf8]{inputenc}
\usepackage{amsmath,amsfonts}
\usepackage{color,enumerate}
\usepackage{graphics,afterpage,subfigure,wrapfig,lipsum,caption2}
\usepackage[demo]{graphicx}
\usepackage[utf8]{inputenc}
\usepackage{mathtools}   
\usepackage{threeparttable}
\usepackage{pifont}
\usepackage[round]{natbib} 
\usepackage[OT1]{fontenc}
\usepackage{amscd}
\usepackage{latexsym,amssymb,amsthm}
\usepackage{pdflscape}

  \usepackage{hyperref}
  \hypersetup{colorlinks=true,citecolor=blue, backref,linktocpage=true,linkcolor=blue}

\newcommand{\yn}{{\bf y}_n}
\newcommand{\bfx}{{\bf x}}

\newcommand{\pg}{p_{\bgm}}
\newcommand{\bgm}{\boldsymbol{\gamma}}
\newcommand{\bbta}{\boldsymbol{\beta}}
\newcommand{\Rg}{R_{\bgm}}
\newcommand{\ad}{\alpha_\delta} 
\newcommand{\bta}{\boldsymbol{\theta}}

\newtheorem{lemma}{Lemma}
\newtheorem{remark}{Remark}
\newtheorem{theorem}{Theorem}
\newtheorem{result}{Result}

\title{\bf Targeted Random Projection for Prediction from High-Dimensional Features}
\author{Minerva Mukhopadhyay \and David B. Dunson}
\date{}

\begin{document}

\maketitle
\begin{abstract}
We consider the problem of computationally-efficient prediction with high dimensional and highly correlated predictors when accurate variable selection is effectively impossible. Direct application of penalization or Bayesian methods implemented with Markov chain Monte Carlo can be computationally daunting and unstable. 
A common solution is first stage dimension reduction through screening or projecting the design matrix to a lower dimensional hyper-plane.
Screening is highly sensitive to threshold choice, while projections often have poor performance in very high-dimensions. We propose TArgeted Random Projection (TARP) to combine positive aspects of both strategies. TARP uses screening to order the inclusion probabilities of the features in the projection matrix used for dimension reduction, leading to data-informed sparsity. We provide theoretical support for a Bayesian predictive algorithm based on TARP, including statistical and computational complexity guarantees. Examples for simulated and real data applications illustrate gains relative to a variety of competitors. 
\end{abstract}
\par
\noindent
{\bf Some key words:} Bayesian; Dimension reduction; Large $p$ small $n$;  Screening.
\par
\noindent
 {\bf Short title}: Targeted Random Projection 

\section{Introduction}\label{sec:1}

We are interested in the problem of prediction of a response variable $y$ given enormous dimensional predictors ${\bf x} = (x_1,\ldots,x_p)^{\prime}$ in the case in which $p \gg n$, and the predictors have a potentially high level of dependence.  There is an immense literature on methods for accommodating large $p$ modest $n$ regression, ranging from frequentist penalization methods to Bayesian approaches using variable selection or shrinkage priors.  

Such penalization methods typically define an objective function for estimating regression coefficients $\bbta = (\beta_1,\ldots,\beta_p)^{\prime}$ in a linear or generalized linear model. Some notable developments in this literature include (1) the least absolute shrinkage and selection operator (LASSO) using an L1 penalty and defining a convex optimization algorithm \citep{lasso,tibshirani1997}, (2) the smoothly clipped absolute deviation (SCAD) penalty \citep{SCAD} which improves on Lasso and hard thresholding and leads to a non-convex optimization problem, and (3) the minimax concave penalty (MCP) forming a key component of the MC+ procedure for nearly unbiased high-dimensional variable selection \citep{MCP}. SCAD and MCP are special cases of folded concave penalization, which can have multiple local solutions.  \cite{FXZ2014} propose a one step local linear approximation to deal with this problem and produce an oracle estimator.  \cite{LWM2015} provide theory showing that all stationary points of a broad class of regularized M-estimators (including SCAD and MCP) are close together under some conditions.  A key issue in these and other penalization algorithms is the choice of the penalty parameter, as carefully studied in \cite{WKL2013} and \cite{ZLT2010}.

There is a parallel Bayesian literature defining a broad class of priors for the high-dimensional $\bbta$ vector, along with corresponding algorithms for posterior computation.  The two main classes of priors are spike and slab priors, which have a two component mixture form with one component concentrated at zero (see \cite{MB1988}, \cite{SplikeSlab}), and continuous shrinkage priors, which define the prior probability density to be concentrated at zero with heavy tails.  The concentration at zero shrinks small coefficients towards zero and the heavy tails avoids over shrinkage of large coefficients.  A broad class of such shrinkage priors can be expressed as local-global scale mixtures of Gaussians, with some popular examples including the horseshoe (\cite{CPS2009}), generalized double Pareto (\cite{ADL2013}) and Dirichlet-Laplace (\cite{BPPD2015}).  These Bayesian approaches are typically implemented with Markov chain Monte Carlo (MCMC) algorithms, which can face issues in scaling up to very large $p$ cases, though \cite{YWJ2016} provide some positive theory results suggesting good scalability in some cases.

For time-efficient implementation of most of these methods, some preliminary dimensionality reduction tool is usually employed. Two of the most common strategies are variable screening, focusing on the predictors $x_j$ having the highest marginal associations with $y$, and projection, compressing the $n\times p$ design matrix $X$ by post-multiplying an appropriate matrix; for example, principal component (PC) projection or random projection (RP). Screening is particularly popular due to the combination of conceptual simplicity, computational scalability, and theoretical support. For example,  
sure independence screening (SIS, \cite{ISIS}) can be guaranteed {\em asymptotically} to select a superset of the `true' predictors \citep{fan2009}.

However, when the predictors are highly correlated and/or the true data generating process does not exhibit strong sparsity with a high signal-to-noise ratio, it may be necessary to use a very conservative threshold for the measure of marginal association, limiting the dimensionality reduction occurring in the first stage. 
Projection based methods instead compress the short, fat matrix $X$ into a thinner matrix having fewer columns while maintaining much of the information in $X$ relevant to prediction of $y$. Properties of PC projection are well studied in the literature (see, for example, \cite{LZW2010}). There is a rich literature providing accuracy bounds of RPs in terms of predictive errors. Refer, for example, to
\cite{maillard2009}, \cite{fard2012}, \cite{kaban2014}, \cite{thanei2017}, \cite{GD_2015}, \cite{koop2016}. \cite{GD_2015} concentrate on approximating predictive distributions in Bayesian regression with compressed predictors. The above literature on RPs focuses primarily on random matrices with i.i.d elements.

When predictors are very high-dimensional, existing projection based methods can fail as they tend to include many unimportant predictors in each linear combination, diluting the signal. Potentially, one can attempt to improve performance by estimating the projection matrix, but this results in a daunting computational and statistical problem. Alternatively, we propose a TArgeted Random Projection (TARP) approach, which includes predictors in the projection matrix with probability proportional to their marginal utilities. These utilities are estimated quickly as a preprocessing step using an independent screening-type approach. To reduce sensitivity of the results to tuning parameters and Monte Carlo error, we take the approach of aggregating over multiple realizations. The proposed randomized screening method can be viewed as a type of rapid preconditioning, enabling improved predictive performance in high-dimensional settings. Compared with applying PC projections or RPs, after screening out predictors, TARP has the advantage of removing sensitivity to threshold choice by using a soft probabilistic approach.

\section{The Proposed Method} \label{sec:2}
Let $\mathcal{D}^n=\left\{(\yn; X_n): \yn \in \mathbb{R}^n, X_n \in \mathbb{R}^{n\times p_n} \right\}$ denote the dataset consisting of $n$ observations on $p_n$ predictors $x_1, x_2, \ldots, x_{p_n}$ and a response $y$, and $(y_i; {\bf x}_i )$ denote the $i^{th}$ data point, $i=1,2, \ldots, n$. Suppose that the data can be characterized by a generalized linear model (GLM). The density of $y$ is related to the predictors as 
\begin{equation}
f(y_i| \bbta, \sigma^2)=\exp\left[\frac{1}{d(\sigma^2)} \left\{ y_i a({\bf x}_i^{\prime} \bbta) + b({\bf x}_i^{\prime} \bbta) +c(y_i) \right\} \right], \label{eq_rp1}
\end{equation}
where $a(\cdot)$ and $b(\cdot)$ are continuously differentiable functions, $a(\cdot)$ has non-zero derivative, $d(\cdot)$ is a non-zero function, $\bbta\in \mathbb{R}^{p_n}$, and $\sigma^2\in \mathbb{R}^{+}$. We approximate the density of $y$ in a compressed regression framework as follows: 
\[ f(y_i | \bta, R_n, \sigma^2 )=\exp \left[ \frac{1}{d(\sigma^2)} \left\{ y_i a \left((R_n {\bf x}_i)^{\prime} \bta \right) + b\left((R_n {\bf x}_i)^{\prime} \bta\right) +c(y_i) \right\}\right].\]
Here $R_n \in \mathbb{R}^{m_n \times p_n}$ is a projection matrix, $\bta\in \mathbb{R}^{m_n}$ is vector of compressed regression coefficients, and $m_n \ll p_n$. We discuss the choice of the projection matrix $R_n$ in Section \ref{sec:2.1}, and illustrate the method in detail in Section \ref{sec:2.2}. 

{\it Priors.}  We assume that the covariates are standardized. Taking a Bayesian approach, we assign priors to $\bta$ and $\sigma^2$:
$N({\bf 0}, \sigma^2 I)$ and 
\emph{Inv-Gamma ($a_{\sigma}, b_{\sigma}$)}, with $a_{\sigma}, b_{\sigma}>0$. This Normal-Inverse Gamma (NIG) prior is a common choice for GLMs. 
For Gaussian likelihoods, this prior is conjugate, and the posterior and predictive distributions are analytic.

\subsection{Choice of projection matrix}\label{sec:2.1} 
The projection matrix $R_n$ embeds $X_n$ to a lower dimensional subspace.  If $p_n$ is not large, the best linear embedding can be estimated using a singular value decomposition (SVD) of $X_n$.  
However, if $p_n \gg n$, then it is problematic to estimate the projection, both computationally and statistically, and random projection (RP) provides a practical alternative. If an appropriate RP matrix is chosen, due to Johnson-Lindenstrauss (JL) type embedding results, distances between sample points are preserved (see \cite{dasgupta2003}, \cite{Achlioptas_2003}). 

Our focus is on constructing projection matrices so that predictors $x_j$ having relatively weak marginal relationships with $y$ are less likely to be included in downstream analyses. In particular, TArgeted Random Projection (TARP) matrices are constructed as follows: 
\begin{eqnarray}
\boldsymbol{ \gamma}&=&\left( \gamma_1, \gamma_2, \ldots, \gamma_{p_n} \right)^{\prime}\quad \mbox{and}\quad \gamma_j
\stackrel{i.i.d.}{\sim} ~\mbox{Bernoulli} \left(q_j \right)  \mbox{~~~where }\notag\\
&&\hskip17pt q_j \propto |r_{x_j,y}|^\delta~~~~ \mbox{for some constant~~ $\delta>0$, ~}j=1,2,\ldots,p_n,  \label{eq_rp2} \\
R_{\overline{\bgm}} &=& O_{m_n\times (p_n - p_{\bgm})} \quad \mbox{and }~~~\Rg = R_n^{*}, \notag
 \end{eqnarray}
where $r_{x_j,y}$ is a marginal dependence measure of $x_j$ and $y$, 
$\Rg$ and $R_{\overline{\bgm}}$ are sub-matrices of $R_n$ with columns corresponding to non-zero and zero values of $\bgm$, $O_{m_n\times (p_n - p_{\bgm})}$ is the $m_n\times (p_n - p_{\bgm})$ matrix of zeros with $p_{\bgm}=\sum_j \gamma_j$, and $R_n^{*}$ is a $m_n\times p_{\bgm}$ projection matrix.

We obtain a Randomized Independence Screening (RIS) approach by choosing a random subset of predictors using a function of marginal utilities, ${\bf q}=(q_1,q_2,\ldots,q_{p_n})^{\prime}$, as inclusion probabilities. 
The selected subset is then projected to a lower dimensional subspace using $R_n^{*}$. RIS can be considered as a practical alternative to SIS for prediction problems with ultrahigh-dimensional nearly collinear predictors, and can be applied as a preprocessing step to any prediction algorithm. In Supplementary Materials (SM, Section \ref{sm:1}) we discuss applicability of RIS in detail. 

There are many possible choices of $R_n^{*}$ which can reduce dimension without sacrificing prediction accuracy.  Two predominant classes that we focus on are based on partial SVD and random projections facilitating JL type embedding. 

\vskip5pt
\noindent{\bf Random projection.}
For some constant $\psi \in (0,0.5)$, each element $R_{k,j}^{*}$ of $R_n^{*}$  is sampled independently from a three point distribution as 
\begin{equation}
   R_{k,j}^{*} = \pm \frac{1}{\sqrt{2\psi}} \quad \mbox{with probability}~~ \psi; \quad\mbox{and} \quad   R_{k,j}^{*} =0 \quad \mbox{with probability}~~ 1-2\psi. \label{eq_rp18}
\end{equation}
Projection matrices of this form are widely used due to their inter point distance preservation property. 
We refer to the method that
generates $R_n^*$ in  (\ref{eq_rp2}) from (\ref{eq_rp18}) as RIS-RP.
\begin{remark}\label{rm3}
The choice of projection matrix in (\ref{eq_rp18}) can be replaced by a wide variety of matrices having i.i.d. components with mean zero and finite fourth moments. One of the sparsest choices is of the form $R_{k,j}^{*}=\pm n^{\kappa/2}/\sqrt{m_n}$ with probability $1/2n^{\kappa}$, $0$ with probability $(1-1/n^{\kappa})$,
where $m_n\sim n^\kappa$, $\kappa\in(0,1)$ (see \cite{LHC_2006}). 
This choice is useful in compressing extremely large dimensional data. Our theoretical results extend to this case.
\end{remark}

\vskip5pt
\noindent{\bf Principal component projection.}
Let $X_{\bgm}$ be the sub-matrix of $X_n$ with columns corresponding to non-zero values of $\bgm$. Consider the partial spectral decomposition of $X_{\bgm}$ as $X_{\bgm}^{\prime} X_{\bgm} = V^{\prime}_{\bgm, m_n} D_{\bgm, m_n} V_{\bgm, m_n}$, and let
 \vspace{-.12 in}
\begin{eqnarray}
R_n^{*} = V_{\bgm,m_n}^{\prime}. \label{eq_rp16}
\end{eqnarray}
We refer to this method as RIS-PCR (RIS-Principal Component Regression).

\vskip10pt
The performance of TARP depends on tuning parameters $m_n$, $\delta$ and $\psi$. 
To limit dependence of the results on tuning parameters and random variation in the projection matrix, we generate multiple realizations of the matrix for different tuning parameters, and aggregate these results. Potentially, one could estimate weights for aggregation using Bayesian methods (see \cite{hoeting1999}) or other ensemble learning approaches, but we focus on simple averaging due to its computational and conceptual simplicity. Refer to SM Section \ref{sm:3.2} for a comparison of simple aggregation with other learning approaches.

\begin{remark}\label{rm:review1}
The marginal utility measure $r$ in (\ref{eq_rp2}) can be the correlation coefficient for normal linear models. For GLMs, there are many suitable choices for $r$ (see, for example, \cite{ZA_2000}, \cite{fan2009}, \cite{fan2010}, \cite{KS2017}). These include the maximum marginal likelihood estimator (MMLE), maximized marginal likelihood (MML), regression correlation coefficient (RCC), etc. TARP is not sensitive to the choice of $r$ as it considers a wide list of predictors randomly, and relies on aggregation. We discuss this in SM Section \ref{sm:2p1}.
\end{remark}

\subsection{Posteriors and Predictive Distribution}\label{sec:2.2}
We illustrate TARP in normal linear regression. We replace the normal linear model $ y_i ={\bf x}_i^{\prime}\bbta+e_i$ by $ y_i =\left(R_n {\bf x}_i\right)^{\prime}\bta+e_i$, where $e_i\sim N(0, \sigma^2)$. Given the NIG prior stated above, the marginal posterior of $\bta$, given $R_n$,
follows a scaled $m_n$-variate $t$ distribution with degrees of freedom $n+2 a_\sigma$, location $\boldsymbol{\mu}_t$, and scale matrix $\Sigma_t$, where $\boldsymbol{\mu}_t=W_n Z_n^{\prime} \yn$, $\Sigma_t=(\yn^{\prime}\yn - \boldsymbol{\mu}_t^{\prime} W_n^{-1} \boldsymbol{\mu}_t+2 b_{\sigma})W_n/(n+2a_\sigma)$, with $Z_n= X_n R_n^{\prime}$ and $W_n= (I+Z_n^{\prime} Z_n )^{-1}$. The marginal posterior of $\sigma^2$, given $R_n$,
is inverse gamma with parameters $a_\sigma+ n/2$ and $(\yn^{\prime}\yn - \boldsymbol{\mu}_t^{\prime} W_n^{-1} \boldsymbol{\mu}_t)/2+b_{\sigma}$.

Consider prediction of $y$ for $n_{new}$ new data points, $X_{new}$, given the dataset $\mathcal{D}^n$. The predicted values ${\bf y}_{new}$ can be obtained using the Bayes estimator of $\bta$ under squared error loss as~
$ \hat{\bf y}_{new} = X_{new} R_n^{\prime} \hat{\bta}_{Bayes} ~~\mbox{where}~~ \hat{\bta}_{Bayes} =\boldsymbol{\mu}_t.$
Moreover, the posterior predictive distribution of ${\bf y}_{new}$ is
a $n_{new}$-variate $t$ distribution with degrees of freedom $n+2 a_{\sigma}$, location vector $\hat{{\bf y}}_{new}$ and scale parameter $(\yn^{\prime}\yn - \boldsymbol{\mu}_t^{\prime} W_n^{-1} \boldsymbol{\mu}_t+2 b_{\sigma})(I+X_{new}W_nX_{new}^{\prime})/(n+2a_\sigma)$.

For non-Gaussian likelihoods, predictive distributions can be approximated using 
Laplace's method \citep{tierney1986} or MCMC, as we discuss in SM Section \ref{sm:2p1}.

\subsection{Choices of Tuning Parameters }\label{sec:3.1}
\noindent {\bf Choice of $m_n$.} The parameter $m_n$ determines the number of linear combinations of predictors we consider. We choose $m_n$ over the range $(2\log p_n, \min\{3n/4,p_n\})$, consistent with our theoretical results in Section \ref{sec:3} and with numerical experiments.

\noindent {\bf Choice of $\delta$ and ${\bf q}$.} 
Higher values of $\delta$ lead to fewer variables selected in RIS. 
If $p_n \gg n$, one should select a relatively small proportion of predictors, compared with $p_n\sim n$. We recommend  $\delta=\max\{0,(1+\log(p_n/n))/2\}$ as a default, selecting all the predictors if $p_n\ll n$, and becoming more restrictive as $p_n$ increases compared to $n$. The selection probabilities in the RIS step are then $q_j=|r_{x_j,y}|^{\delta}/\max_j |r_{x_j,y}|^{\delta}$, $j=1,\ldots,p_n$. Hence, the variable with highest marginal utility is definitely included.

\noindent{\bf Choice of $\psi$.} The value of $\psi$ controls sparsity in the random matrix in RIS-RP,         
$\psi \in (0,0.5).$  \cite{Achlioptas_2003} suggests choosing $\psi=1/6$ as a default value. We consider multiple realizations of $\psi$ from the range $[0.1,0.4]$ avoiding very sparse and dense cases.
\subsection{Computational Algorithm and Complexity}\label{sec:3.2}
\noindent{\bf RIS-RP.} Assuming a normal linear model, with the correlation coefficient as $r_{x_j,y}$, for a specific choice of $(m_n, \delta, \psi)$, calculation of $\hat{\bf y}_{new}$ using RIS-RP involves the following steps:
  \begin{algorithmic}[1]
      \State Calculate $r_{x_j,y}$ for $j=1,\ldots,p_n$.
        \State Generate $\gamma_j\sim \mbox{Bernoulli}(q_j)$ where $q_j=|r_{x_j,y}|^{\delta}/\max\{|r_{x_j,y}|^{\delta}\},$ $j=1,\ldots,p_n$. \linebreak
         IF $\gamma_j=1$, generate $R_n$ with $R_{i,j}$ as in (\ref{eq_rp18}), $i=1,\ldots,n$. ~ 
ELSE set $R_{i,j}=0$.
\State Post-multiply $R_n$ with $X_n$. Set $Z_n=X_nR_n^{\prime}$.
\State Compute $\hat{\bta}=\left(Z_n^{\prime}Z_n+I\right)^{-1}Z_n^{\prime} \yn$.
      \State For a given $X_{new}$, compute $Z_{new}=X_{new}R_n^{\prime}$ and $\hat{\bf y}_{new}=Z_{new} \hat{\bta}$.
  \end{algorithmic}
 The complexity of steps 1, 2-3, 4 and 5 are $O(p_n)$, $O(n\pg m_n)$, $O(n m_n^2)$ and $O\left(n_{new} \pg m_n \right)$, respectively, where $\pg =\sum \gamma_j$. Thus, if $n_{new}\leq n$, the total complexity for a single choice of $(m_n, \delta, \psi)$ is $O(p_n)+2O(n m_n\pg )+O( n m_n^2) $ without any parallelization.

\vskip5pt
\noindent{\bf RIS-PCR.} 
RIS-PCR differs from RIS-RP in step 2 of the algorithm. After generation of $\bgm$, RIS-PCR applies SVD of $X_{\bgm}$ involving complexity $O\left(n\pg \min\{n,\pg\}\right)$.
Therefore, the two methods have comparable time complexity unless either $n$ or $\pg$ is much larger than $m_n$. Although theoretically we do not impose any restriction on $\pg$, in practice when $p_n=\exp\{o(n)\}$ and $\delta\geq 2$, $\pg$ is usually of order $n$. 

\vskip5pt
\noindent{\bf Increment of complexity due to aggregation.}
Suppose $N$ different choices of $(m_n, \psi, R_n)$ are considered. Each choice yields a model $\mathcal{M}_l : y\sim f\left(y | {\bf x}, m_{n,l}, \psi_l, R_{n,l} \right)$ along with a corresponding estimate of ${\bf y}_{new}$ (say $\hat{\bf y}_{new,l}$), where $l\in\{1,2, \ldots, N\}$. The proposed estimate is the simple average of these $N$ estimates of ${\bf y}_{new}$.

Step 1 in TARP is not repeated over the replicates, while other steps are repeated $N$ times.  Moreover, screening and aggregration steps are embarrassingly  parallelizable. Given $k$ CPUs, if $n_{new}\leq n$, the total complexity is $O(p_n/k)+2O(Nn m_n\pg/k )+ O(N n m_n^2/k) $ for RIS-RP, and $O(p_n/k)+O( N n \pg \min\{n, \pg\}/k) +O(Nn m_n\pg/k )+ O( N n m_n^2/k )$ for RIS-PCR. 

\vskip5pt
\noindent{\bf Increment of complexity for non-Gaussian GLMs.}
 For non-Gaussian GLMs, appropriate marginal utility measures include MMLE, MML or RCC. Analytic expressions of these estimates are unavailable. Approximation algorithms, like Fisher's scoring and coordinate descent, increase the complexity of Step 1 to an approximate order $O(np_n)$. However, in practice these measures do not improve performance over a simple correlation coefficient even in binary response cases.
 In Step 4, we calculate the posterior mean of the compressed regression coefficients, where we rely on either Laplace type approximations or MCMC for non-Gaussian likelihoods. However, as we are now dealing with $m_n$-dimensional regressors, the computational cost of MCMC is quite modest. Refer to SM Section \ref{sm:2p1}.


\section{Theory on Predictive Accuracy} \label{sec:3}

We study asymptotic performance of the predictive distribution of TARP for a single random projection matrix without considering aggregation.  We focus on \emph{weakly sparse} or \emph{dense} cases where the absolute sum of the true regression coefficients is bounded. This condition includes \emph{strong sparsity} where only a few covariates have non-zero coefficients. 

The projection matrix in TARP depends on the random variable $\bgm$, and therefore is denoted by $\Rg$. 
We denote a particular realization of the response variable as $y$, and a particular realization of the predictors $(x_1, x_2, \ldots, x_{p_n})^{\prime}$ as ${\bf x}$.  
Let $f_0$ be the true density of $y$ given the predictors, and $f(y|{\bf x}, \bgm ,\Rg, \bta)$ be the conditional density of $y$ given the model induced by $\left(\bgm,\Rg\right)$, $\bta$ drawn from the posterior distribution,  and ${\bf x}$. We follow \cite{Jiang_2007} in showing that the predictive density under our procedure is close to the true predictive density in an appropriate sense.

\subsection{The Case with Bounded Predictors}
In this subsection we assume that each covariate $x_j$ is standardized with $|x_j|<M$, for $j=1,2,\ldots,p_n$, and $M$ is a constant. We also assume that the scale parameter $\sigma^2$ in (\ref{eq_rp1}) is known.  We require the following two assumptions on $f_0$ and the design matrix.

\vskip5pt
\noindent{\bf Assumption (A1)}  Let $r_{x_j,y}$ denote the value of the chosen marginal utility measure for the observed values of $x_j$. Then for each data point $(y,{\bf x})$ and constant $\delta$ in (\ref{eq_rp2}), there exists a positive constant $\alpha_{\delta}$ such that

{\centering  $\displaystyle\lim_{p_n\rightarrow\infty} \frac{1}{p_n} \sum_{j=1}^{p_n} \displaystyle x_{j}^2 |r_{x_j,y}|^{\delta} \rightarrow {\alpha}_{\delta}.$ \par}

\vskip5pt
\noindent {\bf Assumption (A2)} 
Let $q(\bgm)= \prod_{i=1}^{n} q_j^{\gamma_j} (1-q_j)^{(1-\gamma_j)}$, with $q_j$ defined in (\ref{eq_rp2}), denote the probability of obtaining a particular $\bgm=(\gamma_1,\ldots,\gamma_{p_n})'$ in the RIS step.  Let $\Gamma_l \subset \{0,1\}^{p_n}$ denote the set of $\bgm$ vectors such that $p_{\bgm} = l$, and let $\mathcal{M}_l \subset \Gamma_l$ denote the first $p_n^{k_n}$ elements of $\Gamma_l$ ordered in their $q(\bgm)$ values. Let $\mathcal{A}_n$ be the event that $\bgm \in \mathcal{M}_l$ for some $l=1,\ldots,p_n$.  Then, 
$P(\mathcal{A}_n^c)=P\big( \{\bgm: \bgm \notin \cup_l \mathcal{M}_l\} \big)
\le \exp(-n\varepsilon_n^2/4),$ almost surely under $f_0$, for some increasing sequence of integers $\{k_n\}$ and sequence $\{\varepsilon_n\}$ satisfying $0<\varepsilon_n^2<1$ and $n\varepsilon_n^2\rightarrow\infty$.

\begin{remark}\label{rm2}
	As the probability of selection in the random screening step depends on the marginal utility estimates of the predictors and the response, assumption (A2) is on the data generating process. 
	Violations of (A2) would imply that large numbers of predictors have marginal utility estimates that are not close to zero. We discuss (A2) in more detail in Section \ref{sm:2.1.0} of SM. 
\end{remark}

{\it Measures of closeness:} Let $\nu_{\bf x}(d{\bf x})$ be the probability measure for ${\bf x}$, and $\nu_y(dy)$ be the dominating measure for conditional densities $f$ and $f_0$. The dominating measure of $(y,{\bf x})$ is taken to be the product of $\nu_y(dy) \nu_{\bf x}(d{\bf x})$.
The Hellinger distance between $f$ and $f_0$ is

\vskip5pt
{\centering  $d(f,f_0)=\sqrt{\displaystyle\int \left(\sqrt{f}-\sqrt{f_0} \right)^2 \nu_{\bf x}(d{\bf x}) \nu_y(dy)}.$ \par}

\vskip5pt
\noindent The Kullback-Leibler divergence between $f$ and $f_0$ is

\vskip5pt
{\centering $ d_0(f,f_0)=\displaystyle\int f_0 \ln \left(\frac{f_0}{f} \right) \nu_{\bf x}(d{\bf x})\nu_y(dy).$ \par}

\vskip5pt
\noindent Define  $\quad d_t(f,f_0)=t^{-1}\Big( \displaystyle\int f_0 \displaystyle\left(\frac{f_0}{f} \right)^t\nu_{\bf x}(d{\bf x})\nu_y(dy) -1\Big),$ for any $t>0.$

\vskip5pt
\noindent Consider the following two facts: (i) $d(f,f_0) \leq \left( d_0(f,f_0) \right)^{1/2}$, and (ii) $d_t(f,f_0)$ decreases to $d_0(f,f_0)$ as $t$ decreases to $0$  (see \cite{Jiang_2007}).

Let $\mathcal{P}_n$ be a sequence of sets of probability densities, and $\varepsilon_n$ be a sequence of positive numbers.
Let $ N(\varepsilon_n,\mathcal{P}_n)$ be the $\varepsilon_n$-covering number, i.e., the minimal number of Hellinger balls of radius $\varepsilon_n$ needed to cover $\mathcal{P}_n$. 

\vskip10pt
\noindent{\bf RIS-RP.} The result showing predictive accuracy of RIS-RP is stated below.

\begin{theorem}\label{thm:1}
	Let $\bta\sim N({\bf 0},\sigma_{\theta}^2 I)$, and $f(y|{\bf x}, \bgm ,\Rg, \bta)$ be the conditional density of $y$ given the model induced by $\left(\bgm,\Rg\right)$, where $\Rg$ is as in (\ref{eq_rp2}) and (\ref{eq_rp18}). Let $\bbta_{0}$ be the true regression parameter with $\sum_j |\beta_{0,j} |<K$ for some constant $K$, and assumptions (A1)-(A2) hold. Consider the sequence $\{ \varepsilon_n\}$ as in assumption (A2) satisfying $0< \varepsilon_n^2<1$ and $n\varepsilon_n^2 \rightarrow \infty$, and assume that the following statements hold for sufficiently large $n$:\\
	(i) $ m_n |\log \varepsilon_n^2| < n \varepsilon_n^2/4$, \\
	(ii) $ k_n \log p_n < n \varepsilon_n^2/4$, and \\
	(iii) $m_n \log \left(1 + D\left(\sigma_{\theta} \sqrt{6 n \varepsilon_n^2 p_n m_n^2} \right) \right) < n \varepsilon_n^2/4$, where
	
	\qquad $D(h^{*})=h^{*} \sup_{h\leq h^{*}} |a^{\prime} (h)| \sup_{h\leq h^{*}} |a^{\prime} (h)/b^{\prime}(h)|$, $b(\cdot)$ as in (\ref{eq_rp1}). Then, 
	
	\vskip5pt
	{\centering $P_{f_0}\left[\pi\left\{d(f,f_0)> 4\varepsilon_n| \mathcal{D}^n \right\} > 2 e^{-n\varepsilon_n^2/4 } \right] \leq 2 \displaystyle e^{-n\varepsilon_n^2/5 },  $ \par}
	\vskip5pt
	
	\noindent as ~$\min\{m_n,p_n,n\}\rightarrow\infty$, where $\pi\{ \cdot | \mathcal{D}^n\}$ is the posterior measure.
\end{theorem}

\vskip5pt
\noindent{\bf RIS-PCR.} Asymptotic guarantees on predictive accuracy of RIS-PCR requires an additional assumption.

\vskip5pt
\noindent{\bf Assumption (A3)} ~ Let $X_{\bgm}$ be the sub-matrix of $X_n$
with columns corresponding to non-zero values of $\bgm$, and ${\bf x}_{\bgm}$ be a row of $X_{\bgm}$. Let $V_{\bgm}$ be the $m_n \times \pg$ matrix of $m_n$ eigenvectors corresponding to the $m_n$-largest eigenvalues of $X_{\bgm}^{\prime} X_{\bgm}$.
Then, for each $\bgm$ and data point ${\bf x}_{\bgm}$,

\vskip5pt
{\centering $ \left. \| V_{\bgm} {\bf x}_{\bgm} \|^2 \right/\| {\bf x}_{\bgm} \|^2 \geq \alpha_n,$ \par}
\vskip5pt

\noindent where $\alpha_n \sim (n\varepsilon_n^2)^{-1}$, where the sequence $\{ \varepsilon_n^2\}$ is as in assumption (A2).

\begin{theorem}\label{thm:2}
	Let $\bta\sim N({\bf 0},\sigma_{\theta}^2 I)$, and $f(y|{\bf x}, \bgm ,\Rg, \bta)$ be the conditional density of $y$ given the model induced by $\left(\bgm,\Rg\right)$, where $\Rg$ is as in (\ref{eq_rp2}) and (\ref{eq_rp16}). Let $\bbta_{0}$ be the true regression parameter with $\sum_j |\beta_{0,j} |<K$ for some constant $K$, and assumptions (A1)-(A3) hold. Assume that the conditions (i)-(ii) of Theorem \ref{thm:1} hold, and 
	
	\noindent (iii) 
	$m_n \log \left(1 + D\left(\sigma_{\theta} \sqrt{6 n \varepsilon_n^2 p_n m_n} \right) \right) < n \varepsilon_n^2/4$, where $D(h^{*})$ is as defined in Theorem \ref{thm:1}.
	
	\vskip5pt
\noindent Then, 
	\hspace{1 in} $P_{f_0}\left[\pi\left\{d(f,f_0)> 4\varepsilon_n| \mathcal{D}^n \right\} > 2 e^{-n\varepsilon_n^2/4 } \right] \leq 2 e^{-n\varepsilon_n^2/5 },  $ 
	\vskip5pt
	
	\noindent as ~$\min\{m_n,p_n,n\}\rightarrow\infty$, where $\pi\{ \cdot | \mathcal{D}^n\}$ is the posterior measure.
\end{theorem}

%
\noindent The implications of Assumption (A3), and the conditions (i)-(iii) in Theorems \ref{thm:1} and \ref{thm:2}, are discussed in Remarks \ref{rm5} and \ref{rm6} in Section \ref{sm:2.1} of SM. Section \ref{sm:2.1.1} contains a discussion on interpretations of the theorems. The proof of Theorem \ref{thm:1} is in the Appendix, and that of Theorem \ref{thm:2} is in SM Section \ref{sm:2.2}. 

\subsection{The Case with Gaussian Predictors}\label{sec:unbdd}
In this subsection we modify the theory on asymptotic properties for Gaussian predictors. Let ${\bf x} \sim N\left( {\bf 0}, \Sigma_x\right)$, with $\Sigma_x=((\sigma_{i,j}))$ be positive definite and $\sigma_{i,i}=1$ for all $i=1,\ldots,p_n$. We further assume strong $\rho$-mixing conditions on the covariance structure as follows: 

\vskip5pt
\noindent{\bf Assumption (B1)} The off-diagonal elements of $\Sigma_x$ satisfy $\sigma_{i,j}=\rho_{|i-j|} $ with  $\sum_{j=1}^{\infty} \rho_{j}^2<\infty$.

\begin{remark}\label{rm:e1}
Assumption (B1) holds in many situations, for example, under auto regressive structure, block diagonal structure with finite block sizes, or the cases where $\left| \sigma_{i,j} \right| \leq |i-j|^{-r_0}$ for some $r_0>1/2$. Further, it can be shown using simple algebra that (B1) implies $\sum_{i\neq j} \sigma_{i,j}^2=O\left(p_n\right)$.
\end{remark}

We further assume that

\vskip5pt
\noindent{\bf Assumption (B2)} The largest eigenvalue of $\Sigma_x$ is less than $c_0 l_n$ for some increasing sequence of numbers $l_n$ such that ~$\max\left\{l_n\log (l_n),l_n m_n \sqrt{n \varepsilon_n^2}\right\}=o(p_n)$. The determinant of $\Sigma_x$ satisfies $\log \left|\Sigma_x \right| \geq - p_n\log p_n$.

\vskip5pt
\noindent{\bf Assumption (B3)} Let 

\qquad \quad $g\left(u^{*} \right)=E_{Y|\bfx}\left[\left\{Y a^{\prime}(u^{*})+b^{\prime}(u^{*}) \right\} \exp \left\{Y (a(u^{*})-a(h_0))+ b(u^{*})-b(h_0)\right\} \right]$ 

\noindent where $h_0=\bfx^\prime\bbta_0$, then $\left| g(u) \right|\leq e^{cu}$ for some constant $c>0$.

Assumption (B3) is satisfied for many generalized linear models with canonical link functions, for example, normal linear models with identity link, binomial regression with logit link, exponential regression with log-link etc. However, it does not hold for Poisson regression with log-link.  


\begin{lemma}\label{lm:5}
	Let (B1) hold, then (a) $\| \bfx  \|^2/p_n \rightarrow 1$ almost surely. 
	
	(b) Further, $\| \bfx  \|^2/(n\varepsilon_n)^b p_n \rightarrow 0$ almost surely for some $b>0$ as $n\varepsilon_n^2\rightarrow\infty.$
	%
\end{lemma}


The statements of assumptions (A1)-(A2) remains unchanged here. Note that the constant $\alpha_\delta$ in (A1) is less than 1 almost surely by Lemma \ref{lm:5}. 
Under assumptions (A1)-(A2) and (B1)-(B3), below we restate the result for RIS-RP under Gaussian predictors.  

\begin{theorem}\label{thm:3}
	Let $\bta\sim N({\bf 0},\sigma_{\theta}^2 I)$, $f(y|{\bf x}, \bgm ,\Rg, \bta)$ be the conditional density of $y$ given the model induced by $\left(\bgm,\Rg\right)$, where $\Rg$ is as in (\ref{eq_rp2}) and (\ref{eq_rp18}), and $\bfx \sim N({\bf 0}, \Sigma_x)$ with $\Sigma_x$ satisfying (B1)-(B2). Let $\bbta_{0}$ be the true regression parameter with $\sum_j |\beta_{0,j} |<K$ for some constant $K$, and assumptions (A1)-(A2), (B3) hold. Assume further that the conditions (i)-(iii) of Theorem \ref{thm:1} hold.
	Then, 	
	\vskip5pt
	{\centering $P_{f_0}\left[\pi\left\{d(f,f_0)> 4\varepsilon_n| \mathcal{D}^n \right\} > 2 e^{-n\varepsilon_n^2/4 } \right] \leq 2 \displaystyle e^{-n\varepsilon_n^2/5 },  $ \par}
	\vskip5pt
	
	\noindent as ~$\min\{m_n,p_n,n\}\rightarrow\infty$, where $\pi\{ \cdot | \mathcal{D}^n\}$ is the posterior measure.
\end{theorem}

Before presenting the result for RIS-PCR we modify assumption (A3) as follows:

\noindent{\bf Assumption (A3$^\prime$)} ~ Let $X_{\bgm}$ be the sub-matrix of $X_n$
with columns corresponding to non-zero values of $\bgm$, and ${\bf x}_{\bgm}$ be a row of $X_{\bgm}$. Let $V_{\bgm}$ be the $m_n \times \pg$ matrix of $m_n$ eigenvectors corresponding to the first $m_n$ eigenvalues of $X_{\bgm}^{\prime} X_{\bgm}$.
Then, for each $\bgm$ and data point ${\bf x}_{\bgm}$,
\vskip5pt
{\centering $ \left. \| V_{\bgm} {\bf x}_{\bgm} \|^2 \right/\| {\bf x}_{\bgm} \|^2 \geq \alpha_n,$ \par}
\vskip5pt

\noindent where $\alpha_n \sim (n\varepsilon_n^2)^{-(1-b)}$ for some $b>0$ where the sequence $\{ \varepsilon_n^2\}$ is as in assumption (A2).

\vskip5pt

Finally we restate the result on RIS-PCR under Gaussian predictors as follows:
\begin{theorem}\label{thm:4}
	Let $\bta\sim N({\bf 0},\sigma_{\theta}^2 I)$, $f(y|{\bf x}, \bgm ,\Rg, \bta)$ be the conditional density of $y$ given the model induced by $\left(\bgm,\Rg\right)$, where $\Rg$ is as in (\ref{eq_rp2}) and (\ref{eq_rp16}), and $\bfx \sim N({\bf 0}, \Sigma_x)$ with $\Sigma_x$ satisfying (B1)-(B2). Let $\bbta_{0}$ be the true regression parameter with $\sum_j |\beta_{0,j} |<K$ for some constant $K$, and assumptions (A1), (A2),(A3$^\prime$) and (B3) hold. Assume further that the conditions (i)-(iii) of Theorem \ref{thm:2} hold. Then
	
	\vskip5pt
	{\centering $P_{f_0}\left[\pi\left\{d(f,f_0)> 4\varepsilon_n| \mathcal{D}^n \right\} > 2 e^{-n\varepsilon_n^2/4 } \right] \leq 2 e^{-n\varepsilon_n^2/5 },  $ \par}
	\vskip5pt
	
	\noindent as ~$\min\{m_n,p_n,n\}\rightarrow\infty$, where $\pi\{ \cdot | \mathcal{D}^n\}$ is the posterior measure.
\end{theorem}

\section{Numerical Studies} \label{sec:4_0}
In this section we assess TARP using various simulated and real datasets in comparison with a variety of frequentist and Bayesian methods.  

\subsection{Simulation Studies}\label{sec:4}
We consider four different simulation schemes (\emph{Schemes I} -- \emph{IV}),
focusing mainly on high-dimensional and weakly sparse regression problems with a variety of correlation structures in the predictors.  The sample size is taken to be $200$ in each case, while $p_n$ varies. 

\vskip5pt
\noindent{\it Competitors.} 
The frequentist methods we compare with are: Smoothly clipped absolute deviation, \emph{SCAD}; One-Step SCAD (\cite{FXZ2014}) which is a local linear solution of the SCAD optimization problem,  \emph{1-SCAD}; minimax concave penalty, \emph{MCP}; least absolute shrinkage and selection operator, \emph{LASSO}; ridge regression, \emph{Ridge}; elastic net \citep{ENET},  \emph{EN}; principal component regression, \emph{PCR}; sparse PCR, \emph{SPCR} \citep{SPCR}; and robust PCR, \emph{RPCR} \citep{RPCA}. 
We consider three Bayesian methods, viz., Bayesian compressed regression (\emph{BCR}), Bayesian shrinking and diffusing priors (\emph{BASAD}, \cite{NH2014}) and spike and slab lasso (\emph{SSLASSO}, \cite{sslasso}). The details on specifications of the competitors are provided in SM Section \ref{sm:3.3}. 

\vskip5pt
\noindent{\it The proposed method.} For TARP, $m_n$ is chosen in $[2\log p_n,3n/4]$, and $\delta$ is set at $2$ as the function $\max\{0,(1+\log(p_n/n))/2\}$ is close to 2 for the choices of $(n,p_n)$. The parameters of the inverse gamma priors on $\sigma^2$ are set to $0.02$, to make it minimally informative.  

\vskip5pt
\noindent{\it Simulation Schemes.}
In the first three simulation schemes, the predictors are generated from $N({\bf 0},\Sigma)$, with different choices of $p_n$, $\Sigma$ and the regression coefficients. In \emph{Scheme IV} we consider a functional regression setup. In all cases (except in Scheme \emph{III}) we randomly choose $50$ covariates to be active with regression coefficient $1$ each. Different methods are compared with respect to their performance in out of sample prediction. We calculate mean square prediction error (MSPE), empirical coverage probability (ECP) and the width of $50\%$ prediction intervals (PI) for each of 100 replicates of the datasets.
For each of the schemes, we consider $2$ choices of $p_n$, and present results corresponding to the first choice here. Results corresponding to the other choice of $p_n$ are provided in the SM Section \ref{sm:3.4}.


\afterpage{
\begin{figure}
\centering     
\subfigure[{\it Scheme I}, $p_n=2\times 10^3$.]{\label{fig1}  \includegraphics[height=2.5 in, width=3.1 in]{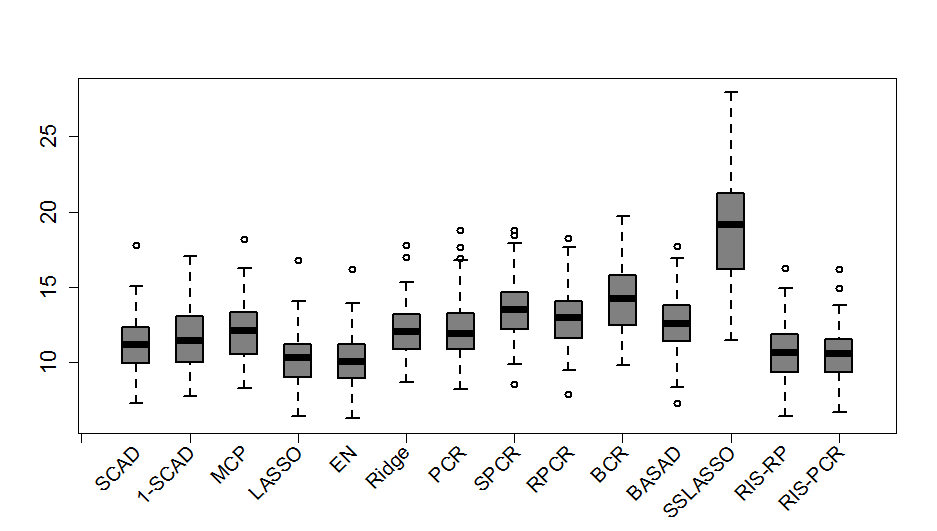}}
\subfigure[{\it Scheme II}, $p_n=10^4$.]{\label{fig2}  \includegraphics[height=2.5 in, width=3.1 in]{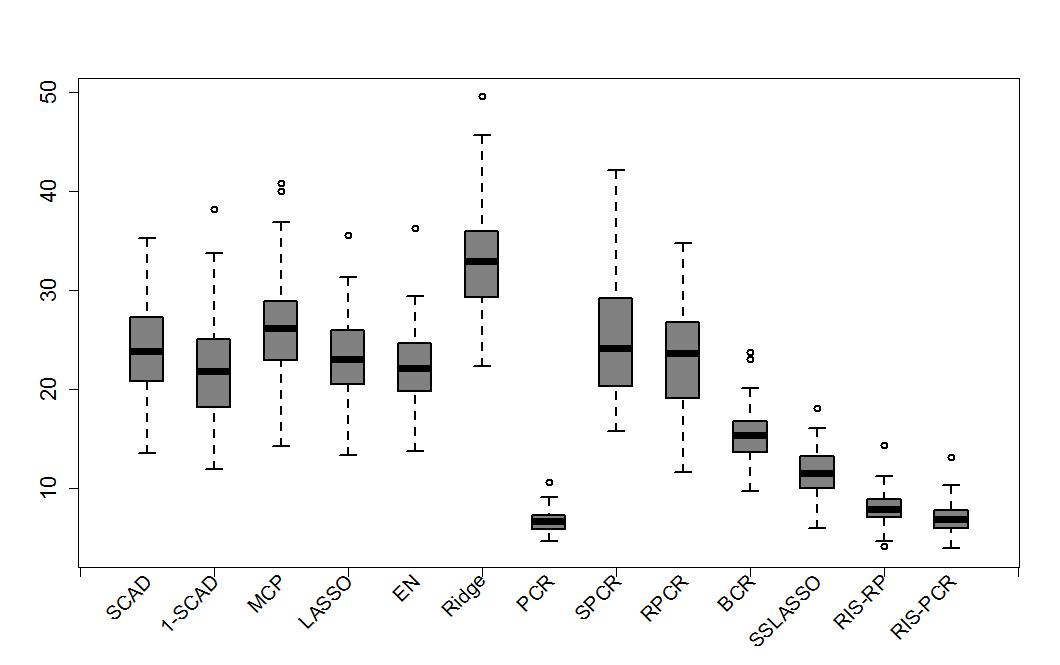}}
\caption{Box-plot of MSPEs for the competing methods in Schemes I and II. } 
\end{figure}
 }

\vskip5pt
\noindent{\it Scheme I: First order autoregressive structure.} $\Sigma_{i,j}=(0.3)^{|i-j|}$, $i,j=1,\ldots,p_n$, 
with $p_n \in \{2\times 10^3,3\times 10^3\}$. Fifty predictors are randomly chosen as active. Figures \ref{fig1} and SM \ref{fig-s1} show MSPE boxplots of the competing methods.

\vskip5pt
\noindent{\it Scheme II: Block diagonal covariance structure.} We choose $(p_n/100 -2)$ blocks of $100$ predictors each, along with $200$ independent predictors, with $p_n \in \{10^4,2\times 10^4\}$. The within-block correlation is $\rho$ and the across-block correlation is zero, with $\rho=0.3$ for half of the blocks and $\rho=0.9$ for the remaining half.  Among the $50$ active predictors, $49$ are chosen from blocks with $\rho=0.9$ and the remaining one is chosen from the independent block. Figures \ref{fig2} and SM \ref{fig-s2} show the results for {\it Scheme II}.

\vskip5pt
\noindent{\it Scheme III: Principal Component Regression.} We first choose a matrix $P$ with orthonormal columns, take $D=\mbox{diag}(15^2,10^2,7^2)$, and set $\Sigma=PDP^{\prime}$. We choose $\boldsymbol{\beta}=P_{\cdot,1}$, where $P_{\cdot,1}$ is the first column of $P$.  This produces an $X_n$ with three dominant principal components (PC), with the response $\yn$ dependent on the first PC and
$p_n \in \{10^4,5\times 10^4\}$. Further $5$ samples are perturbed to mimic outliers. For these samples, the regressors are chosen from independent $\mathrm{Normal}\left({\bf 0},10^2 I\right)$. Figures \ref{fig3}-\ref{fig4} and SM \ref{fig-s3}-\ref{fig-s4} show the results for {\it Scheme III}.

 \afterpage{
\begin{figure}
\centering     
\subfigure[MSPE of all the methods.]{\label{fig3}  \includegraphics[height=2.5 in, width=3.1 in]{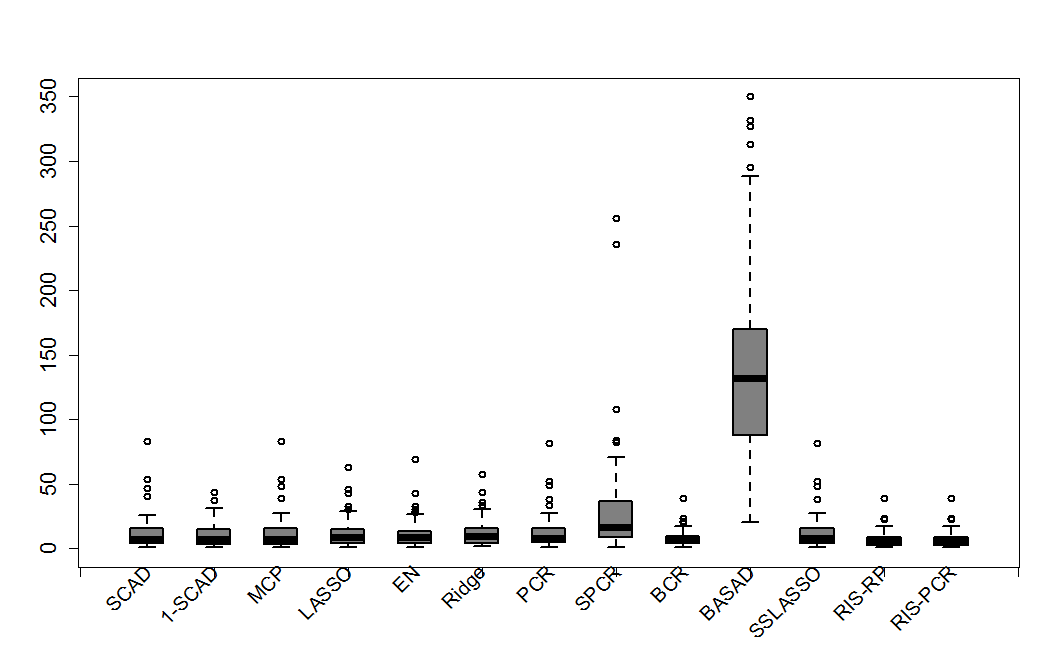}}
\subfigure[MSPE of selected methods.]{\label{fig4}  \includegraphics[height=2.5 in, width=3.1 in]{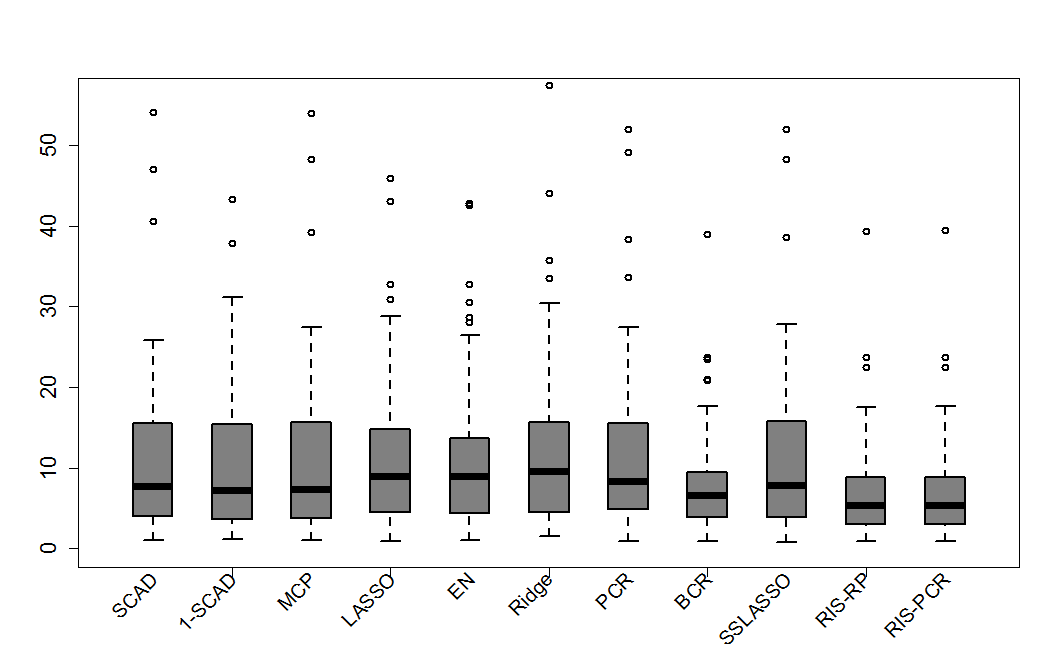}}
\caption{Box-plot of MSPEs for $p_n=5\times 10^3$ in \emph{Scheme III}. } 
\end{figure}
  }

\vskip5pt
\noindent{\it Scheme IV: Functional Regression.} Finally, we consider a functional regression setup, where the covariates are generated from Brownian bridge $B_t$ with $t\in (0,10)$ and values ranging from $(0,10)$. Choices of $p_n$ are $p_n=10^4$ and $2 \times 10^4$. Figures \ref{fig5}-\ref{fig6} and SM \ref{fig-s5}-\ref{fig-s6} show the results.

 \afterpage{
\begin{figure}
\centering     
\subfigure[MSPE of all the methods.]{\label{fig5}  \includegraphics[height=2.5 in, width=3.1 in]{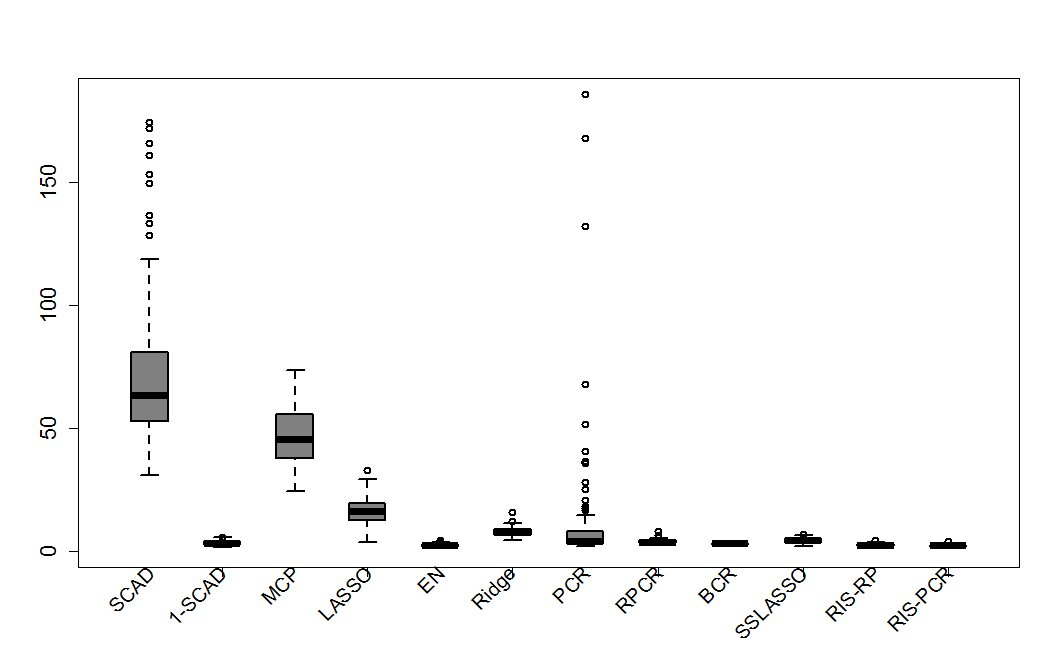}}
\subfigure[MSPE of selected methods.]{\label{fig6}  \includegraphics[height=2.5 in, width=3.1 in]{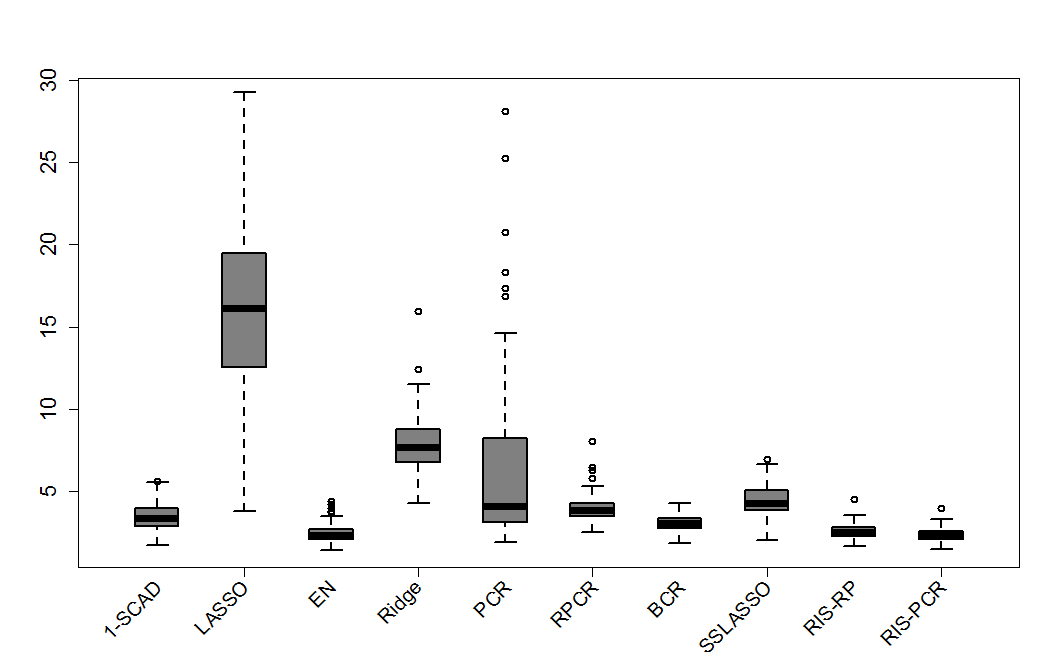}}
\caption{Box-plot of MSPEs for $p_n=10^4$ in \emph{Scheme IV}. } 
\end{figure}
 }

\vskip10pt

{\it Prediction Interval:} Empirical coverage probabilities (ECPs) and the width of $50\%$ prediction intervals (PI) of different methods are summarized in Tables \ref{tab1} and SM \ref{tab-s2}. The methods used to calculate PI for different competitors are described in SM Section \ref{sm:3.3}. 

In SM Section \ref{sm:3.5}, we discuss the simulation results in detail.

 \afterpage{
\begin{table*}[!h]
\renewcommand{\arraystretch}{1.2}
\begin{center}
{\scriptsize
\begin{threeparttable}
\caption{Mean and standard deviation (in brackets) of empirical coverage probabilities (ECP) and width of 50\% prediction intervals.}
\label{tab1}
\begin{tabular}{|p{1.5 cm}|p{1.4 cm} p{1.4 cm} | p{1.55 cm} p{1.4 cm} | p{1.4 cm} p{1.4 cm}| p{1.6 cm} p{1.6 cm} |}  \hline
Schemes $\rightarrow$  &  \multicolumn{2}{|c|}{\it I} & \multicolumn{2}{|c|}{\it II} & \multicolumn{2}{|c|}{\it III} & \multicolumn{2}{|c|}{\it IV} \\ 
$(n,p_n)$ &  \multicolumn{2}{|c|}{$(200,2\times 10^3)$} & \multicolumn{2}{|c|}{$(200,10^4)$} &  \multicolumn{2}{|c|}{$(200,5\times 10^3)$} &  \multicolumn{2}{|c|}{$(200,10^4)$}\\ 
  Methods $\downarrow$ &  ECP  &  Width &  ECP & Width &  ECP  & Width &  ECP & Width  \\ \hline
{\it SCAD}     & 48.8  (6.3) & 4.46  (0.41) 	& 49.8   (8.4) & 6.6  (0.87) 	& 15.2   (12.5) & 1.34   (0.06) 	& 41.2   (5.3) & 27.62    (2.47) \\
{\it 1-SCAD} & 49.7 (10.5) & 4.66 (0.93) & 49.16 (11.3)  & 6.26 (1.36)  & 17.5 (13.1)  & 1.40 (0.22) & 46.65 (10.1) & 2.36  (0.51) \\
{\it MCP}      	& 50.6   (7.2) & 4.75  (0.59) 	& 50.8   (9.8) & 6.96 (1.20) 	& 15.7   (12.8) & 1.35   (0.06) 	& 51.0   (5.8) & 11.42   (2.04) \\
{\it LASSO}   & 48.7   (8.1) & 4.23  (0.56) 	& 47.1   (8.6) & 6.27  (1.06) 	& 13.0   (12.4) & 1.40   (0.16) 	& 49.1  (6.8) & 8.97    ~~(1.06) \\
{\it EN}       	& 48.3   (7.3) & 4.21  (0.44) 	& 47.8   (7.9) & 6.15  (1.01) 	& 13.3  (12.9) & 1.40   (0.19) 	& 76.7   (8.4) & 3.83   ~(0.78) \\
{\it Ridge}  	& 48.0   (6.0) & 4.63  (0.37) 	& 46.1   (7.1) & 7.31  (0.92) 	& 12.3   (10.6) & 1.39   (0.15) 	& 76.3   (8.4) & 6.87   (0.91) \\
{\it PCR}    	& 38.9   (18.3) & 3.74  (2.04) 	& 47.1   (20.5) & 3.42  (1.67) 	& 15.8   (13.3) & 1.35   (0.06) 	& 44.9   (30.6)& 4.83    ~~(8.57) \\
{\it SPCR}   	& 50.4   (5.6) & 4.95  (0.26) 	& 49.6   (5.8) & 6.65  (0.72) 	& 13.3   (11.2) & 1.35   (0.08) 	& 49.5   (5.8) & 39.50~(15.69) \\
{\it RPCR}   	& 49.9   (5.5) & 4.81  (0.26) 	& 45.1   (4.8) & 5.83  (0.60) 	& \tnote{**} &     **           	& 49.6   (5.5) & 2.62   ~~(0.21)\\
{\it BCR}     	& 51.8   (6.0) & 5.30  (0.28) 	& 56.0   (5.4) & 6.01  (0.43) 	& 27.3   (11.8) & 2.02   (0.50) 	& 41.8   (4.8) & 1.92   ~~(0.11) \\
{\it BASAD} 	& 50.5   (5.8) & 4.92  (0.27) 	&  \tnote{*}   &    *         		& 14.2   (12.4) & 1.68   (2.35) 	&      *        & * \\
{\it SSLASSO}& 19.1  (4.1) & 2.12  (0.11) 	& 24.7   (4.9) & 2.08  (0.13) 	& 16.0   (13.2) & 1.37   (0.07) 	& 30.6   (5.2) & 1.64   ~~(0.10) \\
{\it RIS-RP}  	& 40.3   (6.2) & 3.47  (0.15) 	& 51.1   (5.7) & 3.82  (0.14) 	& 56.1   (17.5) & 4.52   (2.00) 	& 66.3   (5.6) & 3.06   ~~(0.22) \\
{\it RIS-PCR}& 33.2   (6.2) & 2.80  (0.22) 	& 34.9   (5.2) & 2.34  (0.20) 	& 25.0   (11.1) & 1.88   (0.56) 	& 40.6   (5.3) & 1.63   ~~(0.12) \\
\hline
\end{tabular}
\begin{tablenotes}
	\item [*] BASAD requires prohibitive computational time for $p_n=10^4$, and hence is removed from comparison
	\item [**] RPCR produces extremely high MSPE and width of PI for {\it Scheme III}, and hence is removed from comparison
	\end{tablenotes} 
\end{threeparttable}}
\end{center}
\end{table*}
 }
 
\vspace{-.05 in}
\subsection{Computational Time} \label{sec:4.1}
Computational time may depend on the simulation scheme due to varying level of complexity  in the dataset. We only present the computational time for \emph{Scheme IV} as an example. Figures \ref{fig11} and \ref{fig12} (in SM) present the time (in minutes) taken by different methods to compute $\hat{\bf y}_{new}$ using a single core, as $p_n$ grows and $n=n_{new}=100$. We run all the methods in R 3.5.1 in a 64 bit Windows 10 Pro desktop with 256 GB random access memory and Intel(R) Xeon(R) Gold 6140 CPU \@ 2.30GHz, 2295 Mhz, 18 Core(s) processor.

 We record the system time of SCAD reported by \emph{SIS} package, MCP by \emph{ncpen} package, LASSO by \emph{biglasso} and EN and Ridge by \emph{glmnet}.    
SIS package chooses $100$ grid values for $\lambda$ by default. For LASSO, EN and Ridge we consider a grid of $200$ different values of $\lambda$ to search from. For BASAD, we consider the \emph{alternative sampling} option which is recommended for large dimensional datasets. Other specifications are kept the same as in Section \ref{sec:4}. We report the computational time only if it is less than 10 hours. In SM Section \ref{sm:3.6}, we discuss the results in detail. 
\begin{figure}
\centering     
\includegraphics[scale=.52]{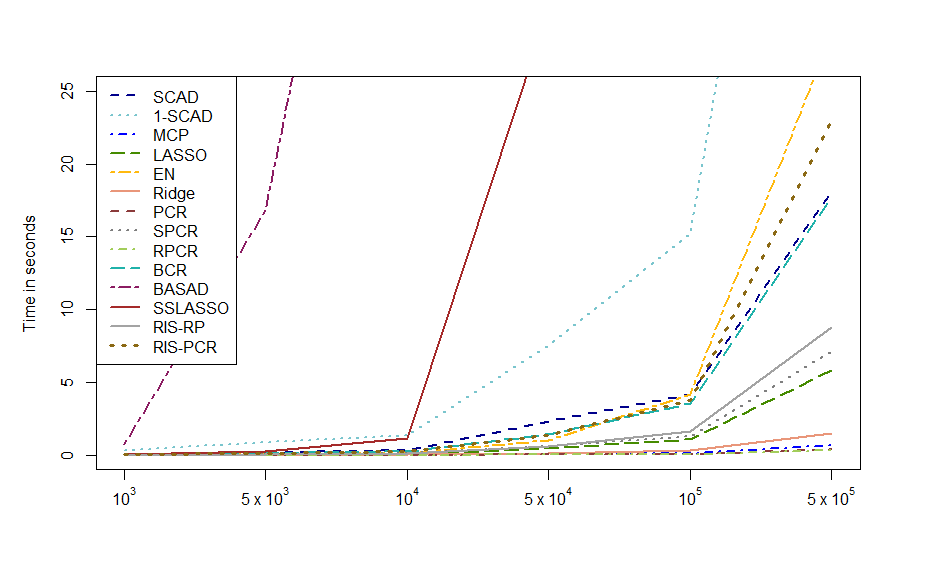}
\caption{Time required by different methods to predict $y$ as $p_n$ grows (magnified view). }
\label{fig11}
\end{figure}

 \subsection{Real Data Analysis} \label{sec:5}
We study the performance of TARP using $4$ real datasets, viz., \emph{Golub}, \emph{GTEx}, \emph{Eye} and \emph{GEUVADIS cis-eQTL} dataset. Results for the \emph{Golub} data are presented here. The other data descriptions, results and discussions are in Section \ref{sm:4} of SM. 
We assess out-of-sample predictive performance by averaging over multiple training-test splits of the data.
 
 \vskip5pt
 \noindent{\bf Golub data.} The Golub data (\url{https://github.com/ramhiser/datamicroarray/wiki}) consist of 47 patients with acute lymphoblastic leukemia (\emph{ALL}) and 25 patients with acute myeloid leukemia (\emph{AML}). Each of the 72 ($=n$) patients had bone marrow samples obtained at the time of diagnosis. Expression levels of 7129 ($=p_n$)  genes have been measured for each patient. We consider a training set of size $60$ with $20$ \emph{AML} patients, and $40$ \emph{ALL} patients. The test set consists of the remaining $12$ samples. 

As Golub data has nominal response, the methods are evaluated by the misclassification rate (in \%) and the area under the receiver operating characteristic (ROC) curve. Table \ref{tab2} provides the average and standard deviation (sd) of percentages of misclassifications, and those for the area under the ROC curve over 100 random subsets of the same size for the competing methods. The results are discussed in Section \ref{sm:4.0}.

We further compare predictive performance of the methods in terms of mean squared differences of predictive and empirical probabilities for Golub and GTEx datasets. 
We provide the details of the predictive calibration in Section \ref{sm:4.2}.

 \afterpage{
\begin{table*}[!h]
\renewcommand{\arraystretch}{1.5}
\begin{center}
{\scriptsize
\caption{Mean and standard deviation (in brackets) of percentage of misclassification and area under the ROC curve for \emph{Golub} dataset.}
\label{tab2}
\begin{tabular}{|p{2 cm}|p{.7 cm} p{.7 cm}   p{.7 cm}  p{0.7 cm} p{0.7 cm} p{0.7 cm} p{0.7 cm} p{0.7 cm} p{0.7 cm} p{0.7 cm} p{0.7 cm} p{0.7 cm}|}  \hline
Methods & \quad SCAD & 1-SCAD &  \quad MCP & \quad LASSO & \quad EN & \quad Ridge & \quad PCR &\qquad SPCR & \qquad RPCR & \quad BCR &  RIS-RP & RIS-PCR  \\ \hline
Misclassification Rate (in \%)  & 11.91 (7.24) & 45.45 (0.00) &  12.91 (7.36) & 13.04 (7.07) & 10.18 (7.11) & 8.73 (7.21) & 8.41 (5.93) & 42.14 (13.22)   & 10.13 (7.22)& 19.45 (9.56)  &  5.45 (4.14) & 5.59 (4.23) \\
\hline
Area under ROC curve  & 0.871 (.078)  & 0.500 (.000)&  0.864 (.079) & 0.858 (.078) & 0.890 (.078) & 0.904 (.079) & 0.908 (.065) & 0.577 (.133) & 0.891 (.079)&  0.817 (.093) &  0.985 (.021) & 0.946 (.041)  \\ \hline
\end{tabular}}
\end{center}
\end{table*}}

The Bayesian methods SSLASSO and BASAD either do not have codes available for binary responses (SSLASSO) or the code crashes for this dataset (BASAD).
For TARP based methods, we used correlation coefficient as the marginal utility measure to save computational time. 
For binary response data, marginal correlation coefficient ($r$) of a standardized predictor is equivalent to the $t$-test statistic of difference of means between groups. Therefore $r$ is also a natural measure of utility. 


\subsection{Discussion of Results}\label{sec:4.3} From the results it is evident that TARP can be conveniently used in ultrahigh-dimensional prediction problems, for e.g., the GEUVADIS cis-eQTL dataset we analyze has $\sim 8$ million predictors (see SM Section \ref{sm:4.3}). For highly correlated datasets, where most of the penalization methods fail (Schemes II and IV), TARP yields much better out-of-sample predictive performance. Another advantage of TARP over other methods is stability, and robustness in the presence of outliers (see, for e.g., results of Scheme III). The interquartile ranges of MSPEs, and widths of $50\%$ PIs, are low for TARP in all the cases. RIS-RP yields better ECP than RIS-PCR in general, and shows overall good predictive performance in the simulated and real datasets (see, in particular, results of Scheme III and Eye data).

TARP enjoys better classification performance as well, with lower misclassification rates and higher area under the ROC curve (see Table \ref{tab2} and SM Table \ref{tab:sm2}). Computational time of TARP is comparable to many frequentist methods even without using parallelization or a lower-level language such as $\mathrm{C++}$. TARP significantly outperforms popular frequentist and Bayesian methods in terms of computational time when $p_n$ is large (see Figures \ref{fig11} and \ref{fig12}). RIS-RP is much faster than RIS-PCR. The difference is due to the computation of the exact SVD of $X_{\bgm}$ involved in RIS-PCR. However, this burden would immediately be reduced if one uses a time-efficient approximation of the SVD. 

 \section{Appendix}\label{sec:7}
 This section contains proofs of Theorems \ref{thm:1} and \ref{thm:3}. In proofs of all the results, unless value of a constant is important we use a generic notation $c$ for the constants, although all of them may not be equal.
 \subsection{Some Useful Results}
 \begin{lemma}\label{lm:1}
 	Let  $\varepsilon_n$ be a sequence of positive numbers such that $n\varepsilon_n^2\succ 1$. Then under conditions
 	
 	\noindent a.~ $ln N(\varepsilon_n,\mathcal{P}_n)\leq n\varepsilon_n^2$ for all sufficiently large $n$,
 	
 	\noindent b. ~$\pi \left( \mathcal{P}_n^c \right)\leq e^{-2n\varepsilon_n^2}$ for all sufficiently large $n$, and 
 	
 \noindent	c.~ $\pi\left\{ f : d_t(f,f_0) \leq \varepsilon_n^2/4\right\} \geq \exp\{-n\varepsilon_n^2/4\}$ for all sufficiently large $n$ and for some $t>0$,
 	
 	\qquad $P_{f_0}\left[\pi\left\{d(f,f_0)> 4\varepsilon_n| (\yn,{\bf X})\right\} > 2 e^{-n\varepsilon_n^2\left(0.5\wedge (t/4)\right) } \right] \leq 2 e^{-n\varepsilon_n^2\left(0.5\wedge (t/4)\right) }.  $
 \end{lemma}
 {\noindent The proof is given in \cite{Jiang_2007}.}
 
 \begin{lemma}\label{lm:2}
 	Let assumption (A1) hold and ${\bf x}$ be a $p_n\times 1$ sample vector of the regressors. Then the following holds as $\min\left\{ n,m_n,p_n \right\}\rightarrow\infty$:
 	
 	a.~ The random matrix $\Rg$ described in (\ref{eq_rp2}) and (\ref{eq_rp18}) satisfies
 	~~$ \|\Rg {\bf x}\|^2 /(m_n p_n ) \xrightarrow{p} c\ad. $
 	
 	b.~ Let $\|{\bf x}_{\bgm}\|^2=\sum_{j=1}^{p_n} x_j^2 \gamma_j$ where $\gamma_j$ is the $j^{th}$ element of the vector $\bgm$, then 
 	
 	$  \| {\bf x}_{\bgm}\|^2 /p_n \xrightarrow{p} c\ad, $~
 	where $c$ is the proportionality constant in (\ref{eq_rp2}).
 \end{lemma} 

 \begin{lemma}\label{lm:3}
 	Under the setup of Theorem \ref{thm:1}, suppose $\bta \sim N(0,\sigma_{\theta}^2 I)$, then for a given $\Rg$, ${\bf x}$ and $y$ the following holds
 	$$P(|(\Rg {\bf  x})^{\prime} \bta - {\bf x}^{\prime} \bbta_0 |< \Delta)> \exp \left\{ - \frac{ ({\bf x}^{\prime} \bbta_0)^2 +\Delta^2 }{\sigma_{\theta}^2 \|\Rg {\bf x}\|^2 } \right\} \frac{2^4 \Delta^4}{\sigma_{\theta}^2 \|\Rg {\bf x}\|^2}.$$ 
 \end{lemma}
 {\noindent The proof is given in \cite{GD_2015}.}

 \subsection{ Proof of the Theorem \ref{thm:1}}
 Without loss of generality we consider $|x_j|<M$ with $M=1$, $j=1,2,\ldots,p_n$ and $d(\sigma^2)=1$, although the proofs go through for any fixed value of $M$ and $\sigma^2$.
 
 \begin{proof}

 	Define the sequence of events 
 	$\mathcal{B}_n= \left\{ \pi\left\{d(f,f_0)> 4\varepsilon_n| (\yn,{\bf X})\right\} > 2 e^{-n\varepsilon_n^2/4 }  \right\}$ and we need to show $P \left( \mathcal{B}_n^c \right) > 1- 2 e^{-n\varepsilon^2/5}$.
 	We first consider the sequence of events $\mathcal{A}_n$ in assumption (A2), and show that $P \left( \mathcal{B}_n^c | \mathcal{A}_n \right) > 1- 2 e^{-n\varepsilon^2/4}$. The proof then follows from assumption (A2) for moderately large $n$.

 	The proof of $P \left( \mathcal{B}_n^c | \mathcal{A}_n \right) > 1- 2 e^{-n\varepsilon^2/4}$ hinges on showing the three conditions of Lemma \ref{lm:1} for the approximating distribution 
 	\begin{eqnarray}
 	f(y)=\exp\{y a(h)+b(h)+c(y)\}\mbox{ with }h=(\Rg {\bf x})^{\prime} \bta, \label{eq_rp8}  
 	\end{eqnarray}
 	and the true distribution $f_0$, where $\Rg$ is as given in (\ref{eq_rp16}).
 	
 	\vskip5pt
 	\noindent{\it Checking condition (a).}
 	Let $\mathcal{P}_n$ be the set of densities $f(y)$ stated above with parameter $|\theta_j|<c_n$, $j=1,2,\ldots,m_n$, where $\{c_n\}=\{\sigma_{\theta} \sqrt{5n} \varepsilon_n\}$ and $\bgm$ is such that $\bgm \in \mathcal{M}_k$, for some $k\in \{0,1, \ldots, p_n\}$, given $\mathcal{A}_n$. 
 	For any $\bgm$ the corresponding set of regression parameters can be covered by $\mathit{l}_{\infty}$ balls of the form $B=(v_j-\epsilon,v_j+\epsilon)_{j=1}^{m_n}$ of radius $\epsilon>0$ and center $v_j$. It takes $(c_n/\epsilon+1)^{m_n}$ balls to cover the parameter space for each model $\bgm$ in $\mathcal{P}_n$. There are at most $\min\{ \binom{p_n}{k}, p_n^{k_n}\}$ models for each $\bgm$ under consideration as we are only concerned with models in $\mathcal{A}_n$ (see assumption (A2)), and there are $(p_n+1)$ possible choices of $k$. Hence it requires at most $ N(\epsilon,k)\leq c p_n^{k_n+1}  (c_n/\epsilon+1)^{m_n}$ $\mathit{l}_\infty$ balls to cover the space of regression parameters $\mathcal{P}_n$, for some constant $c$.  
 	
 	Next we find the number of Hellinger balls required to cover $\mathcal{P}_n$. We first consider the KL distance between $f$ and $f_0$, then use the fact $d(f,f_0) \leq \left( d_0(f,f_0) \right)^{1/2}$. 
 	Given any density in $\mathcal{P}_n$, it can be represented by a set of regression parameters $(u_j)_{j=1}^{m_n}$ falling in one of these $N(\epsilon,k)$ balls $B=(v_j-\epsilon,v_j+\epsilon)_{j=1}^{m_n}$ and $p_{\bgm}=k$. More specifically, let $f_u$ and $f_v$ be two densities in $\mathcal{P}_n$ of the form (\ref{eq_rp8}), where $u=(\Rg {\bf x})^{\prime} \bta_1$, $v=(\Rg {\bf x})^{\prime} \bta_2$ with $|\theta_{i,j}|<c_n$, $i=1,2$ and $p_{\bgm}=k$, then
 	\begin{eqnarray}
 	d_0(f_u,f_v) &=& \int \int  f_v \log \left( \frac{f_v}{f_u} \right) \nu_y (dy) \nu_{\bf x}(d{\bf x}) \notag \\
 	&=& \int \int \left\{ y(a(u)-a(v)) + (b(u)-b(v)) \right\} f_v \nu_y (dy) \nu_{\bf x}(d{\bf x}) \notag \\
 	&=& \int (u-v)  \left\{ a^{\prime}(u_v) \left( - \frac{b^{\prime} (v)}{a^{\prime} (v)} \right) +b^{\prime} (u_v) \right\}  \nu_{\bf x}(d{\bf x}). \notag
 	\end{eqnarray}
 	The last expression is achieved by integrating with respect to $y$ and using mean value theorem, where $u_v$ is an intermediate point between $u$ and $v$. Next consider
 	\quad $ |u-v|= |(\Rg {\bf x})^{\prime} \bta_1 -(R_{\bgm} {\bf x})^{\prime} \bta_2| \leq \| \Rg {\bf x} \| \| \bta_1 - \bta_2 \|, $
 	using the Cauchy-Schwartz inequality.
 	Now, by Lemma \ref{lm:2} we have $\|\Rg {\bf x}\|^2/(m_n p_n) \xrightarrow{p} \ad$ as $n\rightarrow\infty$ for some constant $0<\ad<1$. Therefore we can assume that for sufficiently large $p_n$, $\|\Rg {\bf x}\|\leq   \sqrt{m_n p_n}$. 
 	Also, $\|\bta_1-\bta_2 \| \leq \sqrt{m_n} \epsilon$.
 	Combining these facts we have $|u-v| \leq  \epsilon m_n \sqrt{ p_n}$. Similarly $\max\{|u|,|v|\} \leq   c_n m_n \sqrt{ p_n} $. 
 	These together imply that
 	
 	\vskip5pt
 	{\centering $d_0(f_u,f_v) \leq \epsilon m_n \displaystyle\sqrt{p_n}   \left\{ \displaystyle\sup_{|h|\leq  c_n m_n \sqrt{ p_n} } |a^{\prime}(h)|   \sup_{|h|\leq  c_n m_n\sqrt{ p_n} } \frac{|b^{\prime}(h)|}{|a^{\prime}(h)|} \right\}. $ \par}
 	
 	\noindent Therefore $d(f_u,f_v) \leq \varepsilon_n$ if we choose 
 	
 \qquad \qquad $\epsilon= \varepsilon_n^2 /\left\{  m_n \sqrt{ p_n} \sup_{|h|\leq   c_n m_n \sqrt{ p_n}} |a^{\prime}(h)|   \sup_{|h|\leq c_n m_n \sqrt{ p_n} } \left( |b^{\prime}(h)| / |a^{\prime}(h)| \right) \right\}$.
 	
 \noindent	Therefore, density $f_u$ falls in a Hellinger ball of size $\varepsilon_n$, centered at $f_v$. As shown earlier, there are at most $N(\epsilon, k)$ such balls. Thus, the Hellinger covering number
 	\begin{eqnarray}
 	&&N(\varepsilon_n,\mathcal{P}_n ) \leq N(\epsilon, k) =c p_n^{k_n+1}   \left( \frac{c_n}{\epsilon}+1 \right)^{m_n} \notag \hspace{2 in}\\
 	&&~~~~~ = c p_n^{k_n+1} \left[   \left( \frac{c_n}{\varepsilon_n^2} \left\{  m_n \sqrt{ p_n} \sup_{|h|\leq  c_n m_n\sqrt{ p_n}} |a^{\prime}(h)|   \sup_{|h|\leq  c_n m_n\sqrt{ p_n} } \frac{ |b^{\prime}(h)| }{ |a^{\prime}(h)| } \right\} +1 \right) \right]^{m_n} \notag\\
 	&&~~~~ \leq c p_n^{k_n+1} \left( \frac{1}{\varepsilon_n^2}  D( c_n m_n  \sqrt{p_n})+1 \right)^{m_n}, \hspace{2.2 in} \notag
 	\end{eqnarray}
 	where $D(R)=R \sup_{h\leq R} |a^{\prime} h| \sup_{h\leq R} |b^{\prime}(h)/ a^{\prime} (h)|$. The logarithm of the above quantity is no more than
 	$\log c + (k_n+1) \log p_n  - m_n \log (\varepsilon_n^2)+ m_n \log \left(1+D(c_n m_n \sqrt{p_n}) \right),$
 	as $0<\varepsilon_n^2<1$. Using the assumptions in Theorem \ref{thm:1} condition (a) follows.
 	
 	\noindent{\it Checking condition (b)} For the $\mathcal{P}_n$ defined in condition (a),
 	$\pi (\mathcal{P}_n^c) \leq  \pi (\cup_{j=1}^{m_n} |\theta_j|>c_n).$
 	
 	\noindent Observe that $\pi ( |\theta_j|>c_n) \leq 2 \exp\{-c_n^2/(2\sigma_{\theta}^2 ) \} /\sqrt{2\pi c_n^2/\sigma_{\theta}^2}$ by Mills ratio. Now for the choice that $c_n=\sigma_{\theta} \sqrt{5n} \varepsilon_n$ the above quantity is $2 \exp\{-5 n \varepsilon_n^2/2  \} /\sqrt{10 \pi  n \varepsilon_n^2 }$. Therefore 
 	
 	{\centering$ \pi (\mathcal{P}_n^c) \leq \sum_{j=1}^{m_n} \pi ( |\theta_j|>c_n) \leq 2 m_n\exp\{-5 n \varepsilon_n^2 /2\}/\sqrt{10 \pi  n \varepsilon_n^2 } \leq e^{-2n\varepsilon_n^2}$ \par}
 	
 	\noindent for sufficiently large $n$. Thus condition (b) follows.  
 	
 	\noindent{\it Checking condition (c)}
 	Condition (c) is verified for $t=1$. Observe that
 	
 	{\centering $ d_{t=1}(f,f_0)=\displaystyle\int \int f_0 \displaystyle\left( \frac{f_0}{f} -1 \right) \nu_y (dy) \nu_{\bf x}(d{\bf x}). $\par}
 	
 	Integrating out $y$ we would get $\int E_{y|{\bf x}} \left[ \left\{ (f_0/f) (Y) -1 \right\} \right] \nu_{\bf x}(d{\bf x}). $ Note that under $f$ and $f_0$ we have same function of $y$ as given in (\ref{eq_rp8}) with $h={\bf x}^{\prime} \bbta_0$ for $f_0$. Therefore, the above can be written as $E_{\bf x} \left[ \left\{ (\Rg {\bf x})^{\prime} \bta - {\bf x}^{\prime} \bbta_0 \right\} g\left(u^{*} \right)  \right]$ using mean value theorem where $g$ is a continuous derivative function, and $u^{*}$ is an intermediate point between $(\Rg {\bf x})^{\prime} \bta$ and ${\bf x}^{\prime} \bbta_0$. Therefore, if $\left| (\Rg {\bf x})^{\prime} \bta - {\bf x}^{\prime} \bbta_0 \right|< \Delta_n$, then $|u^{*}|< |{\bf x}^{\prime} \bbta_0 |+\Delta_n$. This in turn implies that for sufficiently small $\Delta_n$, $|g(u^{*})|$ will be bounded, say by $M$.
 	Consider a positive constant $\Delta_n$. From Lemma \ref{lm:3} we have
 	\begin{eqnarray}
 	P(| (\Rg {\bf x})^{\prime} \bta - {\bf x}^{\prime} \bbta_0|< \Delta_n ) &=& \sum_{\bgm} P(| (\Rg {\bf x})^{\prime} \bta - {\bf x}^{\prime} \bbta_0|< \Delta_n | \bgm ) \pi(\bgm) \notag \\
 	&\geq & E_{\bgm} \left[ \exp \left\{ - \frac{ ({\bf x}^{\prime} \bbta_0)^2 +\Delta_n^2 }{\sigma_{\theta}^2 \|\Rg {\bf x}\|^2 } \right\} \frac{2^4 \Delta^4}{\sigma_{\theta}^2 \|\Rg {\bf x}\|^2} \right] \notag \\
 	&=& \frac{2^4 \Delta_n^4}{ ({\bf x}^{\prime} \bbta_0)^2 +\Delta_n^2 } E_{\bgm} \left\{ \frac{Z_{\bgm}}{m_n p_n} \exp \left(-\frac{Z_{\bgm}}{m_n p_n} \right)   \right\}, \label{eq_rp11}
 	\end{eqnarray}
 	where $Z_{\bgm}=\left\{ ({\bf x}^{\prime} \bbta_0)^2 +\Delta_n^2 \right\}/ \left\{\sigma_{\theta}^2 \|\Rg {\bf x} \|^2 /(m_n p_n) \right\}$. By part (a) of Lemma \ref{lm:2}, and continuous mapping theorem $Z_{\bgm} -z_n\xrightarrow{p} 0$ in $\bgm$ where $z_n= \left\{ ({\bf x}^{\prime} \bbta_0)^2 +\Delta_n^2 \right\}/ \left(\sigma_{\theta}^2 c \ad \right) > \Delta_n^2 / \left(\sigma_{\theta}^2 c\ad \right)$.
 	For some non-negative random variable $Z$ and non-random positive numbers $p$, $a$ and $b$, consider the following fact
 	\vspace{-.1 in}
 	\begin{eqnarray}
 	E\left(\frac{Z}{p} \exp \left\{ -\frac{Z}{p} \right\} \right) &\geq& a P\left(\frac{Z}{p} \exp \left\{ -\frac{Z}{p} \right\} > a \right) 
 	\geq a P\left(\frac{Z}{p}> \frac{a}{b}, \exp \left\{ -\frac{Z}{p}\right\} >ab \right) \notag \\
 	&=& a P \left(Z>\frac{ap}{b}, Z< - p\log (ab)\right) 
 	= a P\left(\frac{ap}{b} < Z< - p\log (ab)\right). ~~~\label{eq_rp10}
 	\end{eqnarray}
 	Replacing $Z$ by $Z_{\bgm}$, $p$ by $m_n p_n$ and taking $a=\Delta_n^2 \exp \{-n\varepsilon_n^2/3 \}/ ( \sigma_{\theta}^2 c\ad)$, and $b=m_n p_n$ $ \exp \{-n\varepsilon_n^2/3 \}$. Thus 
 	$-p \log (ab) = -m_n p_n $ $ \log \left[ \Delta_n^2 m_n p_n \exp \{-2n\varepsilon_n^2/3 \} / (\sigma_{\theta}^2 c \ad ) \right]> m_n p_n n \varepsilon_n^2/2$ and  $ap/b= \Delta_n^2/ \left(\sigma_{\theta}^2 c \ad \right)$ 
 	for sufficiently large $n$. Therefore the expression in (\ref{eq_rp10}) is greater than 
 	
 	{\centering $\displaystyle\frac{\Delta_n^2}{\sigma_{\theta}^2 c \ad} \displaystyle e^{-n\varepsilon_n^2/3 }  P\left(\frac{\Delta_n^2}{\sigma_{\theta}^2 c\ad}  \leq Z_{\bgm} \leq \frac{1}{2} m_n p_n n \varepsilon_n^2 \right)  $\par}
 	
 	\vskip10pt
 	Note that $({\bf x}^{\prime} \bbta_0)^2 <\sum_{j=1}^{p_n} |\beta_{0,j}|<K$, and the probability involved in the above expression can be shown to be bigger than some positive constant $p$ for sufficiently large $n$. Using these facts along with equation (\ref{eq_rp11}), we have $ P(| (\Rg {\bf x})^{\prime} \bta - {\bf x}^{\prime} \bbta_0|< \Delta_n ) > \exp\{ -n\varepsilon_n^2/4 \} $ by choosing $\Delta_n =\varepsilon_n^2/(4M)$. Thus condition (c) follows. 
 \end{proof}
 
 \subsection{Proof of Theorem \ref{thm:3}}
\begin{lemma}\label{lm:4}
Suppose assumption (A1) holds, and $\ad$ is as in (A1). Let ${\bf x}\sim N(0,\Sigma_x)$ with $\Sigma_x$ satisfying (B1). Then the statement of Lemma \ref{lm:2} holds almost surely in $\bfx$.
	%
	%
\end{lemma}

\begin{proof}[{\bf Proof of Theorem \ref{thm:3}}]
	As in Theorem \ref{thm:1}, we show that $P \left( \mathcal{B}_n^c | \mathcal{A}_n \right) > 1- 2 e^{-n\varepsilon_n^2/4}$ by checking the conditions (a)-(c) of Lemma \ref{lm:1}. 
	The proof of conditions (a)-(b) remains unchanged under this setup.
	%
	%
	
	To prove condition (c) we need to show that the prior probability of distributions $f$ of the form (\ref{eq_rp8}) satisfies 
	 $d_t(f,f_0) \leq \varepsilon_n^2/4$, for some $t>0$, is bigger than $\exp\left\{-n\varepsilon_n^2/4 \right\}$. Proceeding as in the proof of Theorem \ref{thm:1}, we have 
	
	{\centering
		$d_{t=1}(f,f_0)=E_{\bf x} \left[ \left\{ (\Rg {\bf x})^{\prime} \bta - {\bf x}^{\prime} \bbta_0 \right\} g\left(u^{*} \right)  \right]$ 
		\par}
	
	\noindent where $u^{*}$ is an intermediate point between $(\Rg {\bf x})^{\prime} \bta$ and ${\bf x}^{\prime} \bbta_0$. Here $g(\cdot)$ is as follows

	{\centering
$		g(u^{*}) = \left.\frac{\partial}{\partial h}E_{Y|\bfx}\left[\exp \left\{Y (a(h)-a(h_0))+ b(h)-b(h_0)\right\} \right]\right|_{h=u^{*}}$

$= E_{Y|\bfx}\left[\left\{Y a^{\prime}(u^{*})+b^{\prime}(u^{*}) \right\} \exp \left\{Y (a(u^{*})-a(h_0))+ b(u^{*})-b(h_0)\right\} \right], $
\par}

\noindent	where $h=(\Rg\bfx)^{\prime}\bta$, $h_0=\bfx^{\prime}\bbta_0$ and $u^{*}$ is an intermediate point between $h$ and $h_0$.
	Thus we now must show that the prior probability of $\{f~ : ~
	E_{\bf x} \left[ \left\{ (\Rg {\bf x})^{\prime} \bta - {\bf x}^{\prime} \bbta_0 \right\} g\left(u^{*} \right)  \right] \leq \varepsilon_n^2/4\}$
	~is bigger than $\exp\left\{-n\varepsilon_n^2/4 \right\}$. To this end we split the range of ${\bf x}$ into two parts $A_{p_n}=\{\bfx: \|\bfx\|<\sqrt{3} p_n \} $ and $A_{p_n}^c$, and note that
	\begin{eqnarray}
	E_{\bf x} \left[ \left\{ (\Rg {\bf x})^{\prime} \bta - {\bf x}^{\prime} \bbta_0 \right\} g\left(u^{*} \right)  \right] = E_{\bf x} \left[\left. \left\{ (\Rg {\bf x})^{\prime} \bta - {\bf x}^{\prime} \bbta_0 \right\} g\left(u^{*} \right)\right|  A_{p_n}\right] P\left( A_{p_n}\right) \notag \\
	\hspace{2.5 in} + E_{\bf x} \left[\left. \left\{ (\Rg {\bf x})^{\prime} \bta - {\bf x}^{\prime} \bbta_0 \right\} g\left(u^{*} \right)\right|  A_{p_n}^c \right] P\left( A_{p_n}^c\right). \label{eqn_e4}
	\end{eqnarray}
	For two random variables $W_n$ and $U_n$, depending on the parameters $(\bta,\bgm)$, 
	
	\vskip5pt
	{\centering 
		$\pi\left(W_n+U_n \leq \varepsilon_n^2/4 \right)\geq\pi\left(|W_n|+ |U_n| \leq \varepsilon_n^2/4\right) \geq \pi\left(\left\{ |W_n|\leq \varepsilon_n^2/5 \right\}\bigcap \left\{ |U_n|\leq \varepsilon_n^2/20 \right\}  \right).$ 
		\par}
	\vskip5pt
	
	\noindent Define sequence of events $\{A_n\}=\left\{ |W_n| \leq \varepsilon_n^2/5 \right\}$ and $\{B_n\}=\left\{ |U_n| \leq \varepsilon_n^2/20 \right\}$. To show $\pi(A_n \cap B_n)\geq \exp\{ -n\varepsilon_n^2/4 \}$, it is enough to prove that 
	\begin{equation}
	\pi(A_n) \geq \exp\{-n\varepsilon_n^2/4 +2\log 2 \}~~\mathrm{ and}~~ \pi(B_n)\geq 1-\exp\{ -n\varepsilon_n^2/4 +4\log 2\}.
	\label{eqn_e2}
	\end{equation}
	Showing the first part of equation (\ref{eqn_e2}) is essentially same as the proof of condition (c) in Theorem \ref{thm:1}. The only part which requires attention is the proof of the claim that 
	
	\vskip10pt
	{\centering
		$P\left(\displaystyle\frac{\Delta_n^2}{\sigma_{\theta}^2 c\ad}  \leq Z_{\bgm} \leq \frac{1}{2} m_n p_n n \varepsilon_n^2 \right)$ ~~ is bigger than some constant $p$,
		\par}
	
	\vskip10pt
	\noindent  where $Z_{\bgm}=\left\{ ({\bf x}^{\prime} \bbta_0)^2 +\Delta_n^2 \right\}/ \left\{ \sigma_{\theta}^2 \| {\bf x}_{\bgm} \|^2 / (m_n p_n) \right\}$ and $\Delta_n$ is a constant. Note that by the continuous mapping theorem $Z_{\bgm} - z_n \xrightarrow{p} 0$ where $z_n=\left\{ ({\bf x}^{\prime} \bbta_0)^2 +\Delta_n^2 \right\}/ \left( \sigma_{\theta}^2 c \alpha_\delta \right)$. Note further that $\left({\bf x}^{\prime}\bbta_0\right)^2/p_n  n \varepsilon_n^2 \leq \|\bfx\|^2 \|\bbta_{0}\|^2/p_n  n \varepsilon_n^2 \rightarrow 0$ almost surely in $\bfx$ by Lemma \ref{lm:5} (b), and thus the statement holds almost surely under the restriction $\|\bfx_n\|<\sqrt{3}p_n$. Therefore, $ \Delta_n^2/(\sigma_{\theta}^2 c\ad)\leq  z_n\leq m_n p_n n \varepsilon_n^2/2$ with probability one in $\bfx$, and the probability in the above expression is bigger than some constant $p$. Hence the first part of (\ref{eqn_e2}) is proved.

Next we prove the second part of (\ref{eqn_e2}).
Consider a set $D_n$, such that $\pi\left(\left. (\bta,\bgm) \in D_n \right| \mathcal{A}_n\right) \geq \exp\{-n\varepsilon_n^2/4 +2\log 2 \}$, and show that $D_n\subseteq B_n$ for each $n$, implying $\pi(B_n)\geq \pi(D_n)$.
	For that we consider any $\bgm\in \cup_l \mathcal{M}_l$ (see assumption (A2)) and any $\bta: \|\bta\|\leq \sigma_\theta \sqrt{3n} \varepsilon_n/\sqrt{2}$. To see $\pi  \left(D_n|\mathcal{A}_n\right)\geq \exp\{-n\varepsilon_n^2/4 +2\log 2 \}$, observe that $\|\bta\|^2/\sigma_{\theta}^2$ follows a central Chi-squared distribution with degrees of freedom $m_n$. Therefore
	$P_{\bta} \left( \|\bta\|^2/\sigma_{\theta}^2 > m_n + 2 \sqrt{m_n t}+2t \right)\leq e^{-t} $ by \citet[Lemma 1]{chi_sq_bound}. Choosing $t=n\varepsilon_n^2/4-2\log 2$, we get $m_n + 2 \sqrt{m_n t}+2t< 3n\varepsilon_n^2/2$.
	
	Next note that the quantity  $\left\{(\Rg \bfx)^{\prime}\bta-\bfx^\prime\bbta_0\right\}\leq \left\|\Rg \bfx \right\|\|\bta\|+|\bfx^\prime\bbta_0|$. Using Result \ref{res1}, and by a simple application of Markov inequality it can be shown that $\left. \|\Rg \bfx \|\right/\left(m_n \|\bfx\| \right)=O_p(1)$. Therefore, for sufficiently large $n$, we consider 
	
	{\centering
		$\max\left\{\|\Rg \bfx \| \|\bta\|, \left| \bfx^{\prime} \bbta_0 \right|, \left|(\Rg \bfx)^{\prime}\bta-\bfx^\prime\bbta_0\right| \right\}\leq c\sqrt{n\varepsilon_n^2} m_n \|\bfx\|$, \par}
	
	\noindent for a suitable constant $c>0$.
	Further by assumption (B3), $|g(u)|\leq \exp\{c_0 u\}$ for some fixed $c_0>0$. Thus,
	
	{\centering
		$ E_\bfx \left[ \left. \left\{(\Rg \bfx)^{\prime}\bta-\bfx^\prime\bbta_0\right\} g\left(u^{*}\right) \right| \| \bfx \|> \sqrt{3} p_n \right] \leq E_\bfx  \left[ \left. \exp \left\{c\sqrt{n\varepsilon_n^2} m_n \|\bfx\| \right\}  \right| \| \bfx \|> \sqrt{3}p_n \right]$,\par}
	
	\noindent for a suitable constant $c>0$. Finally observe that
	\begin{eqnarray*}
E_\bfx  \left[ \left. \exp \left\{c\sqrt{n\varepsilon_n^2} m_n \|\bfx\| \right\}  \right| \| \bfx \|> \sqrt{3}p_n \right]
 = \frac{\left| \Sigma_x\right|^{-1/2}}{\left( 2 \pi\right)^{p_n/2}} \int_{ \| \bfx \|> \sqrt{3}p_n}\hspace{-.2 in}  \exp\left\{ c\sqrt{n\varepsilon_n^2} m_n \|\bfx\| -\frac{\bfx^{\prime} \Sigma_x^{-1} \bfx}{2} \right\} d\bfx\qquad
\end{eqnarray*}
\begin{eqnarray*}
		&& \leq \frac{(c p_n)^{p_n/2} }{\left( 2 \pi\right)^{p_n/2}} \int_{ \| \bfx \|> \sqrt{3}p_n}   \exp\left\{ c\sqrt{n\varepsilon_n^2} m_n \|\bfx\| -\frac{1}{c_0 l_n}\|\bfx\|^2 \right\} d\bfx \\
		&& = \frac{(c p_n)^{p_n/2} e^{c \hspace{.025 in}l_n n \varepsilon_n^2m_n^2}}{\left( 2 \pi\right)^{p_n/2}}
		\int_{ \| \bfx \|> \sqrt{3}p_n}  \hspace{-.1 in } \exp\left\{ -\frac{c_1}{c_0 l_n}\left(\|\bfx\|- cl_n\sqrt{n\varepsilon_n^2} m_n \right)^2 -\frac{c_2}{c_0 l_n} \|\bfx\|^2 \right\}d\bfx,\\
		&& \leq  \frac{(c p_n)^{p_n/2} e^{c \hspace{.025 in}l_n n \varepsilon_n^2m_n^2}}{\left( 2 \pi\right)^{p_n/2}} \exp\left\{ -\frac{3 c_1 p_n^2}{c_0l_n}\left(1- \frac{cl_n\sqrt{n\varepsilon_n^2} m_n}{\sqrt{3} p_n} \right)^2\right\} \int_{\|\bfx\|>\sqrt{3} p_n} e^{-c_2\|\bfx\|^2/(c_0l_n)}d\bfx
	\end{eqnarray*}
	where the first inequality is due to assumption (B2) and $c_1+c_2=1$. Thus the above quantity is less than
	
	\vskip5pt
	{\centering
$		\leq \exp\left[ -\displaystyle\frac{3 c_1 p_n^2}{c_0 l_n}\left\{\left(1- \frac{c~l_n\sqrt{n\varepsilon_n^2} m_n}{\sqrt{3}p_n} \right)^2+\frac{c~l_n^2 n\varepsilon_n^2 m_n^2}{p_n^2}+ \frac{cl_n}{p_n}\log\left(\frac{c_0l_n}{c_2}\right)+\frac{c}{p_n}\log(p_n)\right\}  \right]$
	$ \qquad \qquad \qquad\times P\left(\left\| \bfx\right\|>\sqrt{3}p_n | \bfx\sim N({\bf 0},c_0c_2^{-1}l_n \right)$
\par}
	\vskip5pt
	
\noindent Noting that $\max\left\{ l_n\log(l_n),l_n\sqrt{n\varepsilon_n^2} m_n\right\}=o(p_n)$
	
	{\centering 
		$|U_n|=E_{\bf x} \left[ \left\{ (\Rg {\bf x})^{\prime} \bta - {\bf x}^{\prime} \bbta_0 \right\} g\left(u^{*} \right)|A_{p_n}^c  \right] P\left(A_{p_n}^c\right) \leq \exp\{-cp_n\}\leq \varepsilon_n^2/20.$
		\par}
	
	\noindent This completes the proof. 
\end{proof}

 

\newpage 

{\centering

{\LARGE {\bf Supplementary Material}}
\par}

\vskip20pt

\emph{Supplementary Materials} contains additional discussion, numerical illustrations and mathematical details. It is organized in six sections. In Section \ref{sm:1}, we describe applicability of Randomized Independence Screening (RIS) as an independent screening algorithm. We apply RIS to four penalized likelihood methods, and compare their performance with other screening algorithms in light of the simulated datasets considered in Section \ref{sec:4} of the paper.  

Section \ref{sm:2p} is  a supplement to Section \ref{sec:2} of the paper. Choices of appropriate marginal utility measure for TARP under the GLM setup are discussed in Section \ref{sm:2p1}. Sensitivity of RIS-RP and RIS-PCR to different choices of tuning parameters is discussed in Section \ref{sm:3.1}. The method of simple aggregation is compared with other ensemble leaning approaches in Section \ref{sm:3.2}.

Section \ref{sm:2} is a supplement to Section \ref{sec:3} of the paper. It discusses the implications of the conditions and assumptions used in the results on predictive accuracy of TARP, and interpretation of the theorems (Section \ref{sm:2.1}).  Proofs of Lemma \ref{lm:5}, Theorem \ref{thm:2} and Theorem \ref{thm:4} are in Section \ref{sm:2.1.1}, \ref{sm:2.2} and \ref{sm:2.3}, respectively.

Section \ref{sm:3} contains additional simulation works.  The detailed specifications of the competing methods, and formulae applied to find prediction intervals for the competing methods are in Section \ref{sm:3.3}. Performance of the competing methods for some other choices of $p_n$ is shown in Section \ref{sm:3.4}. Detailed discussion of the competitive performance of the methods is in Section \ref{sm:3.5}.

Section \ref{sm:4} is a supplement to Section \ref{sec:5} of the paper. It contains results on three real datasets, the GTEx data, Eye data (in Section \ref{sm:4.0}) and the ultrahigh-dimensional  GEUVADIS cis-eQTL data (Section \ref{sm:4.3}).  Further, it contains predictive calibration of the methods when applied to the binary response datasets (see Section \ref{sm:4.2}).

The final section (Section \ref{sm:5}) contains a proofs of the Lemma \ref{lm:2} and \ref{lm:4} from the Appendix of the paper.

\section{RIS as a Screening Algorithm}\label{sm:1}
Randomized independence screening (RIS) can be considered as a general screening algorithm when the data dimension is very large and the predictors are highly correlated, for e.g., in genome-wide association studies. It serves as a time-efficient alternative to SIS for regressors with multicollinearity. 

Recall that calculation of the posterior mean of regression coefficients involves inversion of the matrix 
~$( X_n^{\prime} X_n  + I )^{-1}=I - X_n^{\prime}(I +X_n X_n^{\prime})^{-1}X_n$ (by Woodbury matrix identity). While the inversion step has complexity at most order $n^3$, the matrix multiplications are of order $n^2p_n$. Screening algorithms reduce the complexity of the whole operation to $\pg^2 n+\min\{\pg^3,n^3\}$, by reducing the number of regressors from $p_n$ to $\pg\ll p_n$. This reduction is in particular beneficial if $\pg<n$. However, due to multicollinearity and huge dimensionality, screening algorithms limited to $n$ marginally optimal predictors may not be appropriate. RIS provides the scope to access a larger list of predictors without making the computational cost much higher. We give a toy example to demonstrate that.

\paragraph{Example.} Consider a normal linear model having regressors with marginal correlation coefficients $|r_j|\approx (1-1/(2n))^{j-1}$. As $n$ is large the marginal correlations of the first $2n$ predictors exceed $\exp(-1)\approx 0.4$, and they should be included in the set of screened predictors. In RIS we take $q_j=|r_j|^\delta=(1-1/(2n))^{\delta(j-1)}$ as the inclusion probability of the $j^{th}$ predictor. The number of predictors selected, $\pg$, is a random quantity following Poisson binomial distribution \citep{PBD_1993} with an expected value $\sum_{j=1}^{p_n} (1-1/(2n))^{\delta(j-1)} \approx n$ for $\delta=2$, $\approx 2n/3$ for $\delta=3$ and so on, for sufficiently large $n$.   

The approach of generating multiple realizations compensates the loss due to randomization without greatly increasing the computational time. Consider for example a simple setup where $2n$ predictors have $|r_j|=0.5$ and the rest have $r_j=0$ (except possibly one with $|r_j|=1$). While it is important to consider all $2n$ predictors, RIS randomly selects about $\pg=n/2^{\delta-1}$ predictors. Even after $M$ repetitions, the computational time can still be less than considering $2n$ predictors if $\delta>(\log_2M+1)/2$.   

Below we provide a brief overview of performance of RIS screening. We consider $4$ methods, viz., RIS-LASSO, RIS-Ridge, RIS-SCAD and RIS-MCP, and compare these methods in $2$ simulation schemes, viz., Schemes I and II,  provided in Section \ref{sec:4}. 
In Section \ref{sec:4}, we provide the best result for SCAD and MCP among the results obtained using various packages including \emph{SIS}, which applies ISIS screening. Similarly, for LASSO and ridge we provide the best result obtained from \emph{glmnet} and \emph{biglasso} package, where the later uses sequential strong rule (SSR) screening.
The results for RIS-LASSO and RIS-Ridge are obtained using \emph{glmnet} package, and RIS-SCAD and RIS-MCP using \emph{ncvreg} package. For methods based on penalization, we do not include the search for optimal $\lambda$ under the aggregation step. We choose the best $\lambda$ for a fixed choice of screened regressors. The results are summarized in Table \ref{tab4}.

\afterpage{
	\begin{table*}[!h]
		\renewcommand{\arraystretch}{1.2}
		\begin{center}
			{\scriptsize
				\caption{Mean and standard deviation (sd) of mean square prediction error, empirical coverage probabilities and width of 50\% prediction intervals of four RIS screened methods.}
				\label{tab4}
				\begin{tabular}{|p{2 cm}|p{1.9 cm} p{1.6 cm}  p{1.6 cm}| p{2 cm}  p{1.6 cm} p{1.7 cm} |}  \hline
					Schemes $\rightarrow$  &  \multicolumn{3}{|c|}{\it I} & \multicolumn{3}{|c|}{\it II} \\ 
					$(n,p_n)$ &  \multicolumn{3}{|c|}{$(200,2\times 10^3)$} & \multicolumn{3}{|c|}{$(200,10^4)$} \\ 
					Methods $\downarrow$ & MSPE & ECP  &  Width &  MSPE & ECP & Width   \\ \hline
					{\it RIS-SCAD}    & 11.19 \; (2.17)  & 49.8 \; (6.4)  &   4.62 \; (0.40)& 
					6.95 \; ~(1.44) & 56.2 \; (6.0)   & 4.21  \; (0.29) \\
					{\it RIS-MCP}     & 10.53 \; (1.83)  & 49.2 \; (5.7)  & 4.49 \; (0.33)  & 
					7.09 \; ~(1.38) & 57.3 \; (6.4)  & 4.44 \; (0.42) \\
					{\it RIS-LASSO}   & 12.79 \; (1.90)  & 49.9 \;  (6.2) & 4.82 \; (0.37)  &
					7.82 \; ~(2.30)  & 49.5 \; (6.6) &    3.72 \; (0.49)\\
					{\it RIS-Ridge}   & 10.22 \; (1.71)  & 57.8 \; (10.3) & 5.21 \; (1.01)  & 
					13.33 \; (4.48)   & 55.8 \; (6.5) & 5.82 \; (0.39) \\ \hline 
				\end{tabular}
			}
		\end{center}
	\end{table*}
}

Observe that there is no visible improvement of SCAD and MCP of under RIS in {\it Scheme I}.
The gain due to RIS screening becomes visible in {\it Scheme II}. Here all the 4 methods show much lower MSPE, higher average ECP and lower width under RIS screening. In particular, for Ridge the averages MSPEs are more than $\sim 33$ under SSR screening (see Table \ref{tab1}), which are reduced to $\sim 13$ under RIS. Width of $50\%$ prediction interval (PI) also decrease for all the methods under RIS.

We compare the computational time of RIS-LASSO, RIS-Ridge and RIS-SCAD with that of LASSO, Ridge tuned by SSR screening (from \emph{Biglasso} package), and SCAD tuned by ISIS screening (from \emph{SIS} package). The results are summarized in Figure \ref{fig7}.
In terms of computational time, RIS based methods require marginally higher time for lower values of $p_n$ due to the aggregation step. However, RIS becomes much more efficient than SSR for higher values of $p_n$.
\afterpage{
	\begin{figure}
		\centering     
		\includegraphics[angle=270,scale=.4]{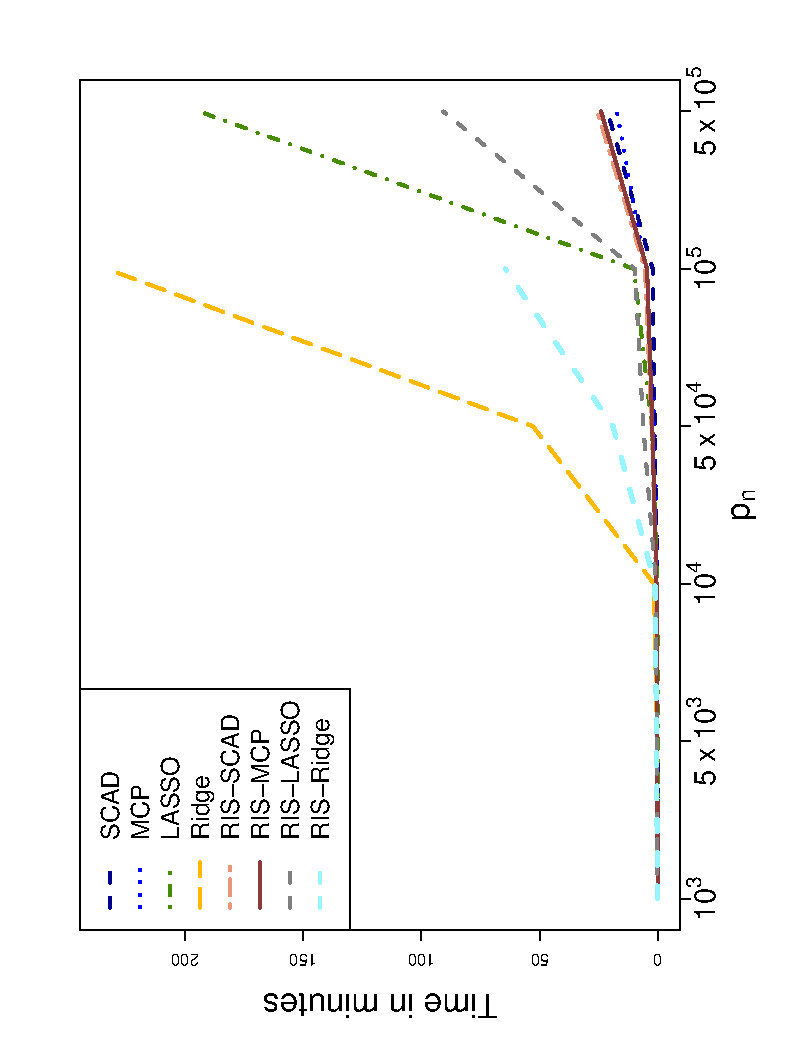}
		\caption{Computational time of SCAD, MCP, LASSO and Ridge under different screening algorithms}
		\label{fig7}
	\end{figure}
}

\section{An Addition to Section \ref{sec:2} of the Paper} \label{sm:2p}
\subsection{Choice of marginal utility measure for TARP in GLM:}\label{sm:2p1}
There are many suggestions of suitable choices of marginal utility functions for screening in GLMs (see, e.g., \cite{ZA_2000}, \cite{fan2009}, \cite{fan2010}, \cite{KS2017}). These include the maximum marginal likelihood estimator (MMLE), maximized marginal likelihood (MML), regression correlation coefficient (RCC), defined as the correlation between marginal estimates and observed response, etc. 
Theoretically, any marginal utility measure which satisfies assumptions (A1) and (A2) serves our purpose.

TARP is not sensitive to the choice of marginal utility function as long as it provides a measure of marginal association between predictors and response. For example, for a binary response, the marginal correlation coefficient of $y$ and a standardized predictor, $x_j$, is proportional to the $t$-test statistic for testing difference of means, and can be used as the marginal screening measure for TARP. 

Usage of MMLE, MML, RCC or equivalent criteria have relatively slow 
computational speed due to maximization of the likelihood. Therefore, we favor the correlation coefficient, as there are no significant differences in variable selection performance compared with using the MMLE, MML or RCC (see also \cite{SIS_package}). 

To emphasize this point we perform two simulation exercises: We consider a scheme similar to Scheme I of the paper with $p_n=10^3$ and $n=100$. The response variables are chosen to be binomial and Poisson, respectively, and are generated using the logit and log link. Arranging the predictors with respect to the marginal utility measure, we observe exactly the same order when using MMLE and absolute correlation. Here the MMLE is calculated via iteratively reweighted least squares (IWLS) available in the $\mathrm{glm}$  package of R. 

\vskip5pt
\noindent{\bf Complexity for non-Gaussian Likelihood:} In Step 4 (see Section \ref{sec:3.2}), we calculate the posterior mean of the compressed regression coefficient. For non-Gaussian likelihoods, the posterior mean is not available in analytic form. We can either rely on analytic approximations like Laplace method, or use MCMC here. However, as we are dealing with $m_n$ compressed regressors only, the computational cost is quite modest. 

For example, consider the situation where the response is binary.
Let $y$ follow a Probit regression model with $P(y_i=1)=\Phi({\bf z}_i^{\prime}\bta)$, where ${\bf z}_i$ is the $i^{th}$ row of $Z_n=X_nR_n^{\prime}$.  
Let $y_i^{*}$ be an auxiliary random variable such that $y^{*}_i\sim N({\bf z}_i^{\prime}\bta,1)$ and $y_i^{*}>0$ iff $y_i=1$, $i=1,\ldots,n$. 
Using Gibbs sampling, we sample from the full conditional distributions. 
The full conditional of $\bta$ given ${\bf y}^{*}$ and $\mathcal{D}^n$ is $m_n$-variate normal with mean $(Z_n^{\prime}Z_n+I)^{-1}Z_n^{\prime} {\bf y}^{*}$ and dispersion $(Z_n^{\prime}Z_n+I)^{-1}$. Given $\bta$ and ${\bf y}$, ${\bf y}^{*}$ is updated using a truncated normal density with mean $Z_n\bta$ and dispersion $I$. The computational cost in each iteration due to matrix multiplication and inversion is at most order $O(m_n^2 n)$. In $M$ MCMC iterations, total complexity of step 4 is $O(M m_n^2 n)$.

\subsection{Sensitivity of TARP to the tuning parameters}\label{sm:3.1}
TARP has two tuning parameters $m_n$ and $\delta$, and RIS-RP has an additional parameter $\psi$. To show the effect of aggregation we take different values of these tuning parameters, and present the results without aggregation. Although we do not take different values of $\delta$ while aggregating, here we will show the effect of different choices of $\delta$ as well.

The results presented here correspond to {\it Scheme I} of the paper with $p_n=2\times 10^3$. We consider $100$ samples and computed the mean square prediction error (MSPE), empirical coverage probability (ECP) and width of $50\%$ prediction intervals (PI). For each of the samples, we predicted $y_{new}$ without using the aggregation step. Four choices of each of the tuning parameters are taken. When one tuning parameter varies, the others are kept fixed at $m_n=100$, $\psi=0.25$ and $\delta=2$.
Table \ref{tab-s1} shows the results.
\afterpage{
	\begin{table*}[!h]
		\renewcommand{\arraystretch}{1.6}
		\begin{center}
			{\footnotesize
				\caption{Mean and sd of MSPE, ECP and width of 50\% PIs over 100 simulations for different choices of tuning parameters of RIS-RP and RIS-PCR.}
				\label{tab-s1}
				\begin{tabular}{|p{.4 cm}|p{1.25 cm} p{1 cm}  p{1.25 cm} | p{.4 cm} | p{1.25 cm} p{1 cm} p{1.25 cm}| p{.4 cm} |p{1.25 cm} p{1 cm}  p{1.25 cm} |}  \hline
					\multicolumn{12}{|c|}{\it RIS-RP}  \\ \hline 
					$m_n$& MSPE & ECP & Width & $\psi$ &MSPE & ECP & Width & $\delta$ & MSPE & ECP & Width\\ \hline
					40  & $16.13_{2.23}$ & $37.7_{5.6}$  & $3.95_{0.26}$ & $0.1$ & $17.74_{3.00}$ & $29.3_{5.0}$ & $3.16_{0.24}$ & $0.5$ & $23.35_{3.94}$ & $34.0_{5.3}$ & $4.19_{0.20}$\\
					80  & $17.22_{2.73}$ & $31.3_{5.3}$  & $3.37_{0.29}$ & $0.2$ & $17.53_{2.78}$ & $30.5_{5.2}$ & $3.22_{0.28}$ & $1.5$ & $20.05_{3.16}$ & $30.2_{4.9}$ & $3.47_{0.22}$\\
					120 & $19.48_{4.00}$ & $27.0_{5.4}$  & $3.03_{0.28}$ & $0.3$ & $18.23_{2.80}$ & $29.6_{5.0}$ & $3.18_{0.25}$ & $2.5$ & $15.62_{2.45}$ & $31.0_{5.1}$ & $3.13_{0.26}$ \\
					160 & $21.27_{4.34}$ & $25.4_{5.5}$  & $2.93_{0.33}$ & $0.4$ & $18.19_{2.88}$ & $29.4_{5.1}$ & $3.21_{0.22}$ & $3.5$ & $14.23_{2.19}$ & $35.8_{5.6}$ & $3.45_{0.30}$\\       \hline 
				\end{tabular}
				\begin{tabular}{|p{.4 cm}|p{1.25 cm} p{1 cm}  p{1.25 cm} | p{.4 cm} |p{1.25 cm} p{1 cm}  p{1.25 cm} |}  \hline
					\multicolumn{8}{|c|}{\it RIS-PCR}  \\ \hline 
					$m_n$& MSPE & ECP & Width  &$\delta$ & MSPE & ECP & Width\\ \hline
					40  & $13.20_{1.91}$ & $30.7_{5.7}$  & $2.90_{0.24}$   & $0.5$ & $12.32_{1.77}$ & $23.6_{4.6}$ & $2.11_{0.16}$\\
					80  & $13.85_{2.18}$ & $28.1_{5.0}$  & $2.70_{0.28}$   & $1.5$ & $13.36_{1.75}$ & $24.3_{4.9}$ & $2.28_{0.20}$\\
					120  & $15.60_{2.79}$ & $26.9_{5.7}$  & $2.67_{0.40}$ & $2.5$ & $15.56_{2.29}$ & $30.5_{5.2}$ & $3.07_{0.32}$\\
					180  & $17.67_{2.82}$ & $26.7_{5.8}$  & $2.85_{0.36}$ & $3.5$ & $14.23_{2.09}$ & $36.7_{6.0}$ & $3.49_{0.30}$\\ \hline 
				\end{tabular}
			}
		\end{center}
	\end{table*}
}

From Table \ref{tab-s1} the following statements can be made:

(i) $m_n$ is the number of linear combinations (principal components) of screened regressors considered for prediction in RIS-RP (RIS-PCR). TARP tends to perform better for lower values of $m_n$. This behavior is theoretically supported, as the conditions for predictive accuracy include $m_n \log p_n <n \varepsilon_n^2/4$ for all sufficiently large $n$ (see Theorems \ref{thm:1}, \ref{thm:2}). So smaller values of $m_n$ imply higher rate of convergence.   

(ii) $\psi$ controls the density of zeros in the random projection matrix in RIS-RP. Variation of $\psi$ does not seem to affect RIS-RP much. In fact we could take much sparser choices of the random matrix as described in Remark \ref{rm3} of the paper. 

(iii) $\delta$ controls the number of screened predictors in the RIS step. While RIS-RP tends to improve performance for higher values of $\delta$, RIS-PCR tends to deteriorate as $\delta$ increases.
The methods BCR and PCR correspond to cases with $\delta=0$ of RIS-RP and RIS-PCR, respectively. The differences in MSPEs of BCR and PCR in {\it Scheme I} also support this observation.  

(iv) RIS-RP seems to gain more advantage due to the aggregation step than RIS-PCR, which is again intuitive, as RIS-RP relies on the random projection matrix unlike RIS-PCR.

\subsection{Comparison of simple aggregation over other ensemble learning approaches }\label{sm:3.2}
We compare three approaches of aggregation, viz., simple averaging, cross-validation and model averaging, with respect to computational complexity and performance on simulated datasets. 

{\it Complexity of RIS-RP.} Recall the steps of RIS-RP in Section \ref{sec:3.2}: (i) Screening with complexity $O(p_n)$, (ii) random matrix generation and matrix post-multiplication with total complexity $O(n\pg m_n)$ where $\pg$ is number of selected regressors in the RIS step, and (iii) calculation of Bayes estimate with complexity $O(m_n^2n)$ as $m_n<n$. 


{\it Complexity of RIS-PCR.}  The second step of RIS-RP is replaced by SVD of $X_{\bgm}$ in RIS-PCR, which involves the complexity of $O\left(n\pg \min\{n,\pg\}\right)$, followed by a multiplication step of complexity $O(n \pg  m_n)$. 

\vskip5pt
\noindent{\bf Aggregating over different choices of tuning parameters.} Note that the first step of screening is not repeated over the steps of aggregation. 

{\it Model Averaging:} Suppose we consider $N$ different choices of $\{m_n, \psi, \bgm, \Rg \}$. For each of the $l$ choices we have a model $\mathcal{M}_l : y\sim f\left(y | {\bf x}, m_{n,l}, \psi_l, \bgm_l, R_{\bgm_l,l} \right)$ and a corresponding estimate of ${\bf y}_{new}$ given $X_{new}$, say $\hat{\bf y}_{new,l}$, where $l\in\{1,2, \ldots, N\}$. The method of model averaging puts forward the expected value of $\hat{\bf y}_{new,l}$ as an estimate of ${\bf y}_{new}$ as

\vskip3pt
\hspace{1.4 in }$\hat{\bf y}_{new} = \sum_{l=1}^{N} \hat{\bf y}_{new,l} P\left(\mathcal{M}_l | \mathcal{D}^n \right)$
\vskip3pt

\noindent where $P\left(\mathcal{M}_l | \mathcal{D}^n \right)$ is the posterior probability of $\mathcal{M}_l$. For normal linear models, as well as for non-Gaussian GLMs, the posterior probability, $P\left(\mathcal{M}_l | \mathcal{D}^n \right)$, requires calculation of $\left|Z_n^{\prime}Z_n +I\right|$ in addition to the components required to calculate $\hat{\bf y}_{new,l}$, which is of order $O(m_n^3)$. Therefore, for model averaging the complexity of step (ii) is increased by a term of $O(m_n^3)$, and the steps (ii) and (iii) are multiplied $N$ times each.

{\it K-fold Cross Validation (CV):} Like model averaging, for K-fold CV we consider $N$ different choices of $\{m_n, \psi, \bgm, \Rg \}$. 
For each of these choices, we split the training dataset into $K$ equal parts and obtain an estimate of ${\bf y}_{new}$, $\hat{\bf y}_{new,l}$, using $K-1$ parts. This estimate is then validated based on the remaining unused part, and a MSPE is obtained. The combined MSPE for the $l^{th}$ model is obtained by aggregating the $K$ MSPEs. Finally that model is considered which yields minimum MSPE. Clearly
K-fold CV requires $N$ repetitions of step (ii) and $KN$ repetitions of step (iii), although the last step now has complexity $m_n^2(n/K+m_n)$.

{\it Simple aggregation:} This method adds the least computational complexity to the method. If we consider $N$ different models $\mathcal{M}_l$, then the steps (ii) and (iii) are repeated $N$ times.

\vskip5pt
\paragraph{Performance of different methods of aggregation in simulated datasets.}  We compare different methods of aggregation under {\it Scheme II} (see Section \ref{sec:4} of the paper). We consider $n=100$ and $3$ choices of $p_n$, viz., $p_n=10^3,5\times 10^3$ and $10^4$ to see the effect of increments of dimension. 
\afterpage{
	\begin{figure}
		\centering     
		\subfigure[$(n,p_n)=(100, 10^3)$.]{\label{fig-rs1}  \includegraphics[height=2.5 in, width=2.5 in]{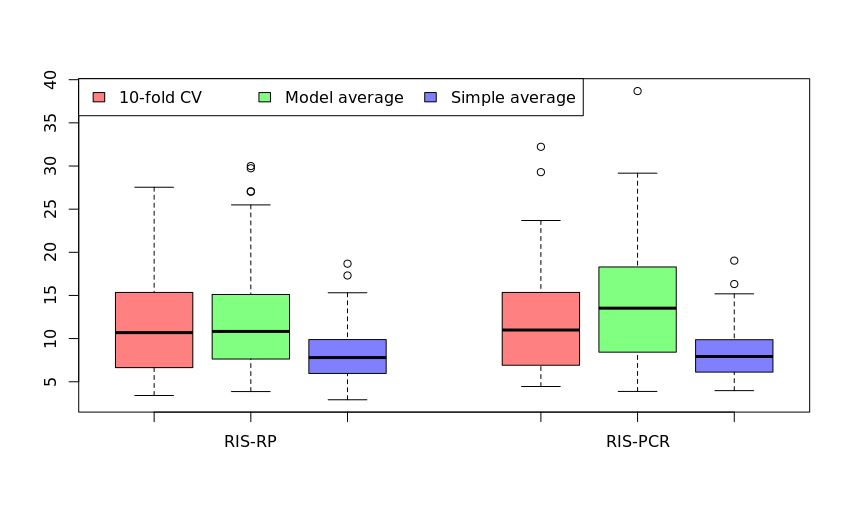}}
		\subfigure[$(n,p_n)=(100, 5\times 10^3)$.]{\label{fig-rs2}  \includegraphics[height=2.5 in, width=2.5 in]{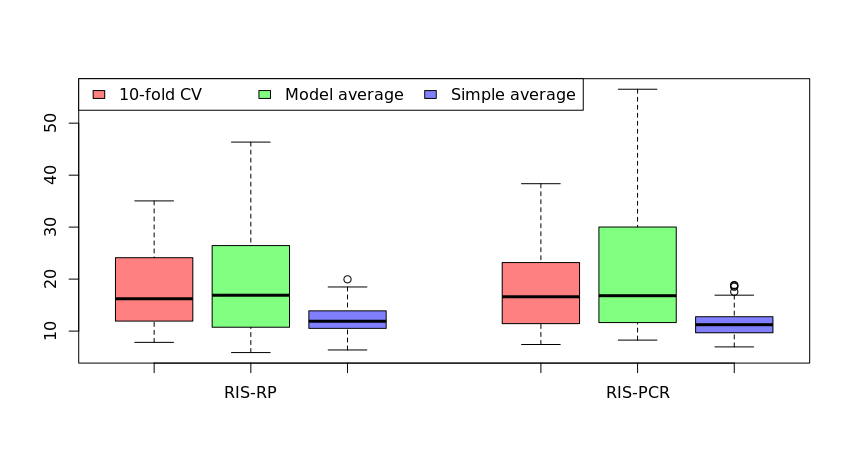}}
		\caption{Box-plot of MSPEs in {\it Scheme II}} 
	\end{figure}
	\begin{figure}
		\centering     
		\subfigure[$(n,p_n)=(100, 10^4)$.]{\label{fig-rs3}  \includegraphics[height=2.5 in, width=2.5 in]{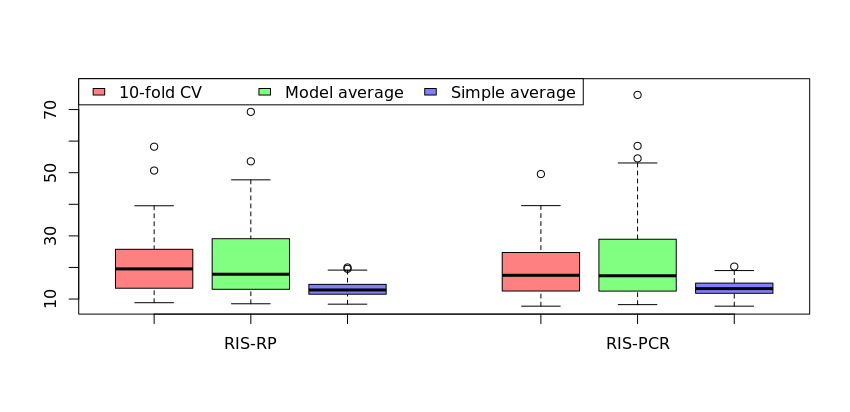}}
		\subfigure[Time in sec., $p_n=10^4$]{\label{fig-rs4}  \includegraphics[height=2.5 in, width=2.5 in]{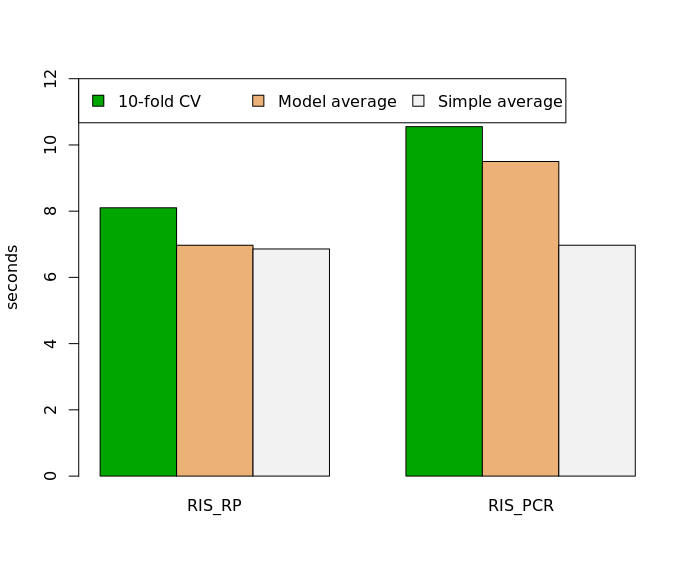}}
		\caption{Box-plot of MSPEs and multiple bar-chart of computational time in {\it Scheme II}. } 
	\end{figure}
	
}
From Figures \ref{fig-rs1}-\ref{fig-rs4}, observe that simple averaging has better and more stable performance than model averaging and cross validation. Difference in performance increases as $p_n$ increases. Model averaging tends to be more affected by increment of dimension. Simple aggregation also requires less time to compute, and the difference is significant for RIS-PCR.

\section{An Addition to Section \ref{sec:3} of the Paper}\label{sm:2}
\subsection{Implications of the assumptions and conditions required to prove the theorems}\label{sm:2.1}
The following two remarks discuss implications of Assumptions (A3) and (A3$^\prime$), and that of the conditions (i)-(iii) in Theorems \ref{thm:1} and \ref{thm:2}, respectively.
\begin{remark}\label{rm5}
	If the matrix $X_{\bgm}^{\prime} X_{\bgm}$ has rank less than $m_n$, then $\alpha_n=1$ by Perseval's identity. Suppose rank of $X_{\bgm}^{\prime} X_{\bgm}$, say $r_n (\leq n)$, is bigger than $m_n$. Then the row space of $X_{\bgm}$, or that of $X_{\bgm}^{\prime} X_{\bgm}$, is spanned by a set of $r_n$ basis vectors ${\bf v}_1, {\bf v}_2, \ldots, {\bf v}_{r_n}$. Therefore, any data point ${\bf x}$ can be written as a linear combination of these $r_n$ vectors as 
	${\bf x}=a_1 {\bf v}_1+ a_2 {\bf v}_2+ \cdots + a_{r_n} {\bf v}_{r_n}$,
	where $a_1, \ldots, a_{r_n}$ are constants not all equal to zero. As the vectors ${\bf v}_j$ are orthonormal, ${\bf v}_j^{\prime} {\bf x}=a_j$ for all $j=1, \ldots, r_n$, which in turn implies that ${\bf x}^{\prime} {\bf x}=\sum_{j=1}^{r_n} a_j^2$. Also, note that the first $m_n$ among these $r_n$ vectors constitute $V_{\bgm}^{\prime}$, which implies $\|V_{\bgm}^{\prime} {\bf x} \|^2=\sum_{j=1}^{m_n} a_j^2$. Thus $\| V_{\bgm} {\bf x} \|^2/\| {\bf x} \|^2= \sum_{j=1}^{m_n} a_j^2 / \sum_{j=1}^{r_n} a_j^2$, and magnitude of the ratio depends on the part of ${\bf x}$ explained by the last few principal component directions. The lower bounds $\alpha_n \sim (n\varepsilon_n^2)^{-1}$ (in (A3)) or $\alpha_n \sim (n\varepsilon_n^2)^{-1+b}$ for some $b>0$ (in (A3$^\prime$)) are reasonable in view of many real data 
	scenarios where most of the variation is explained by the first few principal components.
\end{remark}

\begin{remark}\label{rm6} 
	The conditions (i)-(iii) in Theorems \ref{thm:1} and \ref{thm:2} are related to the sizes of $p_n$, $m_n$ and $k_n$ in comparison with $n\varepsilon_n^2$.
	A sufficient condition for (i) is $m_n \log n < n\varepsilon_n^2/4$, providing an upper bound on the dimension of the subspace, $m_n$. Condition (ii) restricts the permissible number of regressors, $p_n$, and the number of possible models of each dimension. If there is a strict ordering in the marginal utilities $|r_{x_j,y}|$, so that $k_n\leq \kappa$ for some large number $\kappa$, then the condition reduces to $\log p_n < n\varepsilon_n^2/4$.  To illustrate that the condition (iii) tends to be weak, consider distributions of $y$ corresponding to Bernoulli, Poisson and normal. For these cases, the quantity $D(h^{*})$ is at most of order $O(h^{*})$. Therefore, condition (iii) does not impose much additional restriction over (i)-(ii), except $m_n \log p_n < n\varepsilon_n^2/4$, which induces a stronger upper-bound to $m_n$.  
\end{remark}

\subsubsection{Explanation of the Assumption (A2)}\label{sm:2.1.0}
We choose predictors based on marginal utility. Suppose the absolute marginal utilities, $q_j$, are normalized to range $[0,1]$. Then the ``inclusion probability'' of a model $\bgm$ is $q(\bgm)=\prod_{j=1}^{p_n} q_j^{\gamma_j} (1-q_j)^{(1-\gamma_j)}$.

Assumption (A2) effectively limits inclusion of models with small $q(\bgm)$, without adding any restriction to the model size. The class of models considered in (A2), $\mathcal{A}_n=\cup_l\mathcal{M}_l$, contains the top $p_n^{k_n}$ models (ordered w.r.t. marginal inclusion probabilities) of all dimension. Assumption (A2) makes $\mathcal{A}_n$ the effective model space.

Note that all models of dimension $l<k_n$ and $l>(p_n-k_n)$ belong to $\mathcal{A}_n$. Further models considered in SIS, i.e., the models (of any dimension) containing predictors with highest marginal utility, are included in $\mathcal{A}_n$. However, as SIS based methods consider only one selected model for further analysis, sparsity is necessary condition for SIS. But sparsity is not a necessary condition for (A2). Below we illustrate this with some simple examples.

\paragraph{Example 1: Strong Sparsity.}  Consider the situation where there exists a few active regressors. Under high signal-to-noise ratio and partial orthogonality condition, \cite{fan2010} show that the marginal regression parameter of inactive regressors will be zero, and those for the active regressors will exceed a threshold in such situations. They also show that the marginal maximum likelihood estimates (MMLE) will be close to their population counterparts almost surely. Therefore, it is likely that most of the regressors have $q_j$ close to zero, and a few regressors have $q_j$ bigger than a threshold.

We consider a simple scenario with $k_n=1$, $p_n=\exp\{n^s\}$ with $0<s<1$. 
Thus, $p_n$ highest probability models of each dimension are included in $\mathcal{M}_l$. 

{\it Case 1: } Suppose $p_n-m_n$ covariates have normalized utility $q_j \leq  c_1 n^{s_1} p_n^{-(1+\nu)}$, and the remaining $m_n$ have higher utility $q_j>(1-c_2 n^{s_2} p_n^{-\nu})$, for some constants $c_1,c_2,s_1,s_2,\nu>0$, and $m_n=O(n)$. In such cases the probability of the $m_n$-dimensional model with $m_n$ important predictors (i.e., predictors having strong marginal association with $y$) is bigger than
$$\left(1-\frac{c_2n^{s_2}}{p_n^{\nu}} \right)^{m_n} \left(1-\frac{c_1 n^{s_1}}{p_n^{1+\nu}} \right)^{p_n-m_n}\rightarrow 1,\quad \mbox{as}~n\rightarrow\infty.  $$
Further, the rate of convergence is approximately $1- \exp\{-\nu n^s\}$. This probability is greater than $1-\exp\{-n\varepsilon_n^2/4\}$ if $\nu=1$ (see conditions (i)-(iii) of Theorems \ref{thm:1}, \ref{thm:2}). However, this is only one among $p_n^2+1$ models we consider in (A2). Therefore, we expect (A2) to hold for some smaller choices of $\nu$ in this scenario.

{\it Case 2.} Next we slightly generalize the situation, where $m_n$ regressors have $q_j$ at least $1-c_1n^{s_1}p_n^{-\nu}$, ~$p_n-2m_n$ regressors have $q_j$ at most $c_2n^{s_2}p_n^{-(1+\nu)}$, and $m_n$ regressors having intermediate utilities. Let $m_n=O(n^r)$ where $r<s$. In such cases let us consider the probability, $\mathcal{P}$, of all the models having $m_n$ important regressors (i.e., regressors with $q_j>(1-c_1n^{s_1}p_n^{-\nu})$) included, and $p_n-2m_n$ unimportant regressors (i.e., regressors with $q_j<c_2n^{s_2}p_n^{-(1+\nu)}$) excluded. Note that there can be at most $2^{m_n}$ such models, and as $\binom{m_n}{l}<p_n$ for all $l=1,\ldots,m_n$, all these models are included in $\mathcal{A}_n$. Here also we get
\begin{eqnarray*}
	\mathcal{P}&\geq& \left(1-\frac{c_1 n^{s_1}}{p_n^\nu}\right)^{m_n} \left(1-\frac{c_2 n^{s_2}}{p_n^{1+\nu}} \right)^{p_n-2m_n} 
	\approx 1-\exp\{-\nu n^s\},
\end{eqnarray*}
which implies (A2) holds at least for $\nu=1$.

\paragraph{Example 2: Dense cases.} Now we consider some dense cases. Keeping the choices of $(k_n,p_n)$ as before, we consider the case where $p_n/2-m_n$ covariates have $q_j\leq c_1n^{s_1}p_n^{-(1+\nu)}$ and $p_n/2-m_n$ covariates have $q_j\geq (1-c_2 n^{s_2}p_n^{-(1+\nu)})$ for some constants $s_1,s_2, c_1, c_2,\nu>0$. The remaining $2m_n$ covariates have intermediate marginal utilities. Here we also consider $m_n=O(n^r)$ with $r<s$. As before, it is easy to see that the probability, $\mathcal{P}$, of $2^{2m_n}$ models including all the important covariates and excluding all the unimportant covariates is at least $1-\exp\{-\nu n^s\}$.

Further, observe that in each of the situations in Example 1, if we replace the marginal utility, $q_j$, of each regressor by $1-q_j$, then (A2) holds. For e.g., if $p_n-m_n$ covariates have $q_j\geq 1-c_1 n^{s_1}p_n^{-(1+\nu)}$ and the remaining $m_n$ have $q_j\leq c_2n^{s_2}p_n^{-\nu}$, a situation opposite to {\it Case 1} arises. In this case if we calculate the probability, $\mathcal{P}$, of the model with $p_n-m_n$ important covariates, it can similarly be shown that $\mathcal{P}>1-\exp\{-\nu n^s\}$.

{\bf Discussion.} It is difficult to check (A2) in situations other than extremely dense and sparse cases, as calculation of the probability $P(\{\bgm:\bgm\in\mathcal{A}_n\})$ is not trivial. However, (A2) holds if some of the covariates have sufficiently large marginal utility relative to the others. 
Situations where (A2) does not hold include the case where all the regressors have $q_j$s uniformly spread in the interval $[0,1]$. A convenient way of informally checking (A2) is to draw a histogram of the normalized utility measures. If it sufficiently deviates from uniform on $[0,1]$, then (A2) holds.

Further recall that, if $r_j$ is the marginal utility measure (e.g., correlation coefficient) of the $j^{th}$ covariate, then we define $q_j$ as the normalized value of $|r_j|^\delta$. As $\delta$ becomes larger, the distribution of $q_j$s deviates more from $\mathrm{Uniform}(0,1)$. Choosing a suitable $\delta$, we can control $q_j$s as well to satisfy (A2).

\subsubsection{Interpretation of the theorems}\label{sm:2.1.1}
In Theorems 1 and 2 we provide a rate of closeness of the true density $f_0$ and the estimated density $f$ of $y$ under the posterior distribution. These convergence results describe ``often closeness'' between $f$ and $f_0$ (see \cite{Jiang_2007}).
These results imply existence of point estimates of $f_0$ that have the convergence rate $\varepsilon_n$ in a frequentist sense. Such a point estimate can be obtained by finding the center of an $\varepsilon_n$-ball with high posterior probability, or by posterior
expectation (see \cite{GGV_2000}).

In Theorems 1 and 2, we have shown that the predictive density $f(y|\bta,\Rg,{\bf x}_i)$, with $\bta$ drawn from $\pi(\bta|\mathcal{D}^n,\Rg)$ and $\bgm\in \mathcal{A}_n$ ($\mathcal{A}_n=\cup_l\mathcal{M}_l$, as in Assumption (A2)), concentrates around the true predictive density under $f_0$ in the above sense. In order to find a point estimate for prediction, we consider the posterior mean of $\bta$, and simply average over multiple realizations of $\bgm$. To argue that the proposed point estimate is a ``good" estimate (i.e., it lies inside $\varepsilon_n$-Hellinger balls containing $f_0$), it is enough to show that each realization of $\bgm$ considered with inclusion probability $q(\bgm)$ is in $\mathcal{A}_n$, which is evident under the assumption (A2).

\subsubsection{Proof of Lemma \ref{lm:5}}\label{sm:2.1.1}
\begin{proof}[Proof of Lemma \ref{lm:5}(a)]
	By \citet[Theorem 1.8.E]{Serfling}, $\|\bfx\|^2/p_n=\sum_i x_i^2/p_n\rightarrow \sum_j E\left(x_j^2\right) /p_n$ almost surely if $cov\left(x_i^2,x_j^2\right)\leq \rho^{*}_{|i-j|}\sqrt{var\left(x_i^2\right)  var\left(x_j^2\right)}$ with  $\sum_j var\left( x_j^2\right)(\log j)^2/j^2 <\infty$, and $\sum_j \rho_j^{*}<\infty$.  Here $E(x_j^2)=1$, $var(x_j^2)=2$ for all $j$, and $cov(x_i^2,x_j^2)=2\sigma_{i,j}^2$. Therefore by (B1), $\|\bfx\|^2/p_n \rightarrow 1$ almost surely. 
	
	\paragraph{\it Proof of Lemma \ref{lm:5}(b)}  Let $Y_n=\|\bfx\|^2/p_n$. This part is proved noticing that for each $A_n =\left\{\omega:Y_n(\omega)\rightarrow 1\right\}$ implies $\left\{ \omega:Y_n(\omega)/(n\varepsilon_n^2)^{b}\rightarrow 0\right\} $ for any $b>0$ as $n\varepsilon_n^2\rightarrow\infty$. 
\end{proof}

\subsection{Proof of Theorem 2}\label{sm:2.2}
\begin{proof}
	The outline of the proof of Theorem \ref{thm:2} closely follows the arguments given in the proof of Theorem \ref{thm:1}. Therefore we only present those parts of the proof which are different. As in Theorem \ref{thm:1}, we show that $P \left( \mathcal{B}_n^c | \mathcal{A}_n \right) > 1- 2 e^{-n\varepsilon_n^2/4}$ by checking the three conditions of Lemma \ref{lm:1}. 
	
	The proof of Condition (a) is the same as for Theorem \ref{thm:1}, except for the places involving the projection matrix $\Rg$. Observe that given a dataset $\mathcal{D}^n$ and other tuning parameters we fix a particular projection matrix $\Rg$. The only property of $\Rg$ needed to prove condition (a) is $\|\Rg {\bf x} \|^2   \leq m_n p_n $ for sufficiently large $n$. To show this we use that fact that $\Rg$ is a matrix with orthonormal row vectors, and $\Rg\Rg^\prime$ has only one eigenvalue $1$ with algebraic and geometric multiplicity $m_n$. Therefore, $1$ must be an eigenvalue of $\Rg^\prime\Rg$ with algebraic multiplicity at least $m_n$. As the later matrix has only $m_n$ non-zero eigenvalues, this implies that highest eigenvalue of $\Rg^\prime\Rg$ is $1$. Thus, $\|\Rg {\bf x} \| \leq \| {\bf x}_{\bgm} \|  \leq \sqrt{p_n} $.
	
	Therefore the choice of $\epsilon$ required to ensure $d(f_u,f_v)\leq \varepsilon_n$ is  
	
	{\centering
		$\epsilon= \varepsilon_n^2 /\left\{ \sqrt{m_n p_n} \sup_{|h|\leq   c_n \sqrt{m_n p_n}} |a^{\prime}(h)|   \sup_{|h|\leq c_n \sqrt{ m_n p_n} } \left( |b^{\prime}(h)| / |a^{\prime}(h)| \right) \right\}$,
		\par}
	
	\noindent and as before we can show that
	
	\vskip5pt
	{\centering 
		$N(\varepsilon_n,\mathcal{P}_n )    \leq c p_n^{k_n+1} \left( \displaystyle\frac{1}{\varepsilon_n^2}  D( c_n \sqrt{m_n p_n})+1 \right)^{m_n},$
		\par}
	\vskip5pt
	
	\noindent where $D(R)$ is as defined in Theorem \ref{thm:1}. By using the assumptions in Theorem \ref{thm:2} condition (a) follows.
	
	
	The proof of Condition (b) depends only on the prior assigned on $\bta$, and therefore remains the same under the settings of Theorem \ref{thm:2}.
	
	The proof of Condition (c) differs from that of Theorem \ref{thm:1} in showing $ P(| (\Rg {\bf x})^{\prime} \bta - {\bf x}^{\prime} \bbta_0|< \Delta_n ) > \exp\{ -n\varepsilon^2/4 \} $ for some constant $\Delta_n$. To see this consider a positive constant $\Delta_n$. As before, from Lemma \ref{lm:3} we have
	\vspace{-.1 in}
	\begin{eqnarray}
	P(| (\Rg {\bf x})^{\prime} \bta - {\bf x}^{\prime} \bbta_0|< \Delta_n ) 
	& \geq  E_{\bgm} \left[ \exp \left\{ - \displaystyle\frac{ ({\bf x}^{\prime} \bbta_0)^2 +\Delta_n^2 }{\sigma_{\theta}^2 \|\Rg {\bf x}\|^2 } \right\} \displaystyle\frac{2^4 \Delta^4}{\sigma_{\theta}^2 \|\Rg {\bf x}\|^2} \right] \notag \\
	& \geq  E_{\bgm} \left[ \exp \left\{ - \displaystyle\frac{ ({\bf x}^{\prime} \bbta_0)^2 +\Delta_n^2 }{\sigma_{\theta}^2 \alpha_n \| {\bf x}_{\bgm}\|^2 } \right\} \displaystyle\frac{2^4 \Delta^4}{\sigma_{\theta}^2 \| {\bf x}_{\bgm}\|^2} \right] \notag \\
	& = \displaystyle\frac{2^4 \Delta_n^4}{ ({\bf x}^{\prime} \bbta_0)^2 +\Delta_n^2 } E_{\bgm} \left\{ \displaystyle\frac{Z_{\bgm}}{p_n} \exp \left(-\frac{Z_{\bgm}}{\alpha_n p_n} \right)   \right\}, \label{eq_rp19}
	\end{eqnarray}
	where $Z_{\bgm}=\left\{ ({\bf x}^{\prime} \bbta_0)^2 +\Delta_n^2 \right\}/ \left\{ \sigma_{\theta}^2 \| {\bf x}_{\bgm} \|^2 / p_n \right\}$, and $\alpha_n$ is as in (A3). From part (b) of Lemma \ref{lm:2}, and continuous mapping theorem $Z_{\bgm}-z_n \xrightarrow{p} 0$ in $\bgm$ where $z_n= \left\{ ({\bf x}^{\prime} \bbta_0)^2 +\Delta_n^2 \right\}/\left( \sigma_{\theta}^2 c\ad \right)$ $> \Delta_n^2 / \left(\sigma_{\theta}^2 c \ad \right)$.
	For some positive random variable $Z$ and non-random positive numbers $p$, $a$ and $b$, as before we can show that
	\begin{eqnarray}
	E\left(\frac{Z}{p} \exp \left\{ -\frac{Z}{\alpha p} \right\} \right) &\geq& 
	a P\left(\frac{ap}{b} < Z< - \alpha p\log (ab)\right). \label{eq_rp20}
	\end{eqnarray}
	Replacing $Z$ by $Z_{\bgm}$, $p$ by $p_n$, $\alpha$ by $\alpha_n$ and taking $a=\Delta_n^2 \exp \{-n\varepsilon_n^2/3 \}/ (\sigma_{\theta}^2 c \ad)$, and $b=p_n \exp \{-n\varepsilon_n^2/3 \}$ we get
	~$-\alpha_n p_n \log (ab) = -\alpha_n p_n \log \left[ \Delta_n^2 p_n \exp \{-2n\varepsilon_n^2/3 \} / ( \sigma_{\theta}^2 c \ad ) \right]
	\sim 2 p_n \log \left( \Delta_n^2 p_n / ( \sigma_{\theta}^2 c \ad ) \right) /3 >  p_n /2
	$
	~ for sufficiently large $n$ and $ap/b= \Delta_n^2/ \left( \sigma_{\theta}^2 c \ad \right). $
	Therefore the expression in (\ref{eq_rp20}) is greater than 
	
	\vskip5pt
	\hspace{1.75 in}$\displaystyle\frac{\Delta_n^2}{\sigma_{\theta}^2 c \ad} e^{-n\varepsilon_n^2/3 }  P\left(\frac{\Delta_n^2}{\sigma_{\theta}^2 c \ad}  \leq Z_{\bgm} \leq \frac{p_n}{2}  \right).  $
	
	\vskip10pt
	\noindent 	Note that $({\bf x}^{\prime} \bbta_0)^2 <\sum_{j=1}^{p_n} |\beta_{0,j}|<K$, and the probability involved in the above expression can be shown to be bigger than some positive constant $p$ for sufficiently large $n$. Using these facts along with equation (\ref{eq_rp19}), we have $ P(| (\Rg {\bf x})^{\prime} \bta - {\bf x}^{\prime} \bbta_0|< \Delta_n ) > \exp\{ -n\varepsilon_n^2/4 \} $. Choosing $\Delta_n =\varepsilon_n^2/(4M)$ condition (c) follows. 
\end{proof}

\subsection{Proof of Theorem \ref{thm:4}}\label{sm:2.3}
\begin{proof}
	As in the proof of Theorem \ref{thm:3}, we will only prove the condition (c) of Lemma \ref{lm:1}. 
	The proof of condition (c) closely follows that of Theorem \ref{thm:3}. Here also we write $d_{t=1}(f,f_0)=E_{\bf x} \left[ \left\{ (\Rg {\bf x})^{\prime} \bta - {\bf x}^{\prime} \bbta_0 \right\} g\left(u^{*} \right)  \right]$, split it into 2 parts as in equation (\ref{eqn_e4}) of the paper, and argue that in order to prove (c) it is sufficient to show equation (\ref{eqn_e2}) of the paper.

	The first part of equation (\ref{eqn_e2}) is essentially same as the proof of part (c) in Theorem \ref{thm:2}. The only place require attention is the proof of claim that the following expression is no less than $\exp\left\{ -n\varepsilon_n^2/4 + 2 \log 2\right\}$,
	\begin{equation}
	\displaystyle\frac{\Delta_n^2}{\sigma_{\theta}^2 c \ad} e^{-n\varepsilon_n^2/3 }  P\left(\frac{\Delta_n^2}{\sigma_{\theta}^2 c \ad}  \leq Z_{\bgm} \leq -\alpha_n p_n \log \left( \frac{\Delta_n^2 p_n}{ \sigma_{\theta}^2 c \ad} e^{-2n\varepsilon_n^2/3 }    \right) \right).\label{eqn_e6}
	\end{equation}
	where $Z_{\bgm}-z_n \rightarrow 0$ where $ z_n= \left\{ ({\bf x}^{\prime} \bbta_0)^2 +\Delta_n^2 \right\}/\left( \sigma_{\theta}^2 c\ad \right)$ in probability in $\bgm$. The right hand side within the above probability is bigger than $\alpha_n p_n n \varepsilon_n^2/2 \geq p_n \left(n \varepsilon_n^2\right)^b/2$ for some $b>0$ by assumption (A3$^\prime$). 
	Note further that $({\bf x}^{\prime} \bbta_0)^2/\left\{p_n (n\varepsilon_n^2)^b\right\} <\|\bfx\|^2 \|\bbta_0\|^2/\left\{p_n (n\varepsilon_n^2)^b\right\} \rightarrow 0$ almost surely in $\bfx$ by Lemma \ref{lm:5}(b). Therefore $\Delta_n^2/(\sigma_{\theta}^2 c\ad)<z_n< p_n \left(n \varepsilon_n^2\right)^b/2$, almost surely in $\bfx$, 	
	and the probability involved in (\ref{eqn_e6}) is bigger than some positive constant $p$ for sufficiently large $n$. Using these facts and choosing $\Delta_n$ as in the proof of Theorem \ref{thm:2}, we can show that the expression in (\ref{eqn_e6}) is bigger than $\exp\left\{ -n\varepsilon_n^2/4 + 2 \log 2\right\}$.  This completes the first part of (\ref{eqn_e2}).
	
	We prove the second part of (\ref{eqn_e2}) in the same manner as in Theorem \ref{thm:3}.
	Consider the same set $D_n$, such that $\pi\left(\left. (\bta,\bgm) \in D_n \right| \mathcal{A}_n\right) \geq \exp\{-n\varepsilon_n^2/4 +2\log 2 \}$, i.e., we consider any $\bgm\in \cup_l \mathcal{M}_l$ (see assumption (A2)) and any $\bta: \|\bta\|\leq \sigma_\theta \sqrt{3n} \varepsilon_n/\sqrt{2}$. 
	
	Next note that, the quantity  $\left\{(\Rg \bfx)^{\prime}\bta-\bfx^\prime\bbta_0\right\}\leq \left\|\Rg \bfx \right\|\|\bta\|+|\bfx^\prime\bbta_0| \leq \left(\| \bta\|+K \right)\|\bfx\|$, as $\Rg$ is row-orthogonal. Therefore, we consider 
	
	{\centering
		$\max\left\{\|\Rg \bfx \| \|\bta\|, \left| \bfx^{\prime} \bbta_0 \right|, \left|(\Rg \bfx)^{\prime}\bta-\bfx^\prime\bbta_0\right| \right\}\leq c\sqrt{n\varepsilon_n^2}  \|\bfx\|$, \par}
	
	\noindent for a suitable constant $c>0$.
	Further by assumption (B3), $|g(u)|\leq \exp\{c_0 u\}$ for some fixed $c_0>0$. Thus,
	
	{\centering
		$ E_\bfx \left[ \left. \left\{(\Rg \bfx)^{\prime}\bta-\bfx^\prime\bbta_0\right\} g\left(u^{*}\right) \right| \| \bfx \|> \sqrt{3} p_n \right] \leq E_\bfx  \left[ \left. \exp \left\{c\sqrt{n\varepsilon_n^2} \|\bfx\| \right\}  \right| \| \bfx \|> \sqrt{3}p_n \right]$,\par}
	
	\noindent for a suitable constant $c>0$. Finally as in Theorem \ref{thm:3} we observe that
	\begin{eqnarray*}
		E_\bfx  \left[ \left. \exp \left\{c\sqrt{n\varepsilon_n^2} \|\bfx\| \right\}  \right| \| \bfx \|> \sqrt{3}p_n \right]
		\hspace{3.5 in}	\\
		\leq \exp\left[ -\frac{3 c_1 p_n^2}{c_0 l_n}\left\{\left(1- \frac{c~l_n\sqrt{n\varepsilon_n^2}}{\sqrt{3}p_n} \right)^2+\frac{c~l_n^2 n\varepsilon_n^2 }{p_n^2}+ \frac{cl_n}{p_n}\log\left(\frac{c_0l_n}{c_2}\right)+\frac{c}{p_n}\log(p_n)\right\}  \right]\\
		\hspace{2 in}\times P\left(\left\| \bfx\right\|>\sqrt{3}p_n | \bfx\sim N({\bf 0},c_0c_2^{-1}l_n \right).
	\end{eqnarray*}
	Noting that $\max\left\{ l_n\log(l_n),l_n\sqrt{n\varepsilon_n^2} m_n\right\}=o(p_n)$ we have
	
	{\centering 
		$E_{\bf x} \left[ \left\{ (\Rg {\bf x})^{\prime} \bta - {\bf x}^{\prime} \bbta_0 \right\} g\left(u^{*} \right)|A_{p_n}^c  \right] P\left(A_{p_n}^c\right) \leq \exp\{-cp_n\}\leq \varepsilon_n^2/20.$
		\par}
	
	\noindent	This proves the second part of (\ref{eqn_e2}), and the following the same procedure as in Theorem \ref{thm:3} the proof is completed. 
\end{proof}

\section{An Addition to Section \ref{sec:4} of the Paper}\label{sm:3}

\subsection{Details of specifications of the tuning parameters of the competitors}\label{sm:3.3}
SCAD and MCP are calculated using three packages, viz., \emph{SIS}, \emph{ncvreg} and \emph{ncpen} \citep{ncpen}. For one-step SCAD (1-SCAD) we use a R-code provided by the authors of \cite{FXZ2014}. LASSO, ridge and elastic net (EN) are calculated using two packages, \emph{glmnet} and \emph{biglasso} \citep{biglasso}. In each case, the best among all the results is provided. As we are interested in prediction problem, AIC tuning parameter selector is chosen (see \cite{ZLT2010}).  For EN the parameter $\alpha$ is set to $0.5$. The tuning parameter $\lambda$ should converge to $0$ at a certain rate. \cite{FXZ2014} has chosen $\lambda=\sqrt{p_n/n}$ for practical purpose. We consider $200$ equidistant points in the range $[0.0005,2]$ for $\lambda$. Additionally,  we have considered the default data-adaptive range of $\lambda$ provided in the respective packages for all the penalization methods. The best result (calibrated in terms of average MSPE) among these two is provided.

SPCR and RPCR are performed using \emph{PMA} and \emph{rsvd} packages in R, respectively. To estimate PC scores, we rely on approximate SVD using \emph{fast.svd} in the \emph{corpcor} package. For BCR, we average over $100$ different random projection matrices with varying $m_n$ values within the range $[2\log p_n,3n/4]$. We use the \emph{qr} function in R to apply QR factorization in place of Gram-Schmidt orthogonalization of the random matrix, which is computationally prohibitive for large $p_n$. For BASAD we use \emph{basad} R-package with the default settings which includes choosing the initial estimate for number of active covariates by Gibbs sampling, and number of burn-in and estimation iterations as $1000$ each. For SSLASSO we use \emph{SSLASSO} R-package with $\lambda_1=1$ and $\lambda_0$ chosen from the interval $[1,100]$ with grid increments of 1.

\paragraph{Methods used to find prediction intervals (PIs) of the competing methods}
For TARP and BCR, PIs are obtained from quantiles of the posterior predictive distribution of ${\bf y}_{new}$ given $\mathcal{D}^n, X_{new}$. 
For normal-linear model the $100(1-\alpha)\%$ PI of the frequentist methods can be obtained as
$$\hat{y}_{new}\pm t_{\alpha/2,n-k-1}\sqrt{\hat{\sigma}^2\left(1+{\bf x}_{new}^{\prime}(X_{[k]}^{\prime}X_{[k]})^{-1}{\bf x}_{new} \right)},  $$
where $\hat{\sigma}^2=\sum_{i=1}^{n} (y_i-\hat{y}_i)^2/(n-k-1)$, $k$ is the number of predictors used to predict $y$, $X_{[k]}$ is the design matrix with $k$ selected predictors. 
We used this formula for all other methods except the penalized likelihood methods. The penalization methods are subject to over-fitting, and consequently the resulting $\hat{\sigma}$ can be very close to zero. Some possible solutions to this problem are provided in \cite{conformal2018}, \cite{leave_one_out_PI}. The former takes an approach of conformal prediction, while the later takes a simple approach of estimating PI based on leave-one-out residuals. It considers the inter-quartile range of the leave-one-out residuals as the
$50\%$ PI. We consider the later approach as calculation of the conformal prediction interval (provided in \emph{conformalInference} R-package) takes prohibitive time for higher values of $p_n$.

\subsection{Additional simulation results for higher values of $p_n$}\label{sm:3.4}
Here we present the performance of the different methods with respect to mean square prediction error (MSPE), and empirical coverage probability (ECP) and the width of $50\%$ prediction interval (PI) for larger choices of $p_n$ in each scheme considered in Section \ref{sec:4} of the paper.
The relative performance of the methods in terms of MSPE for different simulation schemes are shown in Figures \ref{fig-s1}-\ref{fig-s6}, and that in terms of ECP and width of $50\%$ PI are presented in Table \ref{tab-s2}.

\afterpage{
	\begin{figure}
		\centering     
		\subfigure[{\it Scheme I}, $p_n=3\times 10^3$.]{\label{fig-s1}  \includegraphics[height=2.5 in, width=3 in]{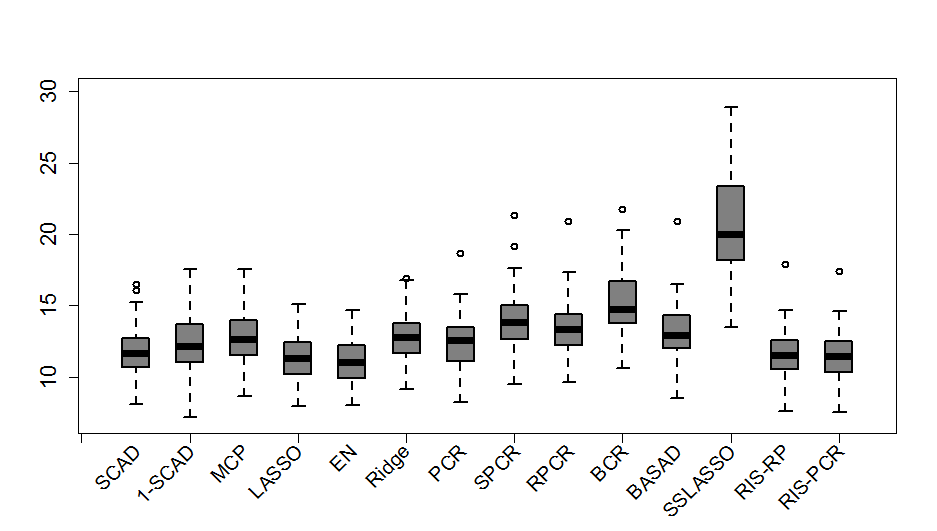}}
		\subfigure[{\it Scheme II}, $p_n=2\times 10^4$.]{\label{fig-s2}  \includegraphics[height=2.5 in, width=3.1 in]{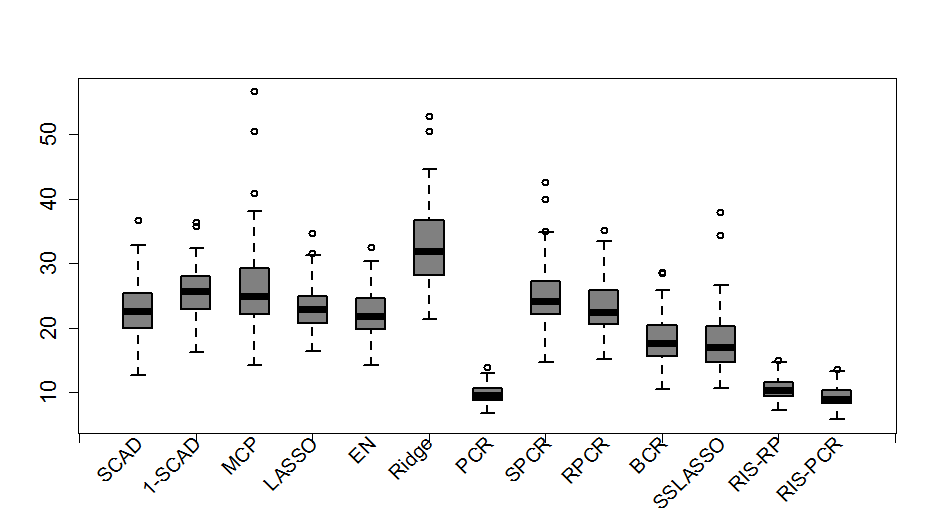}}
		\caption{Box-plot of MSPEs for the competing methods in Schemes I and II. } 
	\end{figure}

	\begin{figure}
		\centering     
		\subfigure[MSPE of all the methods.]{\label{fig-s3}  \includegraphics[height=2 in, width=3.1 in]{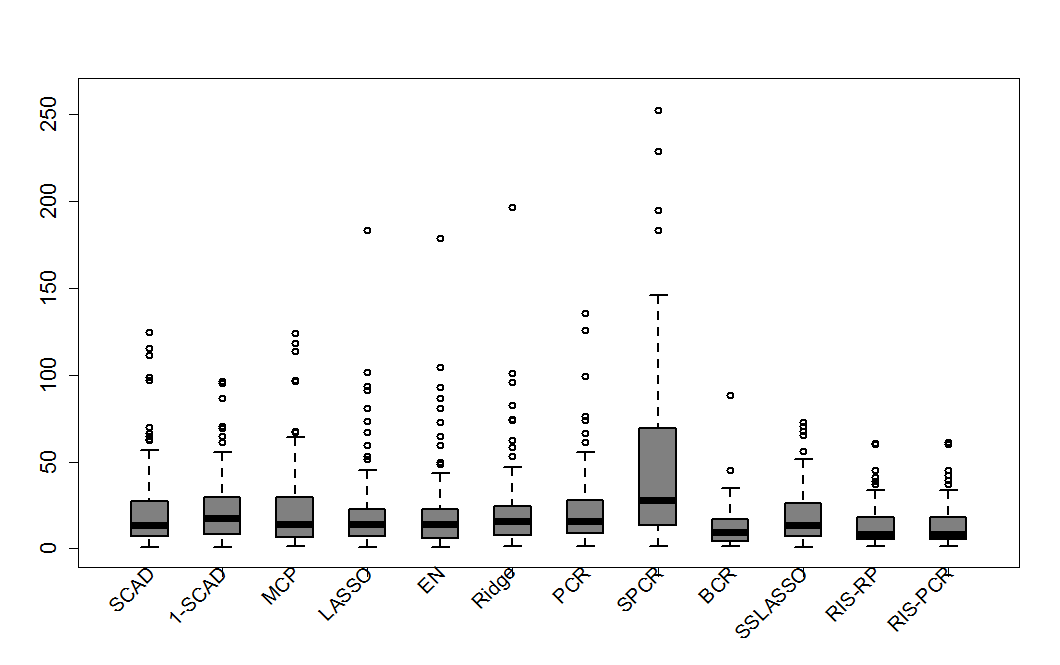}}
		\subfigure[MSPE of selected methods.]{\label{fig-s4}  \includegraphics[height=2 in, width=3.1 in]{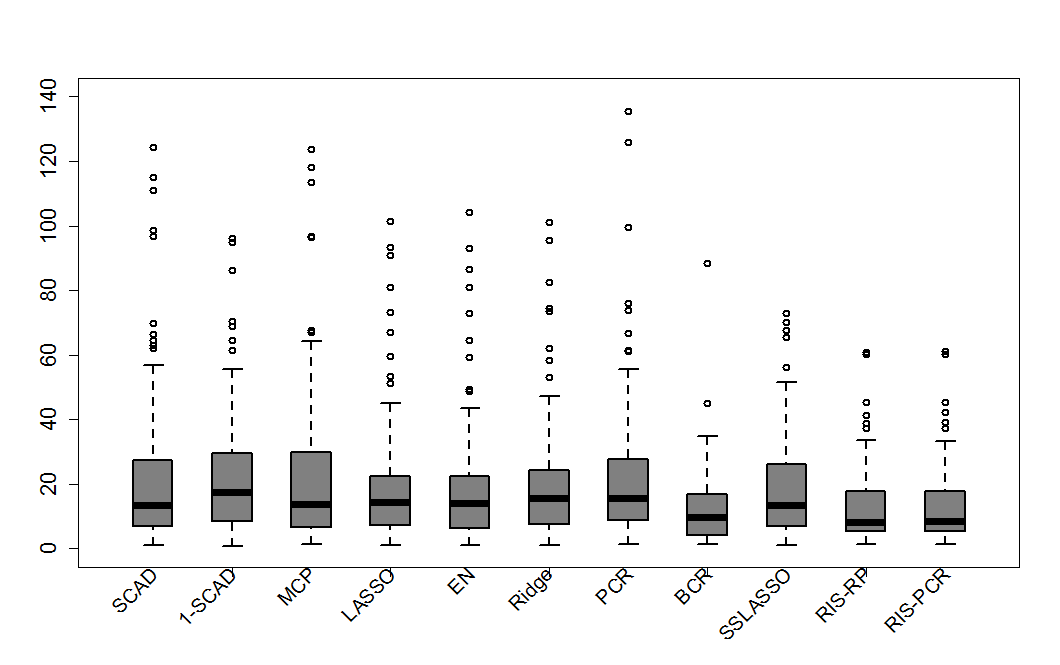}}
		\caption{Box-plot of MSPEs for $p_n=10^4$ in \emph{Scheme III}. } 
	\end{figure}
	
	\begin{figure}
		\centering     
		\subfigure[MSPE of all the methods.]{\label{fig-s5}  \includegraphics[height=2.5 in, width=3.1 in]{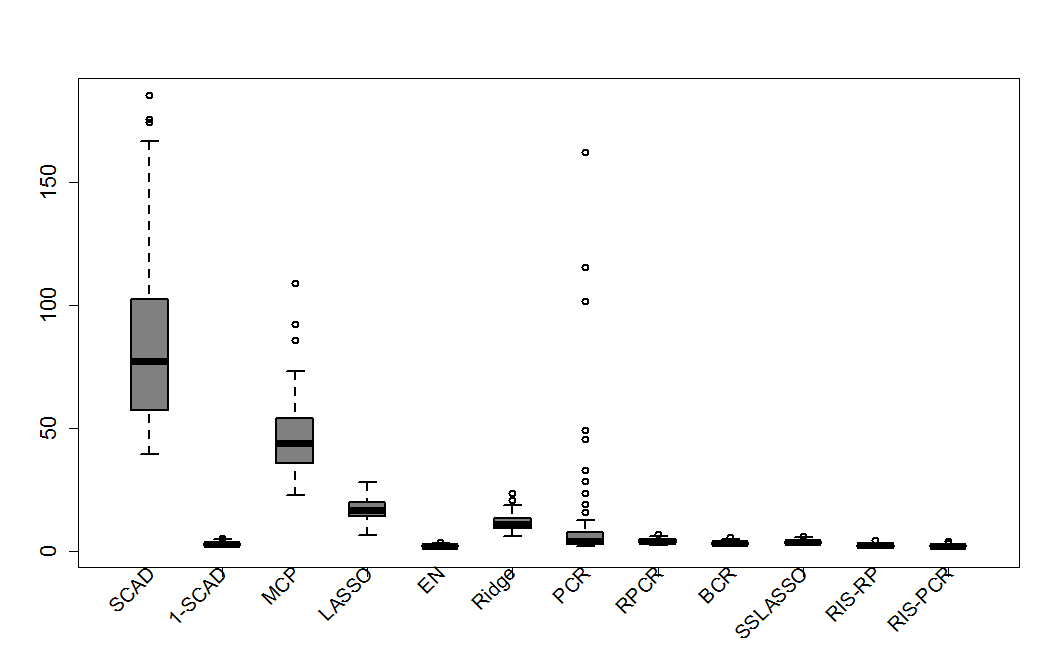}}
		\subfigure[MSPE of selected methods.]{\label{fig-s6}  \includegraphics[height=2.5 in, width=3.1 in]{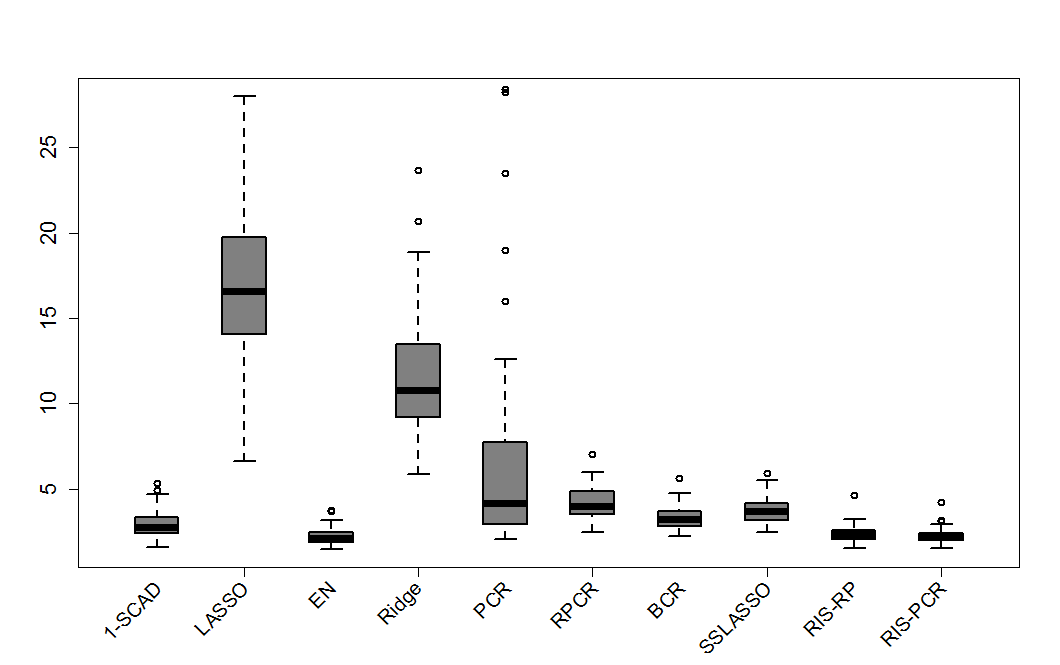}}
		\caption{Box-plot of MSPEs for $p_n=2\times 10^4$ in \emph{Scheme IV}. } 
\end{figure}}

\afterpage{
	\begin{table*}[!h]
		\renewcommand{\arraystretch}{1.2}
		\begin{center}
			{\scriptsize
				\begin{threeparttable}
					\caption{Mean and standard deviation (in bracket) of empirical coverage probabilities (ECP) and width of 50\% prediction intervals for 13 competing methods.}
					\label{tab-s2}
					\begin{tabular}{|p{1.5 cm}|p{1.4 cm} p{1.4 cm} | p{1.4 cm} p{1.4 cm} | p{1.4 cm} p{1.4 cm}| p{1.4 cm} p{1.7 cm} |}  \hline
						Schemes $\rightarrow$  &  \multicolumn{2}{|c|}{\it I} & \multicolumn{2}{|c|}{\it II} & \multicolumn{2}{|c|}{\it III} & \multicolumn{2}{|c|}{\it IV} \\ 
						$(n,p_n)$ &  \multicolumn{2}{|c|}{$(200,2\times 10^3)$} & \multicolumn{2}{|c|}{$(200,10^4)$} &  \multicolumn{2}{|c|}{$(200,5\times 10^3)$} &  \multicolumn{2}{|c|}{$(200,10^4)$}\\ 
						Methods $\downarrow$ &  ECP  &  Width &  ECP & Width &  ECP  & Width &  ECP & Width  \\ \hline
						{\it SCAD}    & 49.8   (6.0) & 4.58  (0.42) & 50.2  (8.3) & 6.51  (0.88) & 11.3   (10.1) & 1.37   (0.08) & 48.1   (6.8) & 11.74    (2.30) \\
						{\it 1-SCAD} & 49.4 (11.0) & 4.77 (1.07) & 46.3 (9.6) & 6.29 (1.36) & 10.8 (10.6) & 1.42 (0.21) & 49.2 (9.6) & 2.23 (0.42) \\ 
						{\it MCP}     & 51.7   (6.7) & 4.92  (0.62) & 49.1   (9.9) & 6.84  (1.22) & 11.7   (10.8) & 1.39   (0.19) & 50.9   (6.7) & 9.2  ~~ (1.24) \\
						{\it LASSO}   & 47.9   (6.9) & 4.80  (0.50) & 47.4   (8.5) & 6.36  (1.05) &  10.6   (11.0) & 1.38   (0.12) & 47.3   (6.7) & 5.09    ~(0.60) \\
						{\it EN}      & 48.9   (6.5) & 4.47  (0.39) & 48.0   (9.0) & 6.31  (0.95) &  10.4   (11.0) & 1.38   (0.12) & 58.2   (7.8) & 2.40    ~(0.33) \\
						{\it Ridge}   & 47.2   (6.2) & 4.65  (0.42) & 48.6   (7.3) & 7.55  (0.92) &  10.3   (9.9)  & 1.38   (0.12) & 52.8   (7.5) & 4.89   ~(0.66) \\
						{\it PCR}     & 41.3   (28.1)& 4.45  (3.70) & 37.8   (24.5)& 3.34  (2.41) & 11.0   (9.5)  & 1.37   (0.08) & 51.7   (32.1)& 4.45   ~~(3.59) \\
						{\it SPCR}    & 50.4   (5.3) & 5.03  (0.24) & 49.5   (5.4) & 6.58  (0.50) &  9.4   (7.5)  & 1.38   (0.08) & 49.8   (5.5) & 49.05 (13.68) \\
						{\it RPCR}    & 50.2   (5.5) & 4.91  (0.22) & 47.2   (5.2) & 5.98  (0.41) & \tnote{**} &     **           & 49.0   (5.3) & 2.67   ~~(0.28) \\
						{\it BCR}     & 52.3   (5.5) & 5.45  (0.26) & 54.6   (6.1) & 6.25  (0.41) & 26.5   (13.2) & 2.42   (0.69) & 44.8   (5.4) & 2.16   ~~(0.14) \\
						{\it BASAD}   & 50.3   (5.4) & 4.95  (0.22) &  \tnote{*}   &    *         & *    & * &      *       & * \\
						{\it SSLASSO} & 18.7   (3.9) & 2.15  (0.12) & 22.0   (4.4) & 2.28  (0.13) & 11.1   (9.3)  & 1.38   (0.09) & 33.3   (4.0) & 1.64   ~~(0.10) \\
						{\it RIS-RP}  & 39.3   (4.9) & 3.53  (0.15) & 47.0   (5.4) & 4.03  (0.12) & 67.6   (18.6) & 8.16   (4.60) & 64.5   (5.1) & 2.84   ~~(0.18) \\
						{\it RIS-PCR} & 28.4   (5.0) & 2.50  (0.22) & 29.4   (4.8) & 2.26  (0.16) & 21.9   (8.7)  & 2.22   (0.85) & 39.6   (4.5) & 1.55   ~~(0.12) \\
						\hline
					\end{tabular}
					\begin{tablenotes}
						\item [*] BASAD requires prohibitive computational time for $p_n=10^4$, and hence removed from comparison
						\item [**] RPCR produces extremely high MSPE and width of PI for {\it Scheme III}, and hence removed from comparison
					\end{tablenotes} 
			\end{threeparttable}}
		\end{center}
	\end{table*}
}

\subsubsection{Summary of comparative performance of the competing methods in schemes I-IV.}\label{sm:3.5}  All the methods except SSLASSO yield reasonable performance in terms of MSPE in {\it Scheme I} (see Figures \ref{fig1} and \ref{fig-s1}). Further SSLASSO has the lowest coverage among all the methods in {\it Scheme I}. Among the other methods, the best overall performance in terms of MSPE are by EN, RIS-PCR and RIS-RP. However, TARP has lower empirical coverage than most of the methods. All the other methods have nearly comparable performance in terms of MSPE and ECP.
Finally PCR has average ECP and width comparable to RIS-RP. However, it has the highest variance of both ECP and width reflecting lack of stability. 1-SCAD also have high variance of ECP and width in {\it Scheme I}.

In {\it Scheme II} the best overall performance is by PCR and RIS-PCR (see Figures \ref{fig2} and \ref{fig-s2}). However, RIS-PCR has lower average ECP than most of the methods, and PCR also has low coverage and highest variance of ECP compared to all the methods. RIS-RP closely follows the former two methods in terms of MSPE, and also yield higher coverage probability. 
SSLASSO and BCR have somewhat better results than others in terms of MSPE. However, like in {\it Scheme I}, here also SSLASSO has the lowest ECP among all the competing methods. BCR has the highest coverage among all the competing methods. Average MSPE of all the penalization methods, SPCR and RPCR are on the higher side, and among these methods Ridge and MCP have the worst performance in terms of MSPE. MCP and Ridge fail to perform well in terms of MSPE in \emph{Scheme II}. 
Finally, we skip BASAD as it requires prohibitive computational time for $p_n\sim 10^4$.  

In {\it Scheme III} (see Figures \ref{fig3}-\ref{fig4} and \ref{fig-s3}-\ref{fig-s4}), RPCR has unrealistically high MSPE, therefore we skip it from comparison. BASAD shows the worst performance in terms of MSPE among the other methods. SPCR also has larger, occasionally extremely high,  MSPE values compared to other methods. Among the other methods the box-plot of MSPE is most stable for $3$ methods, viz., RIS-RP, RIS-PCR and BCR. All the other methods show similar performance in terms of MSPE. All of them have some large outlying MSPE values indicating lack of robustness in the presence of outliers. In terms of ECP as well, all the methods except RIS-RP have low coverage probabilities. BCR and RIS-PCR also have better ECP and width than others. All other methods have comparable average ECP and width, although BASAD has the highest width of PI among all.

In {\it Scheme IV} (see Figures \ref{fig5}-\ref{fig6} and \ref{fig-s5}-\ref{fig-s6}), SPCR has the worst overall performance. SCAD and MCP also show poor performance in terms of MSPE compared to others. Among the other methods LASSO has the worst performance in terms of MSPE, followed by Ridge and PCR. MSPE of PCR frequently becomes extremely large indicating instability. The other methods have comparable performance in terms of MSPE, although TARP and EN have the best MSPE results. In terms of ECP, all the methods except SSLASSO has comparable ECP. However, RIS-RP has highest average ECP, and RIS-PCR has average ECP nearly 40\%. As in {\it Scheme I}, PCR has highest variance of ECP in {\it Scheme IV} as well, followed by 1-SCAD. Except SPCR, SCAD and MCP, all the other methods have reasonable average width of 50\% PI.


\subsection{Summary of performance of the methods with respect to computational time}\label{sm:3.6}
When $p_n$ is below $10^4$ all the methods except BASAD require comparable computational time (see Figures \ref{fig11} and \ref{fig12}). BASAD takes more than an hour for $p_n=10^4$ and the code crashes  $p_n\geq 5\times 10^4$. 
Computation of SSLASSO also takes prohibitive time when $p_n=5\times 10^5$, and it is much higher than all other methods (except BASAD) for $p_n\geq 5\times 10^4$. 1-SCAD takes much longer time among the other methods. The computational time is more than 15 minutes for $p_n=10^5$, and more than an hour for $p_n=5\times 10^5$ for 1-SCAD. SCAD, BCR, RIS-PCR and EN require comparable computational time up to $p_n=10^5$. For $p_n=5\times 10^5$, EN requires highest time, followed by RIS-PCR. SPCR, LASSO and RIS-RP require comparable time throughout, RIS-RP requiring maximum time and LASSO requiring minimum time to compute among these three. The remaining 4 methods, MCP, Ridge, PCR and RPCR have best overall performance. Among these methods, Ridge takes highest time to compute, less than 1.5 minutes for $5\times 10^5$.  

\begin{figure}
	\centering     
	\includegraphics[scale=.5]{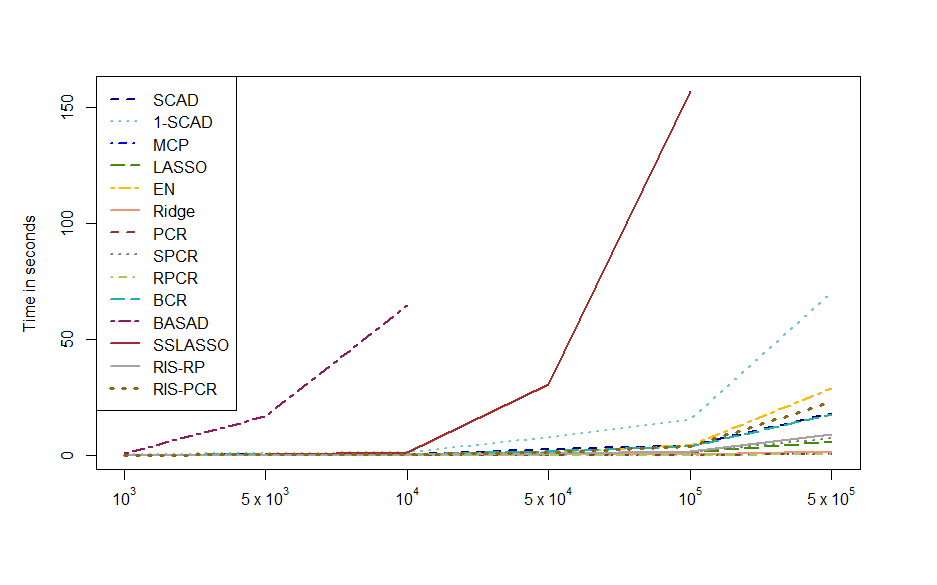}
	\caption{Time required by different methods to predict $y$ as $p_n$ grows. }
	\label{fig12}
\end{figure}
The increment of computational time of RIS-PCR is due to the computation of exact SVD of the screened design matrix $X_{\bgm}$. However, this can be reduced if one uses some approximation of the SVD.

\section{An Addition to the Section \ref{sec:5} of the Paper}\label{sm:4}
\subsection{Analysis of GTEx and Eye datasets}\label{sm:4.0}
\paragraph{GTex Dataset} To understand the functional consequences of genetic variation, \cite{gtex} presented an analysis of RNA sequencing data from 1641 samples across 43 tissues from 175 individuals, generated as part of the pilot phase of the Genotype-Tissue Expression (GTEx) project.
We selected RNA-seq data on two normal tissues, viz., Artery-Aorta and Artery-Tibial. The dataset contains RNA-seq expressions on 36115 $(=p_n)$ genes and 556 $(=n)$ samples, among which 224 are from Artery-Aorta, and 332 are from Artery-Tibial. A training set of $100$ samples from each of the tissue types is considered, and the remaining $446$ samples are used as test set. Table \ref{tab:sm2} provides the average and standard deviation (sd) of percentages of misclassification, and those for the area under the ROC curve over 100 random subsets of the same size for the competing methods.
\afterpage{
	\begin{table*}[!h]
		\renewcommand{\arraystretch}{1.5}
		\begin{center}
			{\scriptsize
				\caption{Mean and standard deviation (in braces) of percentage of misclassification and area under the ROC curve for \emph{GTEx} dataset.}
				\label{tab:sm2}
				\begin{tabular}{|p{2 cm}|p{.7 cm} p{.7 cm}   p{.7 cm}  p{0.7 cm} p{0.7 cm} p{0.7 cm} p{0.7 cm} p{0.7 cm} p{0.7 cm} p{0.7 cm} p{0.7 cm} p{0.7 cm}|}  \hline
					Methods & \quad SCAD & 1-SCAD & \quad MCP & \quad LASSO & \quad EN & \quad Ridge & \quad PCR &\qquad SPCR & \qquad RPCR & \quad BCR &  RIS-RP & RIS-PCR  \\ \hline
					Misclassification Rate (in \%)  
					& 0.00 (0.00) & 34.83 (0.00)  & 0.00 (0.00) & 58.81 (22.39) & 00.42 (0.10) & 0.32 (0.19) & 0.07 (0.14) & 3.26 (3.03) & 0.16 (0.19) & 13.59 (2.00) &  0.42 (0.18) & 0.50 (0.32)  \\
					\hline
					Area under ROC curve  
					& 1.000 (.000) & 0.500 (.000) &  1.000 (.000) & 0.570 (.168) & 0.999 (.001) & 0.998 (.001) & 0.999 (.001) & 0.966 (.033) & 0.999 (.001) & 0.877 (.041)  & 1.000 (.000) & 0.996 (.002)  \\ \hline
			\end{tabular}}
		\end{center}
\end{table*}}

\paragraph{Eye Dataset}
The Eye dataset consists of gene expressions for $200$ ($=p_n$) gene probes from the microarray experiments of mammalian-eye tissue samples of 120 ($=n$) rats (see \cite{eyedata}). The response variable is the expression level of the TRIM32 gene. 
We consider $100$ sample points as the training set, and the remaining $20$ samples as the test set.
The eye dataset has continuous response, and therefore we evaluate the methods by MSPE and empirical coverage probabilities (ECP) of $50\%$ prediction intervals (PI) as in Section \ref{sec:4}.  As variation in the expression levels of the TRIM32 gene is very small (the range is 1.37), we multiply the MSPEs of different methods by $10$ to increase the variability. Table \ref{tab8} provides the mean and sd of MSPEs, ECPs of $50\%$ PIs, and widths of the PIs over 100 different training and test sets selected from the dataset, for the competing methods.

\afterpage{
	\begin{table*}[!h]
		\renewcommand{\arraystretch}{1.5}
		\begin{center}
			{\scriptsize {\centering
					\caption{Mean and standard deviation (in braces) of mean square prediction errors, empirical coverage probabilities and widths of $50\%$ prediction interval of 13 competing methods in Eye dataset.}
					\label{tab8}
					\begin{tabular}{|p{.9 cm}|p{.65 cm} p{.65 cm}  p{.65 cm}  p{0.65 cm} p{0.65 cm} p{0.65 cm} p{0.65 cm} p{0.65 cm} p{0.65 cm} p{0.65 cm} p{0.7 cm} p{0.65 cm} p{0.65 cm} p{0.65 cm}|}  \hline
						Methods & \quad SCAD & 1-SCAD & \quad MCP & \quad LASSO & \quad EN & \quad Ridge & \quad PCR &\qquad SPCR & \qquad RPCR & \quad BCR &\quad BASAD &  SS-LASSO & RIS-RP & RIS-PCR  \\ \hline
						MSPE  & 11.66 (4.07) & 18.76 (17.15) & 11.66 (4.06) & 8.89 (3.37) & 8.90 (3.12) & 10.21 (2.70) & 13.84 (3.94) & 8.65 (3.08) & 7.67 (3.30) & 10.01 (4.04) & 68.09 (26.42) & 11.52 (7.81) & 8.54 (3.09) &  8.29 (2.99)  \\\hline
						ECP  & 0.502 (.138) & 0.459 (.132) & 0.502 (.138) & 0.574 (.144) & 0.573 (.122) & 0.537 (.115) &  0.423 (.325) & 0.508 (.123) & 0.522 (.114) & 0.564 (.117) & 0.541 (.148) & 0.467 (.143) &  0.598 (.101) & 0.507 (.107)  \\\hline
						Width & 1.208 (.057) &  1.340 (0.018)  & 1.208 (.057) & 3.05 (11.524) & 1.41 (.189) & 1.489 (.184) & 1.884 (1.61) & 1.202 (.079) & 1.055 (.049)  & 1.249 (.056) & 1.600 (.288) & 1.069 (.042) & 1.341 (.038) &  1.056 (.036) \\\hline
					\end{tabular}\par} }
		\end{center}
\end{table*}}


\paragraph{Results:}
For the Golub data set, both the lowest misclassification rate and the highest area under ROC curve are achieved by RIS-RP, which is closely followed by RIS-PCR. TARP based methods attain lower sd than other methods as well. PCR and Ridge also yield reasonable performance with average misclassification rate lower than $9\%$ and area under ROC more than $0.9$. RPCR, LASSO, EN, SCAD and MCP produce average misclassification rates of at least $10\%$,  with area under the ROC about $0.9$. BCR possesses high misclassification rate (about $19\%$), although area under ROC is more than $0.8$. Finally, either the MSPE, nor the area under ROC curve, is satisfactory for SPCR and 1-SCAD.  

For the GTEx dataset, perfect classification is achieved by SCAD and MCP. These methods along with RIS-RP also have the highest area under the ROC curve. PCR, RPCR, EN, Ridge, RIS-RP and RIS-PCR also yield satisfactory results, having less than $0.5\%$ average misclassification rate and more than $99\%$ area under the ROC curve.  SPCR yield reasonable performance with an average MSPE of less than $4\%$. BCR attains $13.3\%$ average misclassification rate, with the area under the ROC curve almost $0.9$.  Finally LASSO and 1-SCAD fail to show any discriminatory power with average MSPE more than $34\%$.  

RPCR, RIS-PCR, RIS-RP, SPCR, LASSO and EN yield excellent performance in terms of MSPE in the eye data with an average MSPE of less than 1 (see Table \ref{tab8}). All of these methods show stable performance in terms of ECP. However, LASSO has much higher width than all other methods, with exceptionally large variance. BCR, Ridge, SSLASSO, SCAD and MCP have similar overall performance. In terms of MSPE, BCR outperforms the other three methods. PCR and 1-SCAD are not quite as good in terms of either measures. 1-SCAD also show high variability in MSPE results. 
Finally, performance of BASAD is worst in terms of MSPE, although it yields comparable results with respect to the other measures.   

\subsection{Predictive calibration in binary response datasets}\label{sm:4.2}
Apart from measuring the misclassification rates and the area under ROC curve, we validate TARP in terms of it's ability to quantify uncertainly in real datasets with binary responses. To this end, we partition the interval $[0,1]$ into ten equal sub-intervals, viz., $[0,0.1), ~[0.1,0.2)$ and so on, and classify the test data points $({\bf x},y)_{i,new}$ to the $k^{th}$ class if predictive probability of $y_{i,new}$ falls in that class. Next, we consider the squared difference of the empirical proportion of $y_{i,new}=1$ among the data points classified in a given interval with the middle point of the interval, and consider the mean of these squared differences (\emph{MSD}) of all the intervals. If a method is well calibrated, then the \emph{MSD} would be small. The following table (Table \ref{tab-s7}) shows means and standard deviations of \emph{MSD}s of the competing methods for Golub and GTEx datasets.

\afterpage{
	\begin{table*}[!h]
		\renewcommand{\arraystretch}{1.5}
		\begin{center}
			{\scriptsize
				\caption{Mean and standard deviation (in braces) of mean square differences (\emph{MSD}) of empirical and predictive probabilities.}
				\label{tab-s7}
				\begin{tabular}{|p{1.3 cm}|p{.7 cm} p{.7 cm} p{.7 cm}  p{0.7 cm} p{0.7 cm} p{0.7 cm} p{0.7 cm} p{0.7 cm} p{0.7 cm} p{0.7 cm} p{0.7 cm} p{0.7 cm}|}  \hline
					Methods$\rightarrow$ \quad Dataset $\downarrow$ & \quad SACD & 1-SCAD & \quad MCP & \quad LASSO & \quad EN & \quad Ridge & \quad PCR &\qquad SPCR & \qquad RPCR & \quad BCR &  RIS-RP & RIS-PCR  \\ \hline
					Golub  & 4.454 (.000) &  4.454 (.000) &   4.469 (.032) & 4.454 (.000)  & 4.454 (.000) & 4.454 (.000) & 2.589 (.012) & 3.429 (.073)  & 2.587 (.017) & 2.886 (.032)  &  2.611 (.045) & 2.555 (.044) \\
					GTEx & 2.784 (.000) & 5.000 (.000) &  2.784 (.000) & 4.652 (.000) & 4.652 (.000) & 4.652 (.000) & 2.784 (.000) & 3.216 (.102) & 2.784 (.000) & 2.873 (.007) &  2.782 (.001) & 2.783 (.001)  \\
					
					\hline
				\end{tabular}
			}
		\end{center}
	\end{table*}
}

Table \ref{tab-s7} indicates that
TARP based methods, PCR, RPCR and BCR perform relatively well compared to the others in both the datasets. Among these methods, RIS-PCR and RIS-RP have lowest MSD for Golub and GTEx data, respectively. SCAD and MCP has lower MSD in GTEx dataset, but they fail to perform well in Golub dataset. SPCR is deficient in terms of MSD for both the datasets. Finally LASSO, Ridge, EN and 1-SCAD have worst performance in terms of MSD among all the methods.


\subsection{ The GEUVADIS cis-eQTL dataset}\label{sm:4.3} ~ We conclude this section by illustrating the TARP approach on a massive dataset.
The GEUVADIS cis-eQTL dataset \citep{massive_p} is publicly available at \url{http://www.ebi.ac.uk/Tools/geuvadis-das/}. This dataset
consists of messenger RNA and microRNA on lymphoblastoid cell line (LCL) samples from 462 individuals provided by the $1000$ Genomes Project along with roughly $38$ million SNPs. E2F2 plays a key role in the control of the cell cycle. Hence, as in \cite{chen2017}, we choose the gene E2F2 (Ensemble ID: ENSG00000000003) as the response. A total of $8.2$ million $(=p_n)$ SNPs are preselected as candidate predictors on the basis of having at least $30$ non-zero expressions. The total number of subjects included in the dataset is about $450$ $(=n)$. The genotype of each SNP is coded as $0$, $1$ or $2$ corresponding to the number of copies of the minor allele.

TARP is applied on this dataset. We consider four different training sample sizes, viz., $n_t=200, 250$, $300$ and $350$, and test sample size $100$ in each case. As $p_n$ is huge, we applied three different values of $\delta$, namely, $2,5$ and $8$, to analyze the effect of a conservative screening. The recommended choice of $\delta$ lies within $(5,6)$ when $p_n=8.2\times 10^6$ and $n\in[200,400]$. To perform SVD for RIS-PCR, we use \emph{fast.svd} instead of the usual \emph{svd} to cope with the massive number of regressors. Table \ref{tab3} provides the MSPE, the ECP of $50\%$ PI and width of the PI, obtained by two different variants of TARP. 
\afterpage{
	\begin{table*}[!h]
		\renewcommand{\arraystretch}{1.2}
		{\scriptsize
			\caption{MSPE, ECP and width of PI (in order) obtained by RIS-RP and RIS-PCR for three values of $\delta$ and different training sample sizes $(n_t)$.}
			\label{tab3}
			\begin{center}
				\begin{tabular}{|p{1 cm}|p{2.7 cm}  p{2.7 cm}  p{2.7 cm}| }  \hline
					& \multicolumn{3}{|c|}{RIS-RP}\\
					& $\delta=2$   & $\delta=5$ &  $\delta=8$  \\ 
					$n_t $ &MSPE~ECP~ Width & MSPE~ECP~ Width & MSPE~ECP~ Width \\ \hline
					$200$ 	& 0.800  0.39  1.059 & 0.872   0.42  0.983 & 0.855   0.34  0.928 \\
					$250$ & 0.852  0.39  1.102 & 0.920   0.42  1.023 & 0.921   0.35  1.013 \\
					$300$	& 0.860  0.36  1.126 & 0.855   0.44  1.075 &  0.866   0.36  1.069 \\
					$350$	& 0.778  0.45  1.210 & 0.779   0.48  1.221 & 0.829   0.46  1.219 \\\hline
					\hline
					& \multicolumn{3}{|c|}{RIS-PCR}\\
					& $\delta=2$   & $\delta=5$ &  $\delta=8$  \\ 
					$n_t $ &MSPE~ECP~ Width & MSPE~ECP~ Width & MSPE~ECP~ Width  \\\hline
					200& 0.834   0.06   0.177 & 0.838   0.12   0.192 & 0.831   0.10  0.252 \\
					250 &  0.858   0.14   0.355 & 0.882   0.12   0.289 & 0.896   0.19  0.420 \\
					300 & 0.845   0.14   0.399 & 0.867   0.20   0.511 & 0.865   0.20  0.487 \\
					350 &  0.757   0.35   0.893 & 0.786   0.36   0.886 & 0.826   0.41  0.984\\ 
					\hline
				\end{tabular}
		\end{center}}
\end{table*}}

{\it Results:} The MSPEs of RIS-RP and RIS-PCR are comparable for all the choices on $n$. However, RIS-RP yields much better empirical coverage probabilities than RIS-PCR, especially when $n\leq 300$. The three choices of $\delta$ yield comparable results in terms of all the measures in general. For RIS-RP, $\delta=5$ results in higher ECP and for RIS-PCR higher ECP is obtained using $\delta=8$. Moreover, the choice $\delta=8$ makes both the procedures much faster compared to other choices of $\delta$. When the training sample is $350$, $\delta=2,5$ and $8$ select about $290800, 12600$ and $7960$ variables, respectively, on an average in the screening stage out of $8.2\times 10^6$ variables. In view of the results in this massive dimensional dataset, it seems reasonable to use a higher value of $\delta$ for filtering out noisy regressors, and computational convenience.

\section{Mathematical Details}\label{sm:5}
\paragraph{Proof of Lemma \ref{lm:2}}
\begin{proof}[Proof of part a.]
	Consider the conditional expectation and variance of $\|\Rg {\bf x}\|^2$ given $(\bgm,{\bf x})$ as follows: 	
	\quad $E\left( \|\Rg {\bf x}\|^2|\bgm \right)= m_n\|{\bf x}_{\gamma}\|^2$
	\vspace{-.2 in}
	\begin{eqnarray*}
		var\left( \|\Rg {\bf x}\|^2|\bgm \right) &=&    m_n \|{\bf x}_{\gamma}\|^4 \left[ 1+\left\{ (2\psi)^{-1}-2\right\} \sum_{j=1}^{\pg} x_{\gamma,j}^4 / \|{\bf x_{\gamma}}\|^4  \right],
	\end{eqnarray*}
	where ${\bf x}_{\bgm}$ includes the regressors $j$ for which $\gamma_j=1$. The details is given in the proof of Result \ref{res1} below.
	Next consider the conditional expectation of $\|\Rg {\bf x}\|^2$ given ${\bf x}$ is given by
	\begin{eqnarray}
	E_{\bgm}   E\left( \|\Rg {\bf x}\|^2|\bgm \right)=m_n E_{\bgm} \left(\sum_j x_j^2 I(\gamma_j=1) \right)=c ~m_n \sum_j x_j^2 |r_{{\bf x}_j,\yn}|^{\delta}, \label{eq_rp4}
	\end{eqnarray}
	where $c>0$ is the proportionality constant. 
	Also the conditional variance of $\|\Rg {\bf x}\|^2$ given ${\bf x}$ is given by
	$	var_{\bgm} \left\{ E\left( \|\Rg {\bf x}\|^2|\bgm \right) \right\} + E_{\bgm} \left\{  var\left( \|\Rg {\bf x}\|^2|\bgm \right)\right\}.$
	Considering both the terms of the above expression separately we get
	\begin{eqnarray}
	var_{\bgm} \left\{ E\left( \|\Rg {\bf x}\|^2|\bgm \right) \right\} &=& var_{\bgm} \left( m_n \sum_j x_j^2 I(\gamma_j=1)\right) \notag \\
	&=& c~ m_n^2 \sum_j x_j^4 |r_{{\bf x}_j,\yn}|^{\delta} \left( 1 - c|r_{{\bf x}_j,\yn}|^{\delta} \right) \leq c ~ m_n^2 p_n, \label{eq_rp6}
	\end{eqnarray}
	as given ${\bf x}$, $\gamma_j$s are independent, each $|x_j|\leq 1$, and $q_j= c|r_j|^{\delta} <1$. Again
	\begin{eqnarray}
	E_{\bgm} \left\{  var\left( \|\Rg {\bf x}\|^2|\bgm \right)\right\} &= E_{\bgm} \left[m_n \|{\bf x}_{\gamma}\|^4 \left\{1+\displaystyle\left(\frac{1}{2\psi}-2 \right) \frac{\sum_{j=1}^{\pg} x_{\gamma,j}^4}{ \|{\bf x_{\gamma}}\|^4} \right\} \right] \notag \\
	&\leq   c~m_n E_{\bgm} \left[  \|{\bf x}_{\gamma}\|^4    \right] \leq c~m_n E_{\bgm} \left[  \|{\bf x}\|^4    \right] \leq c ~m_n ~p_n^2 
	\label{eq_rp7}
	\end{eqnarray}
	for some constant $c$, as $\sum_{j=1}^{\pg} x_{\gamma,j}^4 < \|{\bf x_{\gamma}}\|^4$.

	Therefore, from (\ref{eq_rp4}), (\ref{eq_rp6}) and (\ref{eq_rp7}) it can be shown that the expectation of $ \|\Rg {\bf x}\|^2 /(m_n p_n)$ converges to $c\ad$, and variance of the same converges to $0$, as $p_n\rightarrow\infty$ and $m_n \rightarrow\infty$.
	
	\vskip5pt
	\noindent	{\it Proof of part b.}
	Observing that $E_{\bgm} \left( \|{\bf x}_{\bgm}\|^2 \right)=c   \sum_j x_j^2 |r_{{\bf x}_j,\yn}|^{\delta}$ and 
	$$ var_{\bgm} \left( \|{\bf x}_{\bgm}\|^2  \right) = c \sum_j x_j^4 |r_{{\bf x}_j,\yn}|^{\delta} \left( 1 - c|r_{{\bf x}_j,\yn}|^{\delta} \right) \leq p_n.$$
	Therefore it can be shown that the expectation of $ \| {\bf x}_{\bgm}\|^2 /p_n$ converges to the limit $c\ad$, and variance of the same converges to $0$.
\end{proof}

\paragraph{Proof of Lemma \ref{lm:4}}
\begin{proof}[Proof of the statement of Lemma \ref{lm:2} (a).]
	Recall Result \ref{res1}. Under assumption (A1) we have 
	
	{\centering
		$ \displaystyle\frac{1}{m_n p_n}E_{\bgm} E\left(\left\| \Rg {\bf x}\right\| \right) \rightarrow \alpha_\delta,$
		\par}
	\vskip5pt
	
	\noindent given ${\bf x}$ for $\alpha_\delta$ as in (A1). 
	To see that the variance $var\left(\| \Rg {\bf x} \|\right) =o(m_n^2 p_n^2)$, observe that 
	\begin{equation}
	var_{\bgm} \left\{ E\left( \|\Rg {\bf x}\|^2|\bgm \right) \right\} 
	= c~ m_n^2 \sum_j x_j^4 |r_{{\bf x}_j,\yn}|^{\delta} \left( 1 - c|r_{{\bf x}_j,\yn}|^{\delta} \right) = o\left(m_n^2 p_n^2 \right), \label{eqn_e5}
	\end{equation} 
	almost surely.
	To verify the last statement note that by \citet[Theorem 1.8.E]{Serfling}, $\sum_i x_i^4/p_n\rightarrow \sum_j E\left(x_j^4\right) /p_n$ almost surely if $cov\left(x_i^4,x_j^4\right)\leq \rho^{*}_{|i-j|}\sqrt{var\left(x_i^4\right)  var\left(x_j^4\right)}$ with $\sum_j \rho_j^{*}<\infty$, and $\sum_j var\left( x_j^4\right)(\log j)^2/j^2 <\infty$. Here $E\left(x_j^4\right) =3$, $var \left(x_j^4\right)=96$ for all $j$, and $cov\left(x_i^4,x_j^4\right) =24\sigma_{i,j}^2 \left(\sigma_{i,j}^2+3 \right)$. Thus it is easy to see that strong law of large numbers (SLLN) holds for $\sum_j x_j^4/p_n$ by assumption (B1), and therefore (\ref{eqn_e5}) holds almost surely.

	Similarly,~
	$E_{\bgm} \left\{  var\left( \|\Rg {\bf x}\|^2|\bgm \right)\right\}\leq   c~m_n E_{\bgm} \left[  \|{\bf x}_{\gamma}\|^4    \right] \leq c~m_n  \|{\bf x}\|^4   \leq o\left(m_n^2 p_n^2\right) $ 
	almost surely. To prove the last statement, we argue as before that $\|\bfx\|^2/p_n \rightarrow 1$ almost surely. Here $E(x_j^2)=1$, $var(x_j^2)=2$ for all $j$, and $cov(x_i^2,x_j^2)=2\sigma_{i,j}^2$. Therefore by (B1) SLLN holds for $\|\bfx\|^2/p_n$, and therefore $\|\bfx\|^4/p_n^2$ is bounded almost surely. As $m_n \rightarrow\infty$ the above statement holds. Therefore the statement of Lemma \ref{lm:2}(a) holds.
	
	%
	
	\vskip10pt
	\noindent	{\it Proof of the statement of Lemma \ref{lm:2} b.}
	Observe that $E_{\bgm} \left( \|{\bf x}_{\bgm}\|^2 \right)=c   \sum_j x_j^2 |r_{{\bf x}_j,\yn}|^{\delta}$ and 
	$$ var_{\bgm} \left( \|{\bf x}_{\bgm}\|^2  \right) = c \sum_j x_j^4 |r_{{\bf x}_j,\yn}|^{\delta} \left( 1 - c|r_{{\bf x}_j,\yn}|^{\delta} \right) \leq c\sum_j x_j^4 .$$
	Thus the expectation of $ \| {\bf x}_{\bgm}\|^2 /p_n$ converges to the limit $c\ad$, and variance of the same converges to $0$ almost surely as $p_n\rightarrow\infty$. This completes the proof.
\end{proof}

\begin{result}
	\label{res1} 
	Consider a random matrix $\Rg$ which depends on another random vector $\bgm$ distributed as in (\ref{eq_rp2}). Then the conditional distribution of $\Rg$ satisfies the following:
	\begin{enumerate}[a.]
		\item $ E\left( \|\Rg {\bf x}\|^2|\bgm \right)=m_n \|{\bf x}_{\gamma}\|^2$, and
		\item    $var\left( \|\Rg {\bf x}\|^2|\bgm \right) =    m_n \|{\bf x}_{\gamma}\|^4 \left[ 1+\left\{ (2\psi)^{-1}-2\right\} \sum_{j=1}^{\pg} x_{\gamma,j}^4 / \|{\bf x_{\gamma}}\|^4  \right] .$
	\end{enumerate}
	\begin{proof}[Proof of part a.] Observe that 
		\begin{eqnarray}
		\|\Rg {\bf x}\|^2 &=& \left\| \left( \sum_j r_{1,j} \gamma_j x_j, \sum_j r_{2,j} \gamma_j x_j, \ldots , \sum_j r_{m_n,j} \gamma_j x_j \right)^{\prime} \right\|^2 \notag \\
		&=& \left( \sum_j r_{1,j} \gamma_j x_j \right)^2 + \left( \sum_j r_{2,j} \gamma_j x_j \right)^2 + \ldots + \left( \sum_j r_{m_n,j} \gamma_j x_j \right)^2. \label{eq_rp12}
		\end{eqnarray}
		Now $ E\left( \sum_j r_{i,j} \gamma_j x_j \right)^2 =E\left\{ \sum_j r_{i,j}^2 \gamma_j x_j^2 + \sum_{j\neq j^{\prime}} r_{i,j} r_{i,j^{\prime}} \gamma_j \gamma_{j^{\prime}} x_j x_{j^{\prime}} \right\}=  \sum_j \gamma_j x_j^2 = \|{\bf x}_{\gamma}\|^2, $
		as $E(r_{i,j}^2)=1$ and $E(r_{i,j} r_{i,j^{\prime}})=0$ as $i=1,2,\ldots, m_n$, $j,j^{\prime}=1,2, \ldots, p_n$, and $j \neq j^{\prime}$. 

		\noindent {\it Proof of part b.} From (\ref{eq_rp12}) we have
		\begin{eqnarray}
		var\left( \|\Rg {\bf x}\|^2|\bgm \right) &=& var \left\{ \sum_i \left( \sum_j r_{i,j} \gamma_j x_j \right)^2   \right\}  = \sum_i var \left( \sum_j r_{i,j} \gamma_j x_j \right)^2  \notag \\
		&& \hskip20pt  + \sum_{i\neq i^{\prime}} cov \left\{ \left( \sum_j r_{i,j} \gamma_j x_j \right)^2, \left( \sum_j r_{i^{\prime},j} \gamma_j x_j \right)^2\right\}.  \label{eq_rp13}  
		\end{eqnarray}
		We will consider each term of (\ref{eq_rp13}) one by one. Consider the first term. Note that 
		\begin{eqnarray}
		var \left( \sum_j r_{i,j} \gamma_j x_j \right)^2 &=& var \left\{ \sum_j r_{i,j}^2 \gamma_j x_j^2 + \sum_{j\neq k} r_{i,j} r_{i,j^{\prime}} \gamma_j \gamma_k x_j x_{j^{\prime}} \right\}  \notag \\
		&=&  var \left\{ \sum_j r_{i,j}^2 \gamma_j x_j^2 \right\} + var \left\{ \sum_{j\neq j^{\prime}} r_{i,j} r_{i,j^{\prime}} \gamma_j \gamma_k x_j x_{j^{\prime}} \right\} \notag \\
		&& \hskip25pt + cov \left\{ \sum_j r_{i,j}^2 \gamma_j x_j^2 ,  \sum_{j\neq j^{\prime}} r_{i,j} r_{i,j^{\prime}} \gamma_j \gamma_{j^{\prime}} x_j x_{j^{\prime}} \right\}. \notag
		\end{eqnarray}
		Consider the first term in (\ref{eq_rp13}).
		\begin{eqnarray*}
			var \left\{ \sum_j r_{i,j}^2 \gamma_j x_j^2 \right\} &=& \sum_j var \left( r_{i,j}^2 \gamma_j x_j^2 \right) + \sum_{j \neq j^{\prime} } cov \left( r_{i,j}^2 \gamma_j x_j^2, r_{i,j^{\prime}}^2 \gamma_{j^{\prime}} x_{j^{\prime}}^2 \right)\\
			&=& \sum_j \gamma_j x_j^4 var \left( r_{i,j}^2  \right) + \sum_{j \neq j^{\prime} } \gamma_j x_j^2 \gamma_{j^{\prime}} x_{j^{\prime}}^2 cov \left( r_{i,j}^2 , r_{i,j^{\prime}}^2  \right)\\
			&=& \sum_j \gamma_j x_j^4  \left\{ E\left( r_{i,j}^4 \right) - E^{2}  \left( r_{i,j}^2 \right)   \right\} 
			= \left( \frac{1}{2\psi}-1  \right) \sum_j \gamma_j x_j^4 ,
		\end{eqnarray*}
		as $E\left(r_{i,j}^4\right)=(2\psi)^{-1}$.
		Again,
		\begin{eqnarray*}
			var \left\{ \sum_{j\neq j^{\prime}} r_{i,j} r_{i,j^{\prime}} \gamma_j \gamma_k x_j x_{j^{\prime}} \right\}   = E\left(\sum_{j\neq j^{\prime}} r_{i,j} r_{i,j^{\prime}} \gamma_j \gamma_k x_j x_{j^{\prime}} \right)^2 
			= \sum_{j\neq j^{\prime}}  \gamma_j \gamma_k x^2_j x^2_{j^{\prime}} E \left( r_{i,j}^2 r_{i,j^{\prime}}^2 \right) \\
			+ \sum_{\substack{\text{$(j,j^{\prime})\neq (k,k^{\prime})$} \\ \text{$j\neq j^{\prime}, k \neq k^{\prime}$}}}  \gamma_j \gamma_k \gamma_{j^{\prime}} \gamma_{k^{\prime}} x^2_j x^2_{j^{\prime}} x^2_k x^2_{k^{\prime}} E \left( r_{i,j} r_{i,j^{\prime}} r_{i,k} r_{i,k^{\prime}} \right) = \sum_{j\neq j^{\prime}}  \gamma_j \gamma_k x^2_j x^2_{j^{\prime}}
		\end{eqnarray*}
		as the other term will be zero. Next
		\begin{eqnarray*}
			cov \left\{ \sum_j r_{i,j}^2 \gamma_j x_j^2 ,  \sum_{j\neq j^{\prime}} r_{i,j} r_{i,j^{\prime}} \gamma_j \gamma_{j^{\prime}} x_j x_{j^{\prime}} \right\} = \sum_j \sum_{k\neq k^{\prime}}  \gamma_j x_j^2 ,    \gamma_k \gamma_{k^{\prime}} x_k x_{k^{\prime}} cov\left(r_{i,j}^2 ,r_{i,k} r_{i,k^{\prime}}\right)=0.
		\end{eqnarray*}
		Therefore the first term in (\ref{eq_rp13}) is
		\begin{eqnarray}
		\sum_i var \left( \sum_j r_{i,j} \gamma_j x_j \right)^2 
		=  \left( \frac{1}{2\psi} -2 \right) \sum_j \gamma_j x_j^4 +  \left(\sum_{j}  \gamma_j  x^2_j \right)^2. \label{eq_rp15}
		\end{eqnarray}
		
		The last term in (\ref{eq_rp13}),
		$cov \left\{ \left( \sum_j r_{i,j} \gamma_j x_j \right)^2, \left( \sum_j r_{i^{\prime},j} \gamma_j x_j \right)^2\right\}=0$.
		This is because the $\left( \sum_j r_{i,j} \gamma_j x_j \right)^2$ depends on the $i^{th}$ row of the random matrix $R$ for a fixed $i$, and  $\left( \sum_j r_{i^{\prime},j} \gamma_j x_j \right)^2$ depends on a fixed $i^{\prime}\neq i$. Therefore these two terms are independent, hence uncorrelated. 
		Combining the above result and (\ref{eq_rp15}) the proof follows.
	\end{proof}

\end{result}

\begin{small}
\bibliographystyle{apalike}
 \bibliography{rp}
\end{small}
 \end{document}